\documentclass[12pt]{article}

\usepackage{latexsym,mathrsfs,bm}
\usepackage[english]{babel}
\usepackage{paralist}
\usepackage{ulem}
\usepackage{amsthm,amsfonts,amssymb,amsmath}
\usepackage[latin1]{inputenc}
\usepackage{bm}

\newcommand{\eps}{\varepsilon}
\newcommand{\R}{\mathbb{R}}
\newcommand{\Q}{\mathbb{Q}}
\newcommand{\C}{\mathbb{C}}
\newcommand{\N}{\mathbb{N}}

\newcommand{\Z}{\mathbb{Z}}

\newcommand{\es}[1]{\begin{equation}\begin{split}#1\end{split}\end{equation}}
\newcommand{\est}[1]{\begin{equation*}\begin{split}#1\end{split}\end{equation*}}

\newcommand{\B}{\mathrm{B}}

\renewcommand{\P}{\mathbb{P}}

\newcommand{\tn}[1]{\textnormal{#1}}

\renewcommand{\mod}[1]{~\pr{\textnormal{mod}~#1}}

\newtheorem*{theo*}{Theorem}
\newtheorem{theo}{Theorem}

\newtheorem{lemma}{Lemma}
\newtheorem{corol}[lemma]{Corollary}
\newtheorem{remark}{Remark}

\newtheorem*{rem*}{Remark}
\def\sumstar{\operatornamewithlimits{\sum\nolimits^*}}
\def\sumprime{\operatornamewithlimits{\sum\nolimits^\prime}}

\def\sumdagger{\operatornamewithlimits{\sum\nolimits^\dagger}}
\newcommand{\pr}[1]{\left( #1\right)}

\newcommand{\e}[1]{\operatorname{e}\pr{ #1}}
\newcommand{\cc}{\operatorname{c}}
\newcommand{\Res}{\operatorname*{Res}}
\newcommand{\sign}{\operatorname*{sign}}

\newcommand{\balpha}{\boldsymbol\alpha}
\newcommand{\bbeta}{\boldsymbol\beta}
\newcommand{\bgamma}{\boldsymbol\gamma}
\newcommand{\bdelta}{\boldsymbol\delta}
\newcommand{\bv}{\boldsymbol{v}}
\newcommand{\bu}{\boldsymbol{u}}
\newcommand{\bx}{\boldsymbol{x}}
\newcommand{\by}{\boldsymbol{y}}
\newcommand{\bt}{\boldsymbol{t}}
\newcommand{\bz}{\boldsymbol{z}}
\newcommand{\ba}{\boldsymbol{a}}
\newcommand{\bxi}{\boldsymbol{\xi}}
\newcommand{\bnu}{\boldsymbol{\nu}}

\newcommand{\bzero}{\boldsymbol{0}}
\newcommand{\bw}{\boldsymbol{w}}

\let\originalleft\left
\let\originalright\right
\renewcommand{\left}{\mathopen{}\mathclose\bgroup\originalleft}
\renewcommand{\right}{\aftergroup\egroup\originalright}

\numberwithin{equation}{section}
\makeatletter
\newcommand{\subjclass}[2][2010]{%
  \let\@oldtitle\@title%
  \gdef\@title{\@oldtitle\footnotetext{#1 \emph{Mathematics subject classification.} #2}}%
}
\newcommand{\keywords}[1]{%
  \let\@@oldtitle\@title%
  \gdef\@title{\@@oldtitle\footnotetext{\emph{Key words and phrases.} #1.}}%
}
\makeatother
\newcommand{\addresses}{{
  \bigskip
  \footnotesize

  S.~Bettin, \textsc{Dipartimento di Matematica, Universit\`a di Genova; via Dodecaneso 35; 16146 Genova; Italy. 
}\par\nopagebreak
  \textit{E-mail address}: \texttt{bettin@dima.unige.it}

 }}

\begin{document}

\author{Sandro Bettin}
\title{Linear correlations of the divisor function}
\date{}

\maketitle

\begin{abstract}
Motivated by arithmetic applications on the number of points in a bihomogeneous variety and on moments of Dirichlet $L$-functions, we provide analytic continuation for the series $\mathcal A_{\ba}(s):=\sum_{n_1,\dots,n_k\geq1}\frac{d(n_1)\cdots d(n_k)}{(n_1\cdots n_k)^{s}}$ with the sum restricted to solutions of a non-trivial linear equation $a_1n_1+\cdots+a_kn_k=0$. The series  $\mathcal A_{\ba}(s)$ converges absolutely for $\Re(s)>1-\frac1k$ and we show it can be meromorphically continued to $\Re(s)>1-\frac 2{k+1}$ with poles at $s=1-\frac1{k-j}$ only, for $1\leq j< (k-1)/2$.

As an application, we obtain an asymptotic formula with power saving error term for the number of points in the variety  $a_1x_1y_1+\cdots+a_kx_ky_k=0$ in $\mathbb P^{k-1}(\Q)\times \mathbb P^{k-1}(\Q)$.
\end{abstract}

\section{Introduction}

Motivated by some applications which we shall describe below,  we consider the Dirichlet series $\mathcal A_{\ba}(s)$ obtained by adding a linear constraints among the variables of summation when expanding $(\zeta(s)^2)^k$ into a product of $k$ Dirichlet series. More precisely, for $k\in\N_{\geq2}$, $\ba=(a_1\dots,a_k)\in\Z_{\neq0}^k$, and $\Re(s)>1-\frac1k$ we define $\mathcal A_{\ba}(s)$ to be
\es{\label{sds}
\mathcal A_{\ba}(s):=\sum_{\substack{n_1,\dots,n_k\geq1,\\
a_1n_1+\cdots+a_kn_k=0
}}\frac{d(n_1)\cdots d(n_k)}{(n_1\cdots n_k)^{s}}=\sum_{n\geq1}\frac{h_{\ba}(n)}{n^s},
}
where $d(n)$ is the number of divisors of $n$ and where $h_{\ba}(n)$ is defined implicitly by the second identity. Notice we can assume that $a_1,\dots,a_k$ don't all have the same sign, since otherwise $\mathcal A_{\ba}(s)=0$.
The function $\mathcal A_{\ba}(s)$ can be regarded as a degree $2$ analogue of 
\est{
\mathcal S_{\ba}(s):=\sum_{\substack{n_1,\dots,n_k\geq1,\\
a_1n_1+\cdots+a_kn_k=0
}}\frac{1}{(n_1\cdots n_k)^{s}}.
}
This function is a particular case of the Shintani zeta-function, which was investigated in a series of works by Shintani (see, e.g.~\cite{Shi76,Shi78}). In particular, he showed that $\mathcal S_{\ba}(s)$ admits a meromorphic continuation to $\C$ and studied its special values displaying a connection with the values at $s=1$ of Hecke $L$-function of totally real fields.

The value at $s=1$ of the function $\mathcal A_{\ba}(s)$ also has an arithmetic interpretation. Indeed, in~\cite{Bet15} it  was shown that  $\mathcal A_{\ba}(1)$ appears as the leading constant for the moments of a ``cotangent sum'' related to the Nyman-Beurling criterion for the Riemann hypothesis. More specifically, it was shown that as $q\to\infty$
\est{
\sum_{\substack{h=1,\\(h,q)=1}}^{q-1}\cc_0(h/q)^{2k}\sim \frac{q^{2k}}{\pi^{4k}}\sum_{\epsilon\in \{\pm1\}^k}(-1)^{k-\#\{1\leq i\leq 2k\mid \epsilon_i=1\}}\mathcal A_{\epsilon}(1)
}
where $\cc_0(h/q):=-\sum_{m=1}^q\frac{m}{q}\cot(\frac{\pi m h}q)$. We defer to~\cite{Bag,BC13a,BC13b} for more details on $\cc_0$ and on its relation with the Nyman-Beurling criterion. Also, we remark that the asymptotic for the moments of $\cc_0(h/q)$ was previously computed in~\cite{MR16a} with a different expression for the leading constant.

In this paper we are interested in the analytic continuation of $\mathcal A_{\ba}(s)$.
For $k=2$ it is very easy to analytically continue $\mathcal A_{\ba}(s)$ to a meromorphic function on $\C$. Indeed, for $a_1,a_2\in\N$ one has has
\est{
\mathcal A_{\ba}(s)=\sum_{\substack{n_1,n_2\geq1,\\ a_1n_1=a_2n_2}}\frac{d(n_1)d(n_2)}{(n_1n_2)^s}=\sum_{\substack{n\geq1}}\frac{d(a_1n)d(a_2n)}{(abn^2)^s}=\eta_{a_1,a_2}(s)\sum_{\substack{n\geq1}}\frac{d(n)^2}{n^{2s}}=\eta_{a_1,a_2}(s)\frac{\zeta(2s)^4}{\zeta(4s)}
}
by Ramanujan's identity, where $\eta_{a_1,a_2}(s)$ is a certain arithmetic factor which is meromorphic in $\C$ with poles all located  on the line $\Re(s)=0$. 
In the case $k\geq3$ the coefficients $h_{\ba}(n)$ in~\eqref{sds} are no longer multiplicative and the problem of providing meromorphic continuation for $\mathcal A_{\ba}(s)$ becomes significantly harder, but we are still able to enlarge the domain of holomorphicity of  $\mathcal A_{\ba}(s)$ to  $\Re(s)>1-\frac{2}{k+1}$. 

\begin{theo}\label{mtmc}
Let $k\geq3$, $\ba=(a_1\dots,a_k)\in\Z_{\neq0}^k$ with $a_1,\dots,a_k$ not all with the same sign. Then, $\mathcal A_{\ba}(s)$ admits meromorphic continuation to $\Re(s)>1-\frac{2}{k+1}$. 
More precisely, there exist $c_{m,j}(\ba)\in\R$ for $ \frac{k+2}2\leq m\leq k$, $1\leq j\leq m+1$ such that 
\est{
\mathcal F_{\ba}(s):=\mathcal A_{\ba}(s)-\sum_{\frac{k+2}2\leq m\leq k}\bigg(\frac{c_{m,m+1}(\ba)}{(s-(1-\tfrac1m))^{1+m}}+\cdots+\frac{c_{m,1}(\ba)}{s-(1-\tfrac1m)}\bigg)
}
is holomorphic on  $\Re(s)>1-\frac{2}{k+1}$. Moreover, there exists an absolute constant $A>0$ such that  for $\Re(s)>1-\frac{2-\eps}{k+1}$ one has $\mathcal F_{\ba}(s)\ll_{\eps} \pr{ (\frac k\eps \max_{m=1}^k|a_m|)^A  (1+|s|)^7}^{Ak^2(1-1/k-\sigma)+k\eps}$, where the implicit constant depends on $\eps$ only.  
\end{theo}

Notice that Theorem~\ref{mtmc} is uniform in $k,\ba$ and $s$. We remark that the uniformity in $k$ of some bounds for $\mathcal A_{\ba}(s)$ at $s=1$ was crucial in the works~\cite{Bet15} and~\cite{MR16b} and, in general, it is also needed for our application~\cite{Bet}.

The value of the arithmetic factors $c_{m,j}(\ba)$ can be computed explicitly starting from equation~\eqref{dacc1} below. In particular, for $m=k$, $j=k+1$ one has 
\es{\label{mcost}
c_{k,k+1}(\ba)
= \frac{\rho(\ba)}{|a_1\cdots a_k|^{\frac1k}} \frac{k! }{k^{k+1}} \frac{\zeta(k-1)}{\zeta(k)}
\frac{ \Gamma(\tfrac1k)^k}{\Gamma(\tfrac rk)\Gamma(1-\tfrac rk)},
}
where $r$ is the number of $a_i$ which are positive and $\rho(\ba)$ is as defined in~\eqref{arco}; in particular if $\tn{GCD}(a_1,\dots,a_k)=1$ then $1\ll \rho(\ba)\ll_\eps |a_1\cdots a_k|^\eps$. Also, for $1\leq j\leq k$ one has that $c_{k,j}(\ba)$ can be written in terms of an arithmetic factor of shape similar to $\rho(\ba)$ times an expression depending on Euler's constant $\gamma$, the derivatives of $\zeta(s-1)$ and $\zeta(s)$ computed at $k$, and the derivatives of the $\Gamma$-function computed at $\frac1k$, $\frac{r}k$ and $1-\frac rk$ (cf. equation~\eqref{dacc2}). To give an explicit numerical example for $k=3$, $\ba=(-1,1,1)$, we have
\est{
c_{3,4}\approx 0.537228,\quad c_{3,3}\approx 0.554669,\quad c_{3,2}\approx 1.689055, \quad c_{3,1}\approx -64.704169.
}

An interesting question left open by Theorem~\ref{mtmc} is whether $\mathcal A_{\ba}(s)$ extends meromorphically further to the left, perhaps to a meromorphic function on $\C$ with poles at $s=1-\frac1m$ for all $1\leq m\leq k$, or whether it has a natural boundary. As an approach to this problem one could try to input a recursive argument into the proof of Theorem~\ref{mtmc}.
We notice however that the expression for the coefficients $c_{m,j}$ arising in the proof of Theorem~\ref{mtmc} does not visibly extend to a meaningful formula in the case $m\leq \frac12(\sqrt{4k+1}+1)$ for $k>3$ ($m=1$ if $k=3$), thus suggesting these coefficients might change form at some point or perhaps casting doubts on the possibility of a meromorphic continuation of $\mathcal A_{\ba}(s)$ to $\C$. Finally we mention that numerical computations in the case $k=3$ suggest there is a pole also at the subsequent expected location $s=\frac12$ (i.e. a term of order $P(\log X)X^{\frac32}$ in~\eqref{k3} below), however the computations do not clarify whether the corresponding coefficients have the same shape of the previous coefficients or not. 

Our first application of Theorem~\ref{mtmc} is given in the following Corollary.

\begin{corol}\label{mtmcc}
Let $k\geq3$ and $\ba\in\Z_{\neq0}^k$. Let $\Phi(x)$ be a smooth function with support in $[-1,1]$ and such that $\phi^{(j)}(x)\ll j^{Bj}$ for some $B>0$ and all $x\in\R$. Then, for $ \frac{k+1}2<i\leq k$ there exist polynomials $P_{\ba,i}(x)\in\R[x]$ of degree $i$ such that for all $X\geq 1$
\es{\label{mtmcce}
&\sum_{\substack{n_1,\dots,n_k\geq1,\\ a_1n_1+\cdots+a_kn_k=0}}\hspace{-1.5em}d(n_1)\cdots d(n_k)\Phi\pr{\frac{n_1\cdots n_k}{X^k}}\\[-1em]
&\hspace{10em}=\sum_{\frac{k+2}2\leq i\leq k}P_{\ba,i}(\log X)X^{k-\frac{k}{i}}+O_{\eps,\Phi}\big(\big( \tfrac {k^B}\eps  \max_{i=1}^k|a_i|\big)^{Ak} X^{k-2+\frac{ 2+\eps}{k+1}} \big),
}
for some absolute constant $A>0$.
\end{corol}
To give two examples, in the cases $k=3$ and $k=4$ (with $a_1=-1$) Corollary~\ref{mtmcc} gives 
\begin{align}
&\begin{aligned}
&\sum_{n_1,n_2\geq1}\hspace{-0.2em}d(n_1)d(n_2) d(an_1+bn_2)\Phi\pr{{n_1n_2(an_1+ bn_2)}/{X^3}}\\[-0.8em]
&\hspace{18.5em}=Q_{3,a,b}(\log X)X^{2}+O_{\eps,\Phi}\big(\big(a+b\big)^{A} X^{\frac 3{2}+\eps} \big),
\end{aligned}\label{k3}\\
&\begin{aligned}
&\hspace{-0.5em}\sum_{n_1,n_2,n_3\geq1}\hspace{-0.2em}d(n_1)d(n_2)d(n_3) d(an_1+bn_2+cn_3)\Phi\pr{{n_1n_2(an_1+ bn_k+cn_3)}/{X^3}}\\[-0.8em]
&\hspace{7.5em}=R_{4,a,b,c}(\log X)X^{3}+R_{3,a,b,c}(\log X)X^{\frac 83}+O_{\eps,\Phi}\big(\big(a+b+c\big)^{A} X^{\frac {12}{5}+\eps} \big).
\end{aligned}\label{k4}
\end{align}
for any $a,b,c\in\N$ and where $Q_{3,a,b}(x),R_{3,a,b,c}(x)$ and $R_{4,a,b,c}(x)$ are polynomials of degree $3,3$ and $4$ respectively and $A$ is an absolute constant.

We remark that one could use an easier argument to give the leading term in the asymptotic for the left hand side of~\eqref{mtmcce}. In fact the main difficulty of Corollary~\ref{mtmcc} lies in unravelling the complicated combinatorics required to obtain the full main term $P_{\ba,i}(\log X)X^{k-1}$. This difficulties are implicitly treated in Theorem~\ref{mtmc}, which allows us to go even further than the full main term. Indeed, for $k\geq4$ we are able to identify also some new terms whose order is a power smaller than the main term (cf.~\eqref{k4}). This is an example of an arithmetic stratification, where one has other ``main terms'', coming from sub-varieties, of order  (typically) different from the main term one expects from the variety under consideration. This phenomenon was discussed by Manin and Tschinkel~\cite{MT} and explored in the context of the circle method by Vaughan and Wooley in the Appendix of~\cite{VW}. Recently, the arithmetic stratification was also indicated by Wooley as a potential source for the various terms in the Conrey-Keating analyis~\cite{CK} of the  asymptotic for moments of the Riemann zeta-function.  In our case, the lower order contribution could be explained as  coming from affine sub-varieties, that is solutions of $n_1m_1+\cdots+n_km_k=0$ which also satisfy one or more other equations $r_1n_1m_1+\dots+r_kn_km_k=r_0$ for some ``small'' $r_0,\dots,r_k\in\Z$. 

A result similar to~\eqref{k3}, with the significant difference in the different way of counting, was obtained by Browning~\cite{Bro}. He computed the asymptotic with power-saving error term for 
\es{\label{eqdes}
\sum_{n_1,n_2< X}d(L(n_1,n_2))d(L_2(n_1,n_2))d(L_3(n_1,n_2)),
}
where $L_1,L_2,L_3\in\Z[x_1,x_2]$ are linearly independent linear forms. He also considered~\eqref{eqdes} when the sums are unbalanced, i.e.\mbox{} where the sum is restricted to $n_1\leq N_1$, $n_2\leq N_2$ with  $N_2$ smaller than $N_1$. He was able to prove the asymptotic as long as $N_1^{3/4+\eps}\leq N_2\leq N_1$  (for $L_1(n_1,n_2)=n_1-n_2$, $L_1(n_1,n_2)=n_1$, $L_3(n_1,n_2)=n_1+n_2$), a range that was recently enlarged by Blomer~\cite{Blo} who was able to consider the case $N_1^{1/3+\eps}\leq N_2\leq N_1$ (with a smooth cut-off for $n_2$). In a different direction, we also cite the work of Browning and de la Bret\`eche~\cite{dlBB}, who considered the case $k=3$ with a quadratic relation among the variables.

For larger values of $k$, we cite the important work of Matthiesen~\cite{Mat} who considered a variation of~\eqref{mtmcce} as well as the more general case when one has more than one linear constraint. Her work differs from ours in that in her case the variables vary inside a convex set, whereas in our case the variables are essentially summed over the hyperbolic region. Also, her method is based on the Green-Tao transference principle~\cite{GT} which can only give the leading term $cX^{k-1}(\log X)^k$ in the asymptotic formula.  
In particular we notice that neither the work of Matthiesen nor those of Browning and Blomer were able to produce terms of order a lower power of $X$.

Before introducing our second application, we first mention that we shall actually prove a more general version of Theorem~\ref{mtmc} where shifts are introduced, i.e. where instead of each divisor function $d(n)$ we have $\tau_{\alpha_i,\beta_i}(n):=\sum_{ab=n}a^{-\alpha_i}b^{-\beta_i}$ with $\alpha_i,\beta_i\in\C$. We defer to Theorem~\ref{mtmt} in Section~\ref{shifts} for the complete statement. The shifts make our result extremely  flexible. In particular one can use it to count integer solutions $(x_1,\dots,x_k,y_1,\dots,y_k)$ in the flag variety 
\es{\label{darc}
a_1x_1y_1+\cdots+a_kx_ky_k=0
}
when ordered according to various possible choices of height. To give a specific example we take the anticanonical height $(\max_{i}|x_i|\cdot \max_j |y_j|)^{k-1}$, verifying Manin's conjecture in this particular case.

\begin{theo}\label{sect}
Let $k\in\N_{\geq3}$. For $\bx\in\mathbb P^{k-1}(\Q)$, let $(x_1,\dots,x_k)$ be a representative of $\bx$ such that $x_1,\dots,x_k\in\Z$  and $(x_1,\dots,x_k)=1$. Let $H:\mathbb P^{k-1}(\Q)\to\R_{>0}$ be defined by $H(\bx):=(\max_{1\leq i\leq k}|x_i|)^{{k-1}}$ and let $\ba=(a_1,\dots,a_k)\in\Z_{\neq0}^{k}$. Then,
\est{
N(B)&:=\#\bigg\{(\bx,\by)\in \mathbb P^{k-1}(\Q)\times \mathbb P^{k-1}(\Q)\bigg|
\begin{aligned}
&a_1x_1y_1+\cdots+a_kx_ky_k=0,\, H(\bx)H(\by)<B\\
&x_1\cdots x_k\cdot y_1\cdots y_k\neq0
\end{aligned}
\bigg\}\\
&=\Big(\frac{\mathfrak S(\ba)}{k-1} \sum_{i=1}^k\sigma_i(\ba)\Big) B\log B+f(\ba)B+O_{k,\eps,\ba}(B^{1-\frac{k-2}{2 k (15 k-2)}+\eps}),
}
for some explicitly computable $f(\ba)\in\R$, and
\es{\label{fsda}
\mathfrak S(\ba)&:=\frac12\sum_{ \ell\geq1}\frac{\prod_{i=1}^{k}(a_i,\ell)}{\ell^{k}}\frac{\varphi(\ell)}{\zeta(k)^2},\\
\sigma_i(\ba)&:=\Big(\prod_{i=2}^k\int_{-1}^1\int_{-1}^1\Big) \chi_{(-|a_i|,|a_i|)}\Big(\sum_{j=1,\, j\neq i}^ka_ix_iy_i\Big)\prod_{j=1,\, j\neq i}^kdx_idy_i,
}
where $\chi_{X}$ is the characteristic function of the set $X$.
\end{theo}
We remark that we made no effort to optimize the power saving $\delta_k$ which could be easily improved by refining our method (in particular focusing more on the shift dependency in Theorem~\ref{mtmt}).

The variety~\eqref{darc} has already been considered in several papers. In particular, we mention the works of Robbiani~\cite{Rob} (for $k\geq3$), Spencer~\cite{Spe}, Browning~\cite{Bro} (for $k=2$), and Blomer and Br\"udern~\cite{BB1} (for more general multihomogeneous diagonal equations) and, previously, by Franke, Manin and Tschinzel~\cite{FMT} and Thunder~\cite{Thu} (with height function $\|x\|^{{k-1}} \|y\|^{{k-1}}$)\footnote{Here $\|\cdot\|$ is the Euclidean norm and $\bx=(x_1,\dots,x_k)$, $\by=(y_1,\dots,y_k)$} in the more general setting of Fano varieties. Among all these works, the only ones where the full main term, with error term $O(B^{1-\delta})$,  is obtained are~\cite{FMT} (see the Corollary after Theorem 5) and~\cite{BB1} (in~\cite{BB2} the explicit value $\delta=\frac18$ was obtained for $k=3$). Theorem~\ref{sect} thus gives an alternative proof of this result as well as providing an explicit power saving for all $k$. Also, another novelty in our approach is that it shows that also complex analytic methods can be used to tackle these problems.

We notice that Theorem~\ref{sect} appears very similar to Corollary~\ref{mtmcc}, which essentially counts points in~\eqref{darc} ordering them according to the size of the product of all the variables, $|x_1\cdots x_k\cdot y_1\cdots y_k|$.\footnote{One could easily modify our argument to count points with respect to $\max_i|x_iy_i|$, since our proof starts by introducing partitions of unity which localize each $x_iy_i$.} The different way of counting however changes the problem significantly and the deduction of Theorem~\ref{sect} from (the generalization of) Theorem~\ref{mtmc} is much subtler. In particular, the computation of the full main term for $N(B)$ requires a careful analysis of some complex integrals resulting from integrating over the shifts in Theorem~\ref{mtmt}. Notice that also in this case the problem becomes much easier if one only computes the leading term in the asymptotic for $N(B)$.

A third application of Theorem~\ref{mtmc} comes from the theory of the moments of $L$-functions. In~\cite{Bet} it is considered the moment
\est{
M_k(q):=\sumprime_{\chi_1,\dots,\chi_{k-1}\!\mod q}|L(\tfrac12,\chi_1)|^2\cdots |L(\tfrac12,\chi_{k-1})|^2|L(\tfrac12,\chi_1\cdots\chi_{k-1})|^2,
}
where $\sumprime$ indicates that the sum is over primitive characters $\chi_1,\dots,\chi_{k-1}$ modulo $q$ and $L(s,\chi)$ is the Dirichlet $L$-function associated to the character $\chi$. It turns out that the ``diagonal term'' in $M_k(q)$ has the shape $\sum_{\epsilon \in \{\pm1\}^k}c_\epsilon\int_{(2)}\mathcal A_{\epsilon}(s)q^{ks}H(s)\,{ds}$ for a meromorphic function $H(s)$ and some $c_\epsilon\in\R$. Thanks to~Theorem~\ref{mtmc} we are able to evaluate the diagonal term and thus, evaluating also the off-diagonal term using a similar method, we are able to obtain the following asymptotic formula for $M_k(q)$ when $k\geq3$  (the case $k=2$ corresponds to the $4$-th moment of Dirichlet $L$-function and was computed by Young~\cite{You})
\es{\label{asfm}
M_{k}(q)=\varphi(q)^{k-1}\sum_{n=1}^{\infty}\frac{2^{\nu(n)}}{n^{\frac k2}} \pr{(\log\tfrac q{8n\pi })^{k}+(-\pi)^{k}}+O_\eps\pr{k^{Ak}q^{-\delta_k+\eps}},
}
where $\varphi(n)$ is Euler's $\phi$ function,  $\nu(n)$ is the number of different prime factors of $n$ and $\delta_k>0$. We also mention that, thanks to the work~\cite{Bet16},~\eqref{asfm} can be interpreted also as the moment of some functions involving continued fractions.

The proof of Theorem~\ref{mtmc} is quite simple in spirit but it has to face a number of technical challenges, mainly coming from the identification of the polar structure (equivalently, of the main terms in~\eqref{mtmcce}).  Before giving a brief outline of our proof, we mention that one could have chosen to proceed also in different ways, for example using the circle method. The main difficulty however comes from the evaluation of the polar structure and this is not visibly simplified by choosing such different routes. We also remark that the our technique would allow to give analytic continuation also when the constraint is a non-homogeneous linear equation. 
The only difference with our case is that in Lemma~\ref{mlb} below we would need to use the Deshouillers and Iwaniec~\cite{DI} bound for sums of Kloosterman sums  (cf.~\cite{Bet} where this is done for a similar problem). However, for simplicity we content ourself with dealing with the homogenous case only. 

We conclude with a rough sketch of the proof of Theorem~\ref{mtmc} referring for simplicity to Corollary~\ref{mtmcc} which is essentially equivalent to it.
First, we split the sum on the right hand side of~\eqref{mtmcce} introducing partitions of unity to control the size of the $n_i$. When one variable is much larger than the others a simple bound suffices, so we are left with considering the case when the variables have about the same size. In this case we eliminate the larger variable using the linear equation and we separate the remaining variables arithmetically and analytically using, respectively, a slightly modified version of Ramanujan's formula,
\es{\label{rffm}
\sigma_{1-s}(m)=\zeta(s)\sum_{\ell\geq 1}\frac{c_{\ell}(m)}{\ell^s}\qquad\Re(s)>1,
}
where $c_\ell(m)$ is the Ramanujan sum and $\sigma_{\alpha}(n):=\sum_{d|n}d^{\alpha}$ and a generalized version of the Mellin formula for $(1\pm x)^{-s}$ as given in~\cite{Bet}. We end up with a formula of the shape (for $\ba=(1,\dots,1)$)
\est{
\sum_{\ell\geq1}\sum_{\substack{h\mod \ell,\\(h,\ell)=1}}\sum_{n_1,\dots,n_{k-1}\geq1}\sum_{n_1,\dots,n_{k-1}}d(n_1)\cdots d(n_{k-1})\e{\tfrac{hn_1+\cdots+hn_k}{\ell}}f(n_1,\dots,n_{k-1},\ell)
}
for some smooth function $f$. Applying Voronoi's summation formula to each variable $n_1,\dots,n_k$ transform each sum over $n_i$ in a main term $M_{i}$ plus a sum of similar shape but with $h$ replaced by $\overline h$ and thus we obtain
\est{
\sum_{\ell\geq1}\sum_{\substack{h\mod \ell,\\(h,\ell)=1}}\prod_{i=1}^{k-1}\Big(M_i+\sum_{n_i\geq1}d(n_i)\e{\tfrac{\overline h n_i}{\ell}}f_i(\cdots)
\Big)}
for some smooth functions $f_i$. We then treat as main terms the terms where we pick up more $M_i$ than series, and treat the other terms as error terms which we estimate essentially trivially. We then treat and assembly the main terms (which correspond to the poles of $\mathcal A_{\ba}(s)$), an operation which constitutes the main difficulty of the paper as we have to deal with several integral transforms in order to take them to their final form (actually, we choose the equivalent root of moving the lines of integration of several complex integrals, collecting the contribution of the residues of some poles). Combining the two cases for the range of the variables one then deduces Corollary~\ref{mtmcc}.

We notice that the above structure of the proof of Theorem~\ref{mtmc} is at first glance very similar to that of the asymptotic for $M_k(q)$ performed in~\cite{Bet}. There are however several important differences at a more detailed level, e.g. in the ways the integrals are manipulated, in the treatment of the error terms and in the combinatorics.

The paper is organized as follow. In Section~\ref{shifts} we state Theorem~\ref{mtmt} which gives the analytic continuation for the shifted version of $\mathcal A_{\ba}(s)$ and in Section~\ref{ptmtmc} we easily deduce Theorem~\ref{mtmc} from it and we compute the constants given in~\eqref{mcost}. In Section~\ref{sectp} we prove Theorem~\ref{sect} by integrating over the shifts introduced in Theorem~\ref{mtmt} and evaluating the resulting complex integrals. The rest of the paper is dedicated to the proof of Theorem~\ref{mtmt}: in Section~\ref{rab} we give a uniform bound for the region of absolute convergence, whereas in Section~\ref{mtmp} we set up the proof of Theorem~\ref{mtmt} dividing the sum according to the range of the variables. In Section~\ref{pmmlle} we estimate the case where the variables have roughly the same size and in Section~\ref{pmmlle2} we give a trivial bound for the case where there's a large variable. Finally, in Section~\ref{pcompl} we recompose the various sums reconstructing the polar terms.

\subsubsection*{Acknowledgments}
The author wishes to thank Trevor Wooley and Tim Browning for useful comments. 

The work of the author is partially supported by FRA 2015 ``Teoria dei Numeri" and by PRIN ``Number Theory and Arithmetic Geometry". 

\section{The shifted case}\label{shifts}

For $k\geq3$, $\ba$ as above and $\balpha=(\alpha_1,\dots,\alpha_k)$, $\bbeta=(\beta_1,\dots,\beta_k)\in\C^k$ we define the Dirichlet series
\est{
\mathcal A_{\ba;\balpha,\bbeta}(s):=\sum_{a_1n_1+\cdots+a_kn_k=0}\frac{\tau_{\alpha_1,\beta_1}(n_1)\cdots \tau_{\alpha_k,\beta_k}(n_k)}{n_1^{\frac12+s}\cdots n_{k}^{\frac12+s}},
}
where $\tau_{\alpha,\beta}(n):=\sum_{d_1d_2=n} d_1^{-\alpha}d_2^{-\beta}$. If $|\Re(\alpha_m)|,|\Re(\beta_m)|\leq \frac1{2(k-1)}$ for all $m$, then it is easy to see (cf. Lemma~\ref{absc} below) that $\mathcal A_{\ba;\balpha,\bbeta}(s)$ converges absolutely on 
$$\Re(s)>1-\frac1k-\frac1k\sum_{m=1}^k\min(\Re(\alpha_m),\Re(\beta_m)).$$
 The following Theorem gives the analytic continuation for $\mathcal A_{\ba;\balpha,\bbeta}(s)$ to a larger half-plane, provided that
\es{\label{xi}
\eta_{\balpha,\bbeta}:=\tfrac32\sum_{m=1}^k( |\Re(\alpha_m)|+|\Re(\beta_m)|)
}
is not too large. Before stating the theorem we need to introduce (a slight variation of) the Estermann function, which for $\alpha,\beta\in\C$, $h\in\Z$ and $\ell \in\N$ is defined as
\est{
D_{\alpha,\beta}(s,\tfrac {h}{\ell }):=\sum_{n\geq1}\frac{\tau_{\alpha,\beta}(n)}{n^s}\e{n\tfrac h\ell }
}
for $\Re(s)>1-\min(\Re(\alpha),\Re(\beta))$ and where $\e{x}:=e^{2\pi ix}$. The Estermann function can be continued to a meromorphic function on $\C$ satisfying a functional equation (see e.g.~\cite{BC13b}).
\begin{lemma}\label{lest}
Let $\alpha,\beta\in\C$, $h\in\Z$ and $\ell \in\N$ with $(h,\ell )=1$. Then 
\est{
D_{\alpha,\beta}(s,\tfrac {h}{\ell })-\ell ^{1-\alpha-\beta-2s}\zeta(s+\alpha)\zeta(s+\beta)
}
can be extended to an entire function of $s$. Moreover, one has
\es{\label{ffeff}
D_{\alpha,\beta}(s,\tfrac h\ell)=\ell^{1-2s-\alpha-\beta}(\chi_{+1}(s;\alpha,\beta)D_{-\alpha,-\beta}(s,\tfrac{\overline h}\ell)-\chi_{-1}(s;\alpha,\beta)
D_{-\alpha,-\beta}(s,-\tfrac{\overline h}\ell)),
}
where $\overline h$ denotes the inverse of $h\mod \ell$ and 
\est{
\chi_{\pm1}(s;\alpha,\beta):=2(2\pi)^{2s-2+\alpha+\beta}\Gamma(1-s-\alpha)\Gamma(1-s-\beta)\cos\Big(\pi \frac{(s+\alpha)\mp(s+\beta)}{2}\Big).
}
\end{lemma}
For $\balpha,\bbeta\in\{z\in\C\mid|\Re(z)|<\frac1{2(k-1)}\}^k$ we also define
\es{\label{dfm}
&\mathcal M_{\ba;\balpha,\bbeta}(s):=\sum_{\substack{\mathcal I\cup \mathcal J=\{1,\dots,k\},\\ \mathcal I\cap \mathcal J=\emptyset,\ |\mathcal I|\geq \frac{k+2}{2}}}\sum_{\substack{\{\alpha_i^*,\beta_i^*\}=\{\alpha_i,\beta_i\}\\ \forall i\in \mathcal I}}\Big(\prod_{i\in \mathcal I}\frac{\zeta(1-\alpha_i^*+\beta_i^*)}{|a_i|^{-\alpha_i^*+\frac{1+s_{\mathcal I,\balpha^*}}{|\mathcal I|}}}\Gamma(-\alpha_i^*+\tfrac{1+s_{\mathcal I,\balpha^*}}{|\mathcal I|})\Big)\\
&\quad\times \frac{\Delta_{\balpha^*;\mathcal I}}{| \mathcal I|(s-1)+s_{\mathcal I,\balpha^*}+1 } \sum_{ \ell\geq1}\frac{\prod_{i\in \mathcal I}(a_i,\ell)^{1-\alpha_i^*+\beta_i^*}}{\ell^{\sum_{i\in\mathcal  I}(1-\alpha_i^*+\beta_i^*)}}\sumstar_{h\mod \ell}\prod_{j\in \mathcal J} D_{\alpha_j,\beta_j}(1-\tfrac{1+s_{\mathcal I,\balpha^*}}{|\mathcal I|},\tfrac {ha_j}{\ell}),
}
where $\sumstar$ indicates the the sum is over $h$ which are coprime to $\ell$, and the second sum is over $\balpha^*=(\alpha_i^*)_{i\in\mathcal I}$, $\bbeta^*=(\beta^*)_{i\in\mathcal I}\in\C^{|\mathcal I|}$ satisfying the above condition. Also, we put $s_{\mathcal I,\balpha^*}:=\sum_{r\in\mathcal I}\alpha_r^*$ and
\est{
\Delta_{\balpha^*;\mathcal I}:=\bigg(\Gamma\Big(\sum_{\substack{i\in \mathcal I,\, \sign(a_i)=1}}\Big(-\alpha_i^*+\frac{1+s_{\mathcal I,\balpha^*}}{|\mathcal I|}\Big)\Big)\Gamma\Big(\sum_{\substack{i\in \mathcal I,\,  \sign(a_i)=-1}}\Big(-\alpha_i^*+\frac{1+s_{\mathcal I,\balpha^*}}{|\mathcal I|}\Big)\Big)\bigg)^{-1}
}
if neither of the two sums inside the $\Gamma$ functions are empty sums and $\Delta_{\balpha^*;\mathcal I}:=0$ otherwise.
\begin{remark}
Equation~\eqref{dfm} should be interpreted as defining $\mathcal M_{\ba;\balpha,\bbeta}(s)$ as a meromorphic function. Also, the definition of $\mathcal M_{\ba;\balpha,\bbeta}(s)$  can be extended to include the case where $\alpha_i=\beta_i$ since the limit for $\alpha_i\to\beta_i$ exists (cf. the proof of Theorem~\ref{mtmc}).

The absolute convergence of the series over $\ell$ in~\eqref{dfm} for $\balpha,\bbeta\in\{z\in\C\mid|\Re(z)|<\frac1{2(k-1)}\}^k$ is ensured by the convexity bound for the Estermann function,
\es{\label{convbo}
D_{\alpha,\beta}(s,\tfrac h\ell)\ll \delta^{-2}(\ell(1+|s|+|\alpha|)^\frac12(1+|s|+|\beta|)^\frac12)^{1-\Re(s)-\min(\Re(\alpha),\Re(\beta))+\delta},
}
valid for $\delta>0$, $\Re(\alpha),\Re(\beta)\ll1$, $|1-s-\alpha|,|1-s-\beta|>\delta$, and 
\es{\label{rconvbo}
 -\max (\Re(\alpha),\Re(\beta))-\delta\leq \Re(s)\leq 1-\min(\Re(\alpha),\Re(\beta))+\delta.
 }
Indeed, using~\eqref{convbo} one has that the series over $\ell$ converges as long as $|\mathcal I|>2+\frac{|\mathcal J|}{|\mathcal I|}+\sum_{i\in\mathcal I}(\frac{k}{|\mathcal I|}\Re(\alpha_i^*)-\Re(\beta_i^*))-\sum_{j\in\mathcal J}\min(\Re(\alpha_j),\Re(\beta_j))$. The right hand side is less than $4$, so the only problematic case is when $|\mathcal I|=3$. This can happen only for $k=3$ and $k=4$, and in the first case the convergence is clear since there is no Estermann function. Finally, the series converges also for $k=4$, $|\mathcal I|=3$ and $|\mathcal J|=1$ since one can save an extra factor of $\ell^{1-\eps}$ using the convexity bound for $\sum^*_hD_{\alpha,\beta}(s,\tfrac {h}\ell)$
\est{
\sumstar_{h\mod \ell} D_{\alpha,\beta}(s,\tfrac h\ell)\ll_{s,\alpha,\beta,\delta,\eps} \ell^{1-\Re(s)-\min(\Re(\alpha),\Re(\beta))+\delta+\eps},
}
for $s$ satisfying~\eqref{rconvbo} and where we used the bound $c_{\ell}(m)\ll d(\ell)(m,\ell)$ for the Ramanujan sum $c_{\ell}(m):=\sum_{\ell|a}\e{\frac {hm}\ell}$.
\end{remark}

We are now ready to state our main theorem.

\begin{theo}\label{mtmt}
Let $k\geq3$, $\ba\in\Z_{\neq0}^k$.  Then, $\mathcal A_{\ba;\balpha,\bbeta}(s)$ admits meromorphic continuation to $\Re(s)>1-\frac{2-2\eta_{\balpha,\bbeta}}{k+1}$, $\balpha,\bbeta\in\{s\in\C\mid|\Re(s)|<\frac1{2(k-1)}\}^k$. Moreover, if for some $\eps>0$ one has $\balpha,\bbeta\in\{s\in\C\mid|\Re(s)|<\tfrac{1-\eps}{2(k-1)}\}^k$ and $1-\tfrac{2-2\eta_{\balpha,\bbeta}-\eps}{k+1}\leq \Re(s)\leq 1-\tfrac{1-\eps}{k}+\tfrac{\eta_{\balpha,\bbeta}}{k+1}$, then
\es{\label{bdmt}
&\mathcal A_{\ba;\balpha,\bbeta}(s)-\mathcal M_{\ba;\balpha,\bbeta}(s)\\
&\qquad\ll_{\eps}\! \pr{ \pr{\tfrac k\eps \max_{i=1}^k|a_i|}^A \big((1+|s|)(1+\max_{i=1}^{k}(|\Im(\alpha_i)|+|\Im(\beta_i)|)\big)^7}^{Ak\frac{(k+1)(1-\frac1k-\sigma)+\eta_{\balpha,\bbeta}+\eps}{1-\frac1k-\eta_{\balpha,\bbeta}}}
}
for some absolute constant $A>0$. \end{theo}
\begin{remark}
Clearly Theorem~\ref{mtmt} provides analytic continuation also for
\es{\label{dfnas}
\mathcal A^*_{\ba;\balpha,\bbeta}(s):=\sum_{\substack{n_1,\dots,n_k\in\Z_{\neq0},\\ a_1n_1+\cdots+a_kn_k=0}}\frac{\tau_{\alpha_1,\beta_1}(|n_1|)\cdots \tau_{\alpha_k,\beta_k}(|n_k|)}{|n_1|^{\frac12+s}\cdots |n_{k}|^{\frac12+s}},
}
since $\mathcal A^*_{\ba;\balpha,\bbeta}(s)=\sum_{\epsilon\in \{\pm1\}^k}\mathcal A_{\ba_\epsilon;\balpha,\bbeta}(s)$, where $\ba_\epsilon=(\pm_1a_1,\dots,\pm_k a)$. Moreover, the sum over $\epsilon$ in the ``polar term'' $\mathcal M^*_{\ba;\balpha,\bbeta}(s):=\sum_{\epsilon\in \{\pm1\}^k}\mathcal M^*_{\ba_\epsilon;\balpha,\bbeta}(s)$ for $\mathcal A^*_{\ba;\balpha,\bbeta}(s)$ can be executed giving a neater expression for $\mathcal M^*_{\ba_\epsilon;\balpha,\bbeta}(s)$. Indeed, for $\sum_{i\in\mathcal I}z_i=1$ we have
\es{\label{regac}
&\sum_{\mathcal H\subseteq\mathcal I}\frac{\prod_{i\in\mathcal I}\Gamma(z_i)}{\Gamma(\sum_{i\in\mathcal H}z_i)\Gamma(\sum_{i\in\mathcal I\setminus \mathcal H}z_i)}=\Big(\prod_{i\in\mathcal I}\Gamma(z_i)\Big)\sum_{\mathcal H\subseteq\mathcal I}\frac{\sin(\pi \sum_{i\in\mathcal H} z_i)}\pi\\
&\hspace{8em}=\frac{1}{2\pi i} \Big(\prod_{i\in\mathcal I}\Gamma(z_i)(1+e^{\pi i z_i})-\prod_{i\in\mathcal I}\Gamma(z_i)(1+e^{-\pi i z_i})\Big)\\
&\hspace{8em}=\frac{2^{|\mathcal I|}}{2\pi i} \Big(e^{\frac{\pi i}2\sum_{i\in\mathcal I} z_i}\prod_{i\in\mathcal I}\Gamma(z_i)\cos(\tfrac \pi 2z_i)-e^{-\frac{\pi i}2\sum_{i\in\mathcal I} z_i}\prod_{i\in\mathcal I}\Gamma(z_i)\cos(\tfrac \pi 2z_i)\Big)\\
&\hspace{8em}=2\pi^{\frac{|\mathcal I|}2-1} \prod_{i\in\mathcal I}\frac{\Gamma(\frac {z_i}2)}{\Gamma(\frac {1-z_i}2)}\\
}
by the identities for the $\Gamma$ function $(\Gamma(z)\Gamma(1-z))^{-1}=\frac1\pi \sin(\pi z)$ and $\cos\pr{\frac{\pi z}2}\Gamma(z)=\frac{\pi^{1/2}2^{z-1}\Gamma(\frac z2)}{\Gamma(\frac {1-z}2)}$. Thus, since $\sum_{i\in\mathcal I}(-\alpha_i^*+\frac{1+s_{\mathcal I,\balpha^*}}{|\mathcal I|})=1$, it follows that
\es{\label{dfmas}
&\mathcal M^*_{\ba;\balpha,\bbeta}(s)=\sum_{\substack{\mathcal I\cup \mathcal J=\{1,\dots,k\},\\ \mathcal I\cap \mathcal J=\emptyset,\ |\mathcal I|\geq \frac{k+2}{2}}}\sum_{\substack{\{\alpha_i^*,\beta_i^*\}=\{\alpha_i,\beta_i\}\\ \forall i\in \mathcal I}}\Big(\prod_{i\in \mathcal I}\frac{\zeta(1-\alpha_i^*+\beta_i^*)}{|a_i|^{-\alpha_i^*+\frac{1+s_{\mathcal I,\balpha^*}}{|\mathcal I|}}}\frac{\Gamma(-\frac{\alpha_i^*}2+\frac{1+s_{\mathcal I,\balpha^*}}{2|\mathcal I|})}{\Gamma(\frac{1+\alpha_i^*}2-\frac{1+s_{\mathcal I,\balpha^*}}{2|\mathcal I|})}\Big)\\
&\hspace{0.3em}\times \frac{2^{|\mathcal J|+1}\pi^{\frac{|\mathcal I|}2-1}}{| \mathcal I|(s-1)+s_{\mathcal I,\balpha^*}+1 } \sum_{ \ell\geq1}\frac{\prod_{i\in \mathcal I}(a_i,\ell)^{1-\alpha_i^*+\beta_i^*}}{\ell^{\sum_{i\in\mathcal  I}(1-\alpha_i^*+\beta_i^*)}}\hspace{-0.39em}\sumstar_{h\mod \ell}\prod_{j\in \mathcal J} D_{\cos;\,\alpha_j,\beta_j}(1-\tfrac{1+s_{\mathcal I,\balpha^*}}{|\mathcal I|},\tfrac {ha_j}{\ell}),
}
where $D_{\cos;\,\alpha,\beta}(s,\frac h\ell):=\frac12(D_{\alpha,\beta}(s,\frac h\ell)+D_{\alpha,\beta}(s,-\frac h\ell))$. \end{remark}

\section{The proof of Theorem~\ref{mtmc}}\label{ptmtmc}
\begin{remark}
Throughout the rest of the paper, by a bold symbol $\bv$ we indicate the vector $(v_1,\dots,v_k)\in\C^k$. Also, for any set $I\subseteq \{1,\dots,k\}$ by $\bv_{ I}$ we indicate the vector $(v_i)_{i\in I}\in\C^{| I|}$.

For any $c\in\R$, by $\int_{(c)}\cdot\,ds$ we indicate that the integral is taken along the vertical line from $c-i\infty$ to $c+i\infty$. Also, we will often abbreviate $\int_{(c_{1})}\cdots\int_{(c_{r})}$ with $\int_{(c_{1},\dots,c_{r})}$.

Finally, by $\eps$ we indicate a sufficiently small positive real number, and by $A$ a positive absolute constant whose value might be different at each occurrence.

\end{remark}
We now show how to derive Theorem~\ref{mtmc} from Theorem~\ref{mtmt}.
\begin{proof}[Proof of Theorem~\ref{mtmc}] 
First, we assume $|\alpha_i|,|\beta_i|<\frac\eps {2k}$ and $\Re(s)>1$. Then, by the residue theorem
\est{
&\mathcal M_{\ba;\balpha,\bbeta}(s)=\sum_{\substack{\mathcal I\cup \mathcal J=\{1,\dots,k\},\\ \mathcal I\cap \mathcal J=\emptyset,\\ |\mathcal I|\geq \frac{k+2}{2}}}\prod_{i\in \mathcal I}\Big(\frac1{2\pi i}\int_{|w_i|=\frac \eps k }\frac{\zeta(1+\alpha_i+w_i)\zeta(1+\beta_i+w_i)}{|a_i|^{w_i+\frac{1+s_{\mathcal I,\bw'_\mathcal I}}{|\mathcal I|}}}\Gamma(w_i+\tfrac{1+s_{\mathcal I,\bw'_\mathcal I}}{|\mathcal I|}) \Big)\\
&\times \frac{\Delta_{\bw_{\mathcal I};\mathcal I}}{|\mathcal I|(s-1)+s_{\mathcal I,\bw'_\mathcal I}+1} \sum_{ \ell\geq1}\frac{\prod_{i\in \mathcal I}(a_i,\ell)^{1+2w_i+\alpha_i+\beta_i}}{\ell^{\sum_{i\in\mathcal  I}(1+2w_i+\alpha_i+\beta_i)}}\hspace{-0.5em}\sumstar_{h\mod \ell}\prod_{j\in \mathcal J} D_{\alpha_j,\beta_j}(1-\tfrac{1+s_{\mathcal I,\bw'_\mathcal I}}{|\mathcal I|},\tfrac {ha_j}{\ell})\prod_{i\in\mathcal I}dw_i,
}
where $\bw'_{\mathcal I}:=(-w_i')_{i\in\mathcal I}$, $s_{\mathcal I,\bw'_\mathcal I}=-\sum_{i\in\mathcal I}w_i$, and the circles are oriented in the positive direction. Thus, letting $\balpha,\bbeta\to\bzero$ we obtain
\es{\label{fias}
&\mathcal M_{\ba;\bzero,\bzero}(s)=
\sum_{\substack{\mathcal I\cup \mathcal J=\{1,\dots,k\},\\ \mathcal I\cap \mathcal J=\emptyset,\ |\mathcal I|\geq \frac{k+2}{2}}} \Big( \prod_{i\in \mathcal I}\frac1{2\pi i}\int_{|w_i|=\frac\eps k}\frac{\zeta(1+w_i)^2(a_i,\ell)^{2w_i}}{|a_i|^{w_i+\frac{1+s_{\mathcal I,\bw'_\mathcal I}}{|\mathcal I|}}} \Gamma(w_i+\tfrac{1+s_{\mathcal I,\bw'_\mathcal I}}{|\mathcal I|})\Big) \\
&\hspace{0.4em}\times \sum_{ \ell\geq1}\frac{\prod_{i\in \mathcal I}(a_i,\ell)}{\ell^{|\mathcal I|+\sum_{i\in\mathcal I}2w_i}}\sumstar_{h\mod \ell}\Big(\prod_{j\in \mathcal J} D_{0,0}(1-\tfrac{1+s_{\mathcal I,\bw'_\mathcal I}}{|\mathcal I|},\tfrac {ha_j}{\ell})\Big)  \frac{\Delta_{\bw_{\mathcal I};\mathcal I}}{|\mathcal I|(s-1)+s_{\mathcal I,\bw'_\mathcal I}+1}\prod_{i\in \mathcal I} dw_i.
}
Now, for $w\neq0$ we have
\est{
\frac1{w-w_1-\cdots-w_k}=\sum_{m=0}^{k}\frac{m!}{w^{m+1}}\sum_{\substack{S\subseteq\{1,\dots,k\},\\ |S|=m}}\prod_{i\in S}w_i+O_w\Big(\sum_{i=1}^k|w_i|^2\Big)
}
as $w_1,\dots,w_k\to 0$ and  so
\est{
\frac{1}{|\mathcal I|(s-1)+s_{\mathcal I,\bw'_\mathcal I}+1 }=\sum_{m=0}^{|\mathcal I|}\frac{m!}{(|\mathcal I|(s-1)+1)^{m+1}}\sum_{S\subseteq\mathcal I,|S|=m}\prod_{i\in S}w_i+O_s\Big(\sum_{i=1}^m|w_i|^2\Big).
}
Also, if $|w_i|<\frac\eps k$ for all $i\in\mathcal I$ with $\eps$ small enough, then we have
\est{
&\Big(\prod_{i\in\mathcal I}\Gamma(w_i+\tfrac{1+s_{\mathcal I,\bw'_\mathcal I}}{|\mathcal I|})\Big)\Delta_{\bw_{\mathcal I};\mathcal I}\sumstar_{h\mod \ell}\prod_{j\in \mathcal J} D_{0,0}(1-\tfrac{1+s_{\mathcal I,\bw'_\mathcal I}}{|\mathcal I|},\tfrac {ha_j}{\ell})\\
&\hspace{22em}=\sum_{S\subseteq\mathcal I}f_{\mathcal I, S}(\ell)\prod_{i\in S}w_i+O_s\Big(\sum_{i=1}^m|w_i|^2\Big)
}
for some $f_{\mathcal I, S}\in\C$. It follows that
\es{\label{dacc1}
\mathcal M_{\ba;\balpha,\bbeta}(s)&=\sum_{m=0}^{|\mathcal I|}\sum_{\substack{S_1,S_2\subseteq\mathcal I,\\  |S_1|=m,\, S_1\cap S_2=\emptyset }} 
\sum_{\substack{\mathcal I\cup \mathcal J=\{1,\dots,k\},\\ \mathcal I\cap \mathcal J=\emptyset,\ |\mathcal I|\geq \frac{k+2}{2}}} \sum_{ \ell\geq1}\frac{\prod_{i\in \mathcal I}(a_i,\ell)}{\ell^{|\mathcal I|}}  \frac{m!f_{\mathcal I, S_2}(\ell) }{(|\mathcal I|(s-1))+1)^{m+1}} \\
& \quad \times  \prod_{i\in \mathcal I}\Res_{w_i=0}\bigg(\frac{\zeta(1+w_i)^2w_{i}^{\delta_{i\in S_1\cup S_2}}}{|a_i|^{w_i+\frac{1}{|\mathcal I|}}(\ell/(a_i,\ell))^{2w_i}}\prod_{r\in\mathcal I}|a_r|^{\frac{w_i}{\mathcal I}} \bigg),
}
where $\delta_{i\in S_1\cup S_2}=1$ if $i\in S_1\cup S_2$ and $\delta_{i\in S_1\cup S_2}=0$ otherwise. By analytic continuation this gives an expression for $\mathcal M_{\ba;\bzero,\bzero}(s)$ for all $s\in\C$ and so Theorem~\ref{mtmc} follows by Theorem~\ref{mtmt}. 
\end{proof}

We conclude the section by computing the value of the constant $c_{k,k+1}(\ba)$.

First we observe that the terms in this sum with $\mathcal I=\{1,\dots,k\}$ are 
\es{\label{dacc2}
& \sum_{m=0}^{k}\frac{m! }{(k(s-1)+1)^{m+1}}
\sum_{\substack{S_1,S_2\subseteq\{1,\dots,k\},\\ |S_1|=m,\, S_1\cap S_2=\emptyset }}\sum_{ \ell\geq1}\frac{\prod_{i=1 }^k(a_i,\ell)}{\ell^k}  f_{\{1,\dots,k\}, S_2}(\ell)\\
&\hspace{8em}\times
\Big(\prod_{\substack{i=1}}^k|a_i|^{-\frac1{\mathcal I}} \Big)\prod_{\substack{i=1,\\ i\notin S_1\cup S_2}}^k \Big (2 \gamma- \log(|a_i|\ell^2/(\ell,a_i)^2)  +\frac1k\sum_{r=1}^k\log |a_r|\Big)
}
(we remark that $ f_{\{1,\dots,k\}, S_2}(\ell)/\varphi(\ell)$ does not depend on $\ell$). Among these, the term with $m=k$ is 
\es{\label{fcf}
& \frac{k! }{(k(s-1)+1)^{k+1}}\frac{\Delta_{\bzero;\{1,\dots,k\}}\Gamma(\tfrac1k)^k}{|a_1\cdots a_k|^{\frac1{k}}}\sum_{ \ell\geq1}\frac{\prod_{i=1}^k(a_i,\ell)}{\ell^k}\varphi(\ell) \\
}
since $f_{\{1,\dots,k\}, \emptyset}=\varphi(\ell)\Gamma(\frac1k)^k\Delta_{\bzero;\{1,\dots,k\}}$. Finally, 
\est{
\sum_{ \ell\geq1}\frac{\prod_{i=1}^k(a_i,\ell)}{\ell^k}\varphi(\ell)&=\rho(\ba)\prod_{p}\Big(1+\sum_{m=1}^\infty (1-\tfrac1p)p^{-m(k-1)}\Big)=\rho(\ba)\prod_{p}\Big(1+\frac{ p-1}{ p^{k}-p}\Big)\\
&=\rho(\ba)\frac{\zeta(k-1)}{\zeta(k)},
}
where
\es{\label{arco}
\rho(\ba)
&:=\prod_{p| a_1\cdots a_k}\frac{1+\sum_{m=1}^\infty (1-\frac1p)p^m\prod_{i=1}^kp^{-\max(0,m-\nu_p(a_i))}}{1+( p-1)/( p^{k}-p)}\\
}
with $\nu_p(a)=r$ if $p^r||a$. 
Thus, the expression in~\eqref{fcf} can be rewritten as
\est{
\frac{\rho(\ba)}{|a_1\cdots a_k|^{\frac1k}} \frac{k! }{(k(s-1)+1)^{k+1}} 
\frac{\zeta(k-1) \Gamma(\tfrac1k)^k}{\zeta(k)\Gamma(\tfrac rk)\Gamma(1-\tfrac rk)},
}
where $r=\{1\leq i\leq k|\sign(a_i)=1\}$ (and the above expression has to be interpreted as $0$ if $r\in\{0,k\}$). 

We remark that for all $q\geq1$ one has, as expected, $\rho(qa_1,\dots, qa_k)=q\rho(a_1,\dots, a_k)$. Finally, we observe that if we assume $\tn{GCD}(a_1,\dots,a_k)=1$ and let $\kappa(\ba):=\prod_{p|a_1\cdots a_k}p^{m_p}$, where $m_p$ is the second smallest among $v_{p}(a_1),\dots,v_{p}(a_k)$ (the smallest being $1$ by hypothesis), then we have
\es{\label{ffba}
\frac{\zeta(k)}{\zeta(k-1)}\frac{\varphi(\kappa(\ba))}{\kappa(\ba)}d(\kappa(\ba))\leq \rho(\ba)\leq  d(\kappa(\ba))
}
and so, in particular $1\ll \rho(\ba)\ll_\eps |a_1\cdots a_k|^\eps$.
Indeed,  we have
\est{
1+\sum_{m=1}^{n_p} (1-\tfrac1p)\leq 1+\sum_{m=1}^\infty (1-\tfrac1p)p^m\prod_{i=1}^kp^{-\max(0,m-\nu_p(a_i))}\leq 1+\sum_{m=1}^{n_p}  (1-\tfrac1p)+\sum_{m>n_p} (1-\tfrac1p)p^{n_p-m}
}
and so
\est{
(n_p+1) (1-\tfrac1p)\leq 1+\sum_{m=1}^\infty (1-\tfrac1p)p^m\prod_{i=1}^kp^{-\max(0,m-\nu_p(a_i))}\leq 1+{n_p} (1-\tfrac1p)+\tfrac1{p}\leq n_p+1
}
which gives
\est{
\frac{\prod_{p|\kappa(\ba)}(v_p(\kappa(\ba))+1) (1-\tfrac1p)}{1+( p-1)/( p^{k}-p)}\leq \rho(\ba)\leq  \prod_{p|\kappa(\ba)}(v_p(\kappa(\ba))+1)
}
and so~\eqref{ffba} follows.

\begin{remark}
Also the coefficients $c_{k-1,r}$ with $1\leq r\leq k$ can be expressed in terms of the Gamma and zeta functions. Indeed, by Ramanujan's formula~\eqref{rffm} if $\Re(s)>1$, $\Re(s+w)>1$, then
\est{
\sum_{\ell\geq1} \frac1{\ell^{w}} \sumstar_{h\mod \ell} D_{0,0}(s,\pm h/k)&=\sum_{\ell\geq1} \frac1{\ell^{w}}  \sum_{n\geq 1}\frac{d(n)}{n^s}c_\ell(n)=\frac1{\zeta(w)} \sum_{n\geq1}\frac{d(n)\sigma_{1-w}(n)}{n^s}\\
&=\frac{\zeta(s)^2\zeta(s+w-1)^2}{\zeta(w)\zeta(2s+w-1)},
}
since $\sum_{n\geq1}\frac{\sigma_{a}(n)\sigma_{b}(n)}{n^s}=\frac{\zeta(s)\zeta(s-a)\zeta(s-b)\zeta(s-a-b)}{\zeta(2s-a-b)}$ (cf.~(1.3.3) in~\cite{Tit}) and the same formula holds for $\Re(s+w)>2$, $\Re(2s+w)>2$, $\Re(w)>1$ by analytic continuation. In particular, assuming for simplicity $\ba=(-1,1,\dots,1)$ one computes that the contribution of the terms with $|\mathcal I|=k-1$ to~\eqref{fias} is
\est{
&
(k-1)\bigg( \prod_{\substack{i=1}}^{k-1}\frac1{2\pi i}\int_{|w_i|=\frac\eps k}\zeta(1+w_i)^2 \Gamma(w_i+\tfrac{1-s_{*}}{k-1})\bigg)  \frac{\Gamma\big(1-w_1-\frac{1-s_{*}}{k-1}\big)^{-1}\Gamma\big(w_1+\tfrac{1-s_{*}}{k-1}\big)^{-1}}{(k-1)(s-1)-s_{*}+1}\\
&\hspace{16em}\times \frac{\zeta(1-\frac1{k-1}+\frac{s_{*}}{k-1})^2\zeta(k-\frac k{k-1}+\frac {(2k-1)s_{*}}{k-1})^2}{\zeta(k-1+2s_{*})\zeta(k-\frac{2}{k-1}+\frac {2ks_{*}}{k-1} )}\prod_{\substack{i=1}}^{k-1} dw_i,
}
where $s_{*}=s_*(w_1,\dots,w_{k-1}):=\sum_{\substack{r=1}}^{k-1}w_r$. Proceeding as above one can then compute the coefficients $c_{k-1,r}$ for $1\leq r\leq k$.
\end{remark}

\section{The proof of Theorem~\ref{sect}}\label{sectp}
First, we observe that the contribution to $N(B^k)$ coming from the terms where the maximum $\max\{|x_1|,\dots,|x_k|,|y_1|,\dots |y_k|\}$ is attained at more than one of the $|x_i|$, $|y_j|$ is $O_{\ba,\eps}(B^{k-\frac32+\eps})$. Indeed, the contribution of the terms with $|x_i|,|y_i|\leq |x_1|=|y_2|$ for all $i=1,\dots,k$ is
\est{
\sum_{\substack{x_i,y_i\leq x_1=y_2<\sqrt B,\ \forall i=1,\dots,k\\ a_1x_1y_1+\cdots+a_kx_ky_k=0 }}\hspace{-0.4em}1&=\sum_{\substack{x_i,y_i\leq x_1=y_2<\sqrt B,\ \forall i=1,\dots,k\\ x_1(a_1y_1+a_2x_2)=a_3x_3y_3+\cdots+a_kx_ky_k }}\hspace{-0.5em}1\ll\hspace{-0.2em}\sum_{\substack{r<(|a_1|+|a_2|) B,\\x_i,y_i<\sqrt B,\ \forall i=3,\dots,k\\ r=a_3x_3y_3+\cdots+a_kx_ky_k }}\hspace{-0.4em}d(r)B^{\frac12}\ll_{\ba,\eps} B^{k-\frac32+\eps}.\\
}
and one can bound similarly all the other cases. Thus,
\est{
N(B^k)=\sum_{i=1}^kN_i(B)+O_{\ba,\eps}(B^{k-\frac32+\eps}),
}
where
\est{
&N_i(B)=\frac12\# \Bigg\{ (\bx,\by)\in \mathbb \Z_{\neq0}^{2k} \Bigg|
  \begin{aligned}
  & a_1x_1y_1+\cdots+a_kx_ky_k=0,\ (x_1,\dots,x_k)=(y_1,\dots,y_k)=1  \\ 
  &|x_iy_j|<B\ \forall j=1,\dots,k, |x_{j_1}|,|y_{j_2}|<|x_i|\ \forall j_1,j_2=1,\dots,k,\,j_1\neq i 
  \end{aligned}\Bigg\}\\
  &=2^{k-1}\# \Bigg\{ (\bx,\by)\in \mathbb \N^k\times\Z_{\neq0}^{k} \Bigg|
  \begin{aligned}
  & a_1x_1y_1+\cdots+a_kx_ky_k=0,\ (x_1,\dots,x_k)=(y_1,\dots,y_k)=1  \\ 
  &|x_iy_j|<B\ \forall j=1,\dots,k,\, |y_{j_2}|,x_{j_1}\!<x_i\ \forall j_1,j_2=1,\dots,k,\,j_1\neq i 
  \end{aligned}\Bigg\}.\\
  }
Notice that in the first line we divided by $4$ since $-\bx=\bx$ in $\P^{1}(\Q)$ and we multiplied by $2$ since we assumed the maximum among the $x_i,y_j$ is attained at one of the $x_i$. 

By symmetry it is sufficient to consider the case $i=1$. Also, we can assume $B$ is a half-integer.
Using M\"obius inversion formula we find
\est{
N_{1}(B)=2^{k-1}\sum_{d_1,d_2\geq1}\frac{\mu(d_1)\mu(d_2)}{d_1d_2}\sum_{\substack{x_2\dots,x_{k},|y_1|,\dots, |y_{k}|<x_1,\ x_i\in\N,y_i\in\Z_{\neq0}\\ x_1|y_j|<B/d_1d_2,\, \forall j =1,\dots k,\\  a_1x_1y_1+\cdots+a_{k}x_{k}y_{k}=0}}1.
}
Then, we express the inequalities $x_1|y_j|<B/d_1d_2$ and $x_2\dots,x_{k},|y_1|,\dots |y_{k}|<x_1$ analytically via the following formula (see Theorem G in~\cite{Ing})
\est{
\frac1{2\pi i}\int'_{(\delta)}x^{-z}\,\frac{dz}z=\begin{cases}\chi_{(0,1)}(x)+O_\delta(\frac{x^{\delta}}{T|\log x|}) & \text{if }x\in\R_{>0}\setminus\{1\}\\
1/2+O_\delta(T^{-1}) &\text{if }x=1
\end{cases}
}
where $\chi_{[0,1)}(x)$ is the indicator function of the set $[0,1)$ and $\int'$ indicates that the integral is truncated at $|\Im(z)|\leq T$. We shall choose the parameter $1\leq T\leq B$ at the end of the argument. Bounding as above the error coming from the cases where $x_1/x_i=1$, $x_1/y_j=1$ for some $i=2,\dots,k$ or $j=1,\dots,k$, we obtain
\es{\label{dsnb}
N_{1}(B)&=\sum_{d_1,d_2\geq1}\frac{\mu(d_1)\mu(d_2)}{d_1d_2}\sum_{\substack{x_1,\dots,x_{k}\in\N,y_1,\dots, y_{k}\in\Z_{\neq0},\\  a_1x_1y_1+\cdots+a_{k}x_{k}y_{k}=0}}\Big(\prod_{j=1}^k\frac1{(2\pi i)^2}\int'_{(c_{z_j},c_{w_j})}\frac1{y_j^{z_{j}+w_j} } \Big)\\
&\quad\times\!\Big(\prod_{f=2}^k\frac1{2\pi i}\int''_{(c_{u_f})}\frac1{x_f^{u_f}}\Big) \frac{2^{k-1}(B/d_1d_2)^{\sum_{j=1}^kz_{j}}}{x_1^{\sum_{j=1}^k(z_j-w_j)-\sum_{f=2}^k u_f}}\Big(\prod_{j=1}^k\frac{dz_{j}dw_j}{z_{j}w_j}\Big)\Big(\prod_{f=2}^k\frac{du_{f}}{u_f}\Big)+E
}
where $\int''$ indicates the integral is truncated at $|\Im(u_f)|\leq 2T$ and the lines of integration are $c_{z_j}=1-\frac1k+3\eps $, $c_{w_j}=\eps$ and $c_{u_f}=1-\frac1k+\eps$, for some small $\eps>0$. The error term $E$ is $O_{\eps,\ba,k}(B^{k-1+3k\eps}/T)$. Indeed, for example, in the most delicate case one needs to bound sums of the following form (we take $a_1=\dots=a_k=1$ for simplicity, but the same proof extends to the general case):
\es{\label{poet}
&\sum_{\substack{x_1,\dots,x_{k},y_1,\dots, y_{k}\in\Z_{\neq0},\ x_1\neq y_2\\  x_1y_1+\cdots+x_{k}y_k=0}}\frac{B^{k-1+3k\eps}}{|x_1\cdots x_ky_1\cdots y_k|^{1-\frac1k+\eps}}\frac1{T\log|x_1/y_2|}.\\
} 
Now, $(\log|x/y|)^{-1}\ll \frac{|x|}{|x-y|}$ if $|x-y|<|x|/2$ (and thus $|y|\geq |x|/2$) and $(\log|x/y|)^{-1}\ll 1$ otherwise. Thus, this sum is
\est{
&\ll_{k,\eps} \frac{B^{k-1+3k\eps}}T+\sum_{\substack{x_1,\dots,x_{k},y_1,y_2,\dots y_{k}\in\Z_{\neq0},0\neq|x_1-y_2|< |x_1|/2\\  x_1y_1+x_2y_2+\cdots+x_{k}y_k=0}}\frac{B^{k-1+3k\eps}/T}{|x_2\cdots x_ky_1y_3\cdots y_k|^{1-\frac1k+\eps}}\frac{1}{|x_1-y_2||x_1|^{1-\frac2{k}+\eps}}.
}
Now, a simple computation shows that for $m\neq0$
\est{
&\sum_{m=x_3y_3+\cdots+x_{k}y_k}\frac1{|x_3\cdots x_ky_3\cdots y_k|^{1-\frac1k+\eps}}\ll_{\eps} {m^{-\frac2{k}}}\\
}
and so, writing $r=x_1-y_2$,  $n=y_1+x_2$ and $m=x_3y_3+\cdots+x_{k}y_k$  and bounding easily the case $m=0$, we see that the sum in~\eqref{poet} is bounded by
\est{
&\ll_{k,\eps} \frac{B^{k-1+3k\eps}}T+\frac{B^{k-1+3k\eps}}T \sum_{\substack{x_1,m,r,x_2\in\Z_{\neq0}, n\in\Z\\ x_1n-x_2r=m,\\|r|< |x_1|/2,\,n\neq x_2}}\frac{1}{|x_1|^{1-\frac2{k}+\eps}|n-x_2|^{1-\frac1k+\eps}|x_2|^{1-\frac1k+\eps}|r||m|^{\frac2k}}.
} 
The terms with $|m|>( |x_1n|+|x_2r|)/4$ can be bounded easily by using $|m|^{-\frac2k}\ll |x_1n|^{-\frac2k}$ and disregarding the linear equation
. For the terms with $|m|\leq( |x_1n|+|x_2r|)/4$, we have also $|x_1n|\leq \frac53|x_2r|$  and so, since $r< |x_1|/2$, then $|n|<\frac 56|x_2|$ and so $|n-x_2|^{1-\frac1k+\eps}\gg  |x_2|^{1-\frac1k+\eps}$. Thus, we obtain that the sum in~\eqref{poet} is 
\est{
&\ll_{k,\eps}\frac{B^{k-1+3k\eps}}T+\frac{B^{k-1+3k\eps}}T \sum_{\substack{x_1,m,r,x_2,n\in\Z_{\neq0}\\ x_1n-x_2r=m,\\  m\ll |x_2r|,\, r\ll |x_1|}}\frac{1}{|x_1|^{1-\frac2{k}+\eps} |x_2|^{2-\frac2k+\eps}|r||m|^{\frac2k}}.
}
Then, we write $m=\ell +x_1g$ with $\ell\equiv -x_2r\mod{|x_1|}$, $|g|\ll \frac{|x_2r|}{|x_1|}$, $-\frac{|x_1|}{2}< \ell\leq \frac{|x_1|}{2}$ and $(\ell,g)\neq (0,0)$, $\ell +x_1g\neq x_2r$. Dividing according to whether $g\neq0$ and $g=0$ we obtain that the above sum is
\est{
&\ll \hspace{-0.5em} \sum_{\substack{x_1,r,x_2,g\in\Z_{\neq0}\\  |g|\ll {|x_2r|}/{|x_1|},\\   r\ll |x_1|}}\frac{B^{k-1+3k\eps}/T}{|x_1|^{1-\frac2{k}+\eps} |x_2|^{2-\frac2k+\eps}|r||gx_1|^{\frac2k}}
+ \sum_{\substack{x_1,r,x_2\in\Z_{\neq0}\\ 1\leq |\ell|\leq \frac {|x_1|}2,\,  r\ll |x_1|,\\ x_1|(\ell-x_2r)\neq0}}\frac{B^{k-1+3k\eps}/T}{|x_1|^{1-\frac2{k}+\eps} |x_2|^{2-\frac2k+\eps}|r||\ell|^{\frac2k}}\\
&\ll\hspace{-0.5em} \sum_{\substack{x_1,r,x_2\in\Z_{\neq0}\\  r\ll |x_1|,}}\frac{B^{k-1+3k\eps}/T}{|x_1|^{2-\frac2{k}+\eps} |x_2|^{1+\eps}|r|^{\frac2k}}
+ \sum_{\substack{\ell,r,x_2\in\Z_{\neq0}\\ \ell-x_2r\neq0}}\frac{d(\ell-x_2r)B^{k-1+3k\eps}/T}{|x_2|^{2-\frac2k+\eps}|r|^{1+\frac\eps 2}|\ell|^{1+\frac\eps2}}\ll \frac{B^{k-1+3k\eps}}T,\\
}
as claimed.

Now, we go back to~\eqref{dsnb} and make the change of variables $u_f\to u_f+z_f$ for all $f=2,\dots,k$. Summing the Dirichlet series, we obtain 
\est{
N_{1}(B)&=2^{k-1} \Big(\prod_{j=1}^k\frac1{(2\pi i)^2}\int'_{(c_{w_j},c_{z_j})} \Big) \Big(\prod_{f=2}^k\frac1{2\pi i}\int'''_{(c_{u_f})} \Big)\frac{B^{\sum_{j=1}^kz_{j}} \mathcal A^*_{\ba;\bz +\balpha,\bz+\bw}(0)}{\zeta(1+\sum_{j=1}^kz_{j})^2}\\
&\quad\times\Big(\prod_{f}\frac{du_{f}}{u_f+z_f}\Big)\Big(\prod_{j}\frac{dz_{j}dw_j}{z_{j}w_j}\Big)+O_{k,\eps,\ba}(B^{k-1+\eps}/T),
}
where $\mathcal A^*_{\ba_\epsilon;\bz +\balpha,\bz+\bbeta}(0)$ is as defined in~\eqref{dfnas}, $c_{u_f}=-2\eps$, 
$$\balpha:=(-w_1-\sum_{j=2}^k(w_j+u_j),u_2,\dots,u_k)$$
and $\int'''$ indicates the integral is truncated at $|\Im(z_f+u_f)|\leq 2T$.

Now, we apply Theorem~\ref{mtmt} to $\mathcal A^*_{\ba;\bz +\balpha,\bz+\bbeta}(0)=\mathcal A^*_{\ba;\bz +\balpha-\bxi,\bz+\bbeta-\bxi}(1-\frac1k)$, where $\bxi:=(1-\frac1k,\dots,1-\frac1k)$. We keep as main term only the summand in~\eqref{dfmas} with $\mathcal I=\{1,\dots,k\}$, treating the other summands as error terms. Thus, we write
\est{
\mathcal A^*_{\ba;\bz+\balpha,\bz+\bw}(0)=\mathcal M^{**}_{\ba;\bz+\balpha,\bz+\bw}+\mathcal E^{**}_{\ba;\bz+\balpha,\bz+\bw},
}
where
\est{
\mathcal M^{**}_{\ba;\bz+\balpha,\bz+\bw}&=\sum_{\substack{\{\alpha_i^*,\beta_i^*\}=\{\alpha_i,w_i\}\\ \forall i=1,\dots,k}}\Big(\prod_{i=1}^k\frac{\zeta(1-\alpha_i^*+\beta_i^*)}{|a_i|^{-z_i-\alpha_i^*+\frac{1+\sum_{r=1}^k(z_r+\alpha^*_r)}{k}}}\frac{\Gamma(-\frac{z_i+\alpha_i^*}2+\frac{1+\sum_{r=1}^k(z_r+\alpha^*_r)}{2k})}{\Gamma(\frac{1+z_i+\alpha_i^*}2-\frac{1+\sum_{r=1}^k(z_r+\alpha^*_r)}{2k})}\Big)\\
&\quad\times \frac{2\pi^{\frac{k}2-1}}{1+\sum_{i=1}^k(z_i+\alpha_i^*-1) } \sum_{ \ell\geq1}\frac{\prod_{i=1}^k(a_i,\ell)^{1-\alpha_i^*+\beta_i^*}}{\ell^{\sum_{i=1}^k(1-\alpha_i^*+\beta_i^*)}}\varphi(\ell),
}
and $\mathcal E^{**}_{\ba;\bz+\balpha,\bz+\bbeta}$ is holomorphic on a region containing 
\est{
\eta:=\sum_{i=1}^k{(|\Re(z_i+\alpha_i-1+\tfrac1k)|+|\Re(z_i+w_i-1+\tfrac1k)|)}<\tfrac{1-\eps}{2k},
}
where it satisfies
\est{
\mathcal E^{**}_{\ba;\bz+\balpha,\bz+\bbeta}
\ll_{k,\eps,\ba}\! \pr{ (1+\max_{i=1}^{k}(|\Im(\alpha_i+z_i)|+|\Im(z_i+w_i)|)}^{14k(\eta+\eps)}.
}
Assuming $\eps$ is small enough with respect to $k$, we can bound the contribution coming from $\mathcal E^{**}$ by moving the line of integration $c_{z_k}$ to $c_{z_k}=\frac1{2k}-6k\eps$, obtaining a contribution of $O(B^{k-1+26k\eps}(T^{7}B^{-\frac1{2k}}+T^{-1}))$ form the integrals over the new line of integration and on the horizontal segments.
Thus, we obtain 
\est{
&N_{1}(B)=N^{'}_{1}(B)+N^{''}_{1}(B)+O_{k,\eps,\ba}(B^{k-1+26k\eps}(T^{7k}B^{-\frac1{2k}}+T^{-1})),
}
where
\est{
N'_{1}(B)&:= \Big(\prod_{j=1}^k\frac1{(2\pi i)^2}\int'_{(c_{w_j},c_{z_j})} \Big)\Big(\prod_{f=2}^k\frac1{2\pi i}\int'''_{(c_{u_f})} \Big)
\sum_{\substack{(\alpha_1^*,\beta_1^*)=(\alpha_1,w_1)\\ \{\alpha_i^*,\beta_i^*\}=\{u_i,w_i\}\\ \forall i=2,\dots,k}}Q_{\balpha^*,\bbeta^*}(\bz,\bw)\\
&\frac{\prod_{i=1}^{k}\zeta(1-\alpha_i^*+\beta_i^*)}{1+\sum_{i=1}^{k}(z_i+\alpha_i^*-1)}\frac{B^{\sum_{j=1}^kz_{j}} }{\zeta(1+\sum_{j=1}^kz_{j})^2}\Big(\prod_{f=2}^kdu_{f}\Big)\Big(\prod_{j=1}^k\frac{dz_{j}dw_j}{w_j}\Big)
}
with
\est{
\mathcal Q_{\balpha^*,\bbeta^*}(\bz,\bw)&:=\Big(\prod_{i=1}^k\frac{1}{|a_i|^{-z_i-\alpha_i^*+\frac{1+\sum_{r=1}^k(z_r+\alpha^*_r)}{k}}}\frac{\Gamma(-\frac{z_i+\alpha_i^*}2+\frac{1+\sum_{r=1}^k(z_r+\alpha^*_r)}{2k})}{\Gamma(\frac{1-z_i-\alpha_i^*}2+\frac{1+\sum_{r=1}^k(z_r+\alpha^*_r)}{2k})}\Big)\\
&\quad\times \frac{2^{k}\pi^{\frac{k}2-1}}{z_1\prod _{f=2}^kz_f(z_f+u_f)}\sum_{ \ell\geq1}\frac{\prod_{i=1}^k(a_i,\ell)^{1-\alpha_i^*+\beta_i^*}}{\ell^{\sum_{i=1}^k(1-\alpha_i^*+\beta_i^*)}}\varphi(\ell),
}
and where $N''_{1,\epsilon}(B)$ is defined in the same way, but with the condition $(\alpha_1^*,\beta_1^*)=(\alpha_1,w_1)$ in the sum replaced by $(\alpha_1^*,\beta_1^*)=(w_1,\alpha_1)$.
We remark that if $\eps$ is small enough, $\mathcal Q_{\balpha^*,\bbeta^*}(\bz,\bw)$ is holomorphic on a region containing
\est{
 \sum_{i=1}^k\Re(\beta_i^*-\alpha_i^*)>2-k+\eps,\ \Re(z_f+u_f)>0\ \forall f=2,\dots,k\ ,\\
1-\tfrac3{2k}\leq \Re(z_i+\alpha^*_i)< 1-\tfrac1{k}+2k\eps,\ \Re(z_i)>0\ \forall i=1,\dots,k,
} 
where by Stirling's formula it satisfies
\begin{align}
\mathcal Q_{\balpha^*,\bbeta^*}(\bz,\bw)&\ll_{\ba,k,\eps} \frac{\prod_{i=1}^k|\Im(z_i+\alpha_i^*-\frac{1}{k}\sum_{r=1}^k(z_r+\alpha^*_r))|^{-\frac12-\Re(z_i+\alpha_i^*-\frac{1+\sum_{r=1}^k(z_r+\alpha^*_r)}{k})}}{|z_1|\prod _{i=2}^k|z_i(z_i+u_i)|}
\notag\\
&\ll_{\ba,k,\eps} \frac{\prod_{i=1}^k|\Im(z_i+\alpha_i^*-\frac{1}{k}\sum_{r=1}^k(z_r+\alpha^*_r))|^{\frac12-\Re(z_i+\alpha_i^*)+\eps}}{|z_1|\prod _{i=2}^k|z_i(z_i+u_i)|}.\label{bdfq0}
\end{align}

Now we move $c_{z_1}$ to $c_{z_1}=1-\frac3{2k}+6k\eps$ passing through the pole at $k-1 -\sum_{i=1}^k(z_i+\alpha_i^*)=0$. Notice that doing so $\alpha_1^*-\beta_1^*$ stays constant so we don't cross the pole  of $\zeta(1-\alpha_1^*+\beta^*_1)$. The contribution of the integrals on the new line of integration and on the horizontal segments is trivially $O_{k,\eps,\ba}(B^{k-1+12k\eps}(B^{-\frac1{2k}}+T^{-1}))$, whereas the contribution of the residue is 
\est{
&\frac1{2\pi i}\int'_{(c_{w_1})}   \Big(\prod_{j=2}^k\frac1{(2\pi i)^2}\int'_{(c_{w_j})}\int'_{(c_{z_j})} \Big)\Big(\prod_{f}\frac1{2\pi i}\int'''_{(\delta_{c_{u_f}})} \Big)
\sum_{\substack{(\alpha_1^*,\beta_1^*)=(\alpha_1,\beta_1)\\ \{\alpha_i^*,\beta_i^*\}=\{u_i,w_i\}\\ \forall i=2,\dots,k}}Q_{\balpha^*,\bbeta^*}(\bz',\bw)\\
&\Big(\prod_{i=1}^{k}\zeta(1-\alpha_i^*+\beta_i^*)\Big) \frac{B^{k-1+w_1+\sum_{j=2}^k\beta^*_{j}}}{\zeta(k+w_1+\sum_{j=2}^k\beta^*_{j})^2}\Big(\prod_{f}du_{f}\Big)\Big(\prod_{j=2}^k\frac{dz_{j}dw_j}{w_j}\Big)\frac{dw_1}{w_1}
}
where $\bz'=(k-1 -\alpha_1-\sum_{i=2}^k(z_i+\alpha_i^*),z_2,\dots,z_k)$ and where we used that for $z_1=k-1 -\alpha^*_1-\sum_{i=2}^k(z_i+\alpha_i^*)$ one has 
\est{
\sum_{j=1}^kz_{j}=k-1-\sum_{j=1}^k\alpha^*_{j}=k-1+w_1-\sum_{j=2}^k\alpha^*_{j}+\sum_{j=2}^k(w_j+u_j)=k-1+w_1+\sum_{j=2}^k\beta^*_{j}.
}

Next, we observe that the  terms in the sum over $\balpha^*,\bbeta^*$ for which $(\alpha_h^*,\beta_h^*)=(w_h,u_h)$ for some $h\in\{2,\dots,k\}$ are $O_{k,\eps,\ba}(B^{k-1+12k\eps}(TB^{-\frac1{2k}}+T^{-1}))$; indeed one can move $c_{w_h}$ and $c_{u_h}$ to $c_{u_h}=-\frac1{2k}+(6k-2)\eps$ and $c_{w_h}=\frac1{2k}-(6k-1)\eps$ and then bound trivially obtaining the claimed bound. Notice that doing so we don't pass through any poles, since $\Re(1-\alpha_1^*+\beta_1^*)=1+\Re(2w_1+\sum_{j=2}^k{u_j}+w_j)$ stays constant, whereas $\Re(1-\alpha_h^*+\beta_h^*)=\Re(1-w_h^*+u_h^*)$ stays less than one. 

Thus, we only have to consider the term with $(\alpha_j^*,\beta_j^*)=(u_j,w_j)$ for all $j=2,\dots,k$ and moving the lines of integration as above for all $i=2,\dots,k$ one obtains that it's enough to consider the contribution from the residue at $w_j=0$ for all $j=2,\dots,k$. To summarize, we arrive to
\est{
N'_{1}(B)&=\frac1{2\pi i}\int'_{(c_{w_1})} \Big(\prod_{j=2}^k\frac1{(2\pi i)^2}\int'_{(c_{z_j})}\int'''_{(c_{u_j})} \Big)
\sum_{\substack{(\alpha_1^*,\beta_1^*)=(-w_1-\sum_{j=2}^ku_j,w_1)\\ (\alpha_i^*,\beta_i^*)=(u_i,0)\\ \forall i=2,\dots,k}}Q_{\balpha^*,\bbeta^*}(\bz',\bw')B^{k-1+w_1} \\
&\times \zeta(1+2w_1+\sum_{j=2}^k{u_j}) \frac{\prod_{i=2}^{k}\zeta(1-u_i)}{\zeta(k+w_1)^2}\frac{dw_1}{w_1} \prod_{j=2}^kdu_{j}dz_j+O_{k,\eps,\ba}\Big(B^{k-1+12k\eps}\Big(\frac{T^{Ak}}{B^{-\frac1{2k}}}+\frac1T\Big)\Big)
}
with $\bw'=(w_1,0,\dots,0)$. Next we move the line of integration $c_{w_1}$ to $c_{w_1}=-\frac1{4k}+6k\eps$ passing through the pole at $w_1=0$ only, so that bounding trivially the contribution of the new line of integration we obtain
$N'_{1}(B)=C_{1,1}(T)B^{k-1}+O_{k,\eps,\ba}(B^{k-1+12k\eps}(TB^{-\frac1{4k}}+T^{-1})),$
where
\est{
C_{1,1}(T)&:=\Big(\prod_{j=2}^k\frac1{(2\pi i)^2}\int'_{(c_{z_j})}\int'''_{(c_{u_j})} \Big)
\frac{Q_{\bu',\bzero}(\bz',\bzero)}{\zeta(k)^2} \zeta(1+\sum_{j=2}^k{u_j})  \prod_{j=2}^k\zeta(1-u_j)du_{j}dz_j
}
with $\bu'=(-\sum_{j=2}^ku_j,u_2,\dots,u_k)$ and $\bz'=(k-1 -\alpha_1^*-\sum_{i=2}^k(z_i+u_i),z_2,\dots,z_k)=(k-1-\sum_{i=2}^kz_i,z_2,\dots,z_k)$ and lines of integration $c_{u_j}=-2\eps$, $c_{z_i}=1-\frac1k+3\eps$. Notice that~\eqref{bdfq0} in this case gives
\est{
\mathcal Q_{\balpha^*,\bbeta^*}(\bz',\bzero)&\ll_{\ba,k,\eps} \frac{\prod_{i=2}^k|\Im(z_i+u_i)|^{-\frac12+\frac1k}}{|k-1-z_2+\dots+z_k|\prod _{i=2}^k|z_i(z_i+u_i)|}
}
and thus, using the convexity bound $\zeta(1+\sum_{j=2}^k{u_j})\ll_\eps (1+|u_2+\cdots+u_k|)^{(k+1)\eps}$, we obtain
\est{
C_{1,1}(T)&:=C_{1,1}'+O(T^{-\frac1{2}+\frac1k+2k\eps}),
}
where $C_{1,1}'$ is defined as $C_{1,1}(T)$ but where we removed the truncations at $|\Im(z_i)|\leq T$ and $ |\Im(z_i+u_i)|\leq 2T$ from the integrals. Thus,
\est{
N'_{1}(B)=C_{1,1}'B^{k-1}+O_{k,\eps,\ba}(B^{k-1+12k\eps}(TB^{-\frac1{4k}}+T^{-\frac12+\frac1k})).
}

We can treat $N''_{1}(B)$ in the same way, the only difference being that in this case $\sum_{j}z_{j}=k-1-w_1-\sum_{j=2}^k\alpha_j^*$ so that we still obtain a non-negligible contribution only from the summand with $(\alpha_i^*,\beta_i^*)=(u_i,w_i)$ for all $i=2,\dots,k$. We arrive to
\est{
N''_{1}(B)&=\frac1{2\pi i}\int'_{(c_{w_1})} \Big(\prod_{j=2}^k\frac1{(2\pi i)^2}\int'_{(c_{z_j})}\int'''_{(c_{u_j})} \Big)
\sum_{\substack{(\alpha_1^*,\beta_1^*)=(w_1,-w_1-\sum_{j=2}^k u_j)\\ (\alpha_i^*,\beta_i^*)=(u_i,0)\ \forall i=2,\dots,k}}\hspace{-0.7em}Q_{\balpha^*,\bbeta^*}(\bz',\bw')B^{k-1-w_1-\sum_{j=2}^ku_j}\\
& \frac{\zeta(1-2w_1-\sum_{j=2}^k u_j)}{\zeta(k-w_1-\sum_{j=2}^k u_j)^2}\Big(\prod_{j=2}^k\zeta(1-u_j)du_{j}dz_j\Big)\frac{dw_1}{w_1}+O_{k,\eps,\ba}\Big(B^{k-1+12k\eps}\Big(TB^{-\frac1{2k}}+T^{-1}\Big)\Big),
}
with $\bz'=(k-1 +w_1-\sum_{j=2}^ku_j-\sum_{j=2}^kz_j,z_2,\dots,z_k)$ and $\bw'=(w_1,0,\dots,0)$. We move the line of integration $c_{w_1}$ to $c_{w_1}=\frac1{4k}-6k\eps$, passing through a pole at $w_1=-\frac12\sum_{j=2}^ku_j$. The integral on the new line of integration can be bounded trivially, whereas the contribution of the residue is
\est{
& \Big(\prod_{j=2}^k\frac1{(2\pi i)^2}\int'_{(c_{z_j})}\int'''_{(c_{u_j})} \Big)
\sum_{\substack{(\alpha_1^*,\beta_1^*)=(-\frac12\sum_{j=2}^k u_j,-\frac12\sum_{j=2}^k u_j)\\ (\alpha_i^*,\beta_i^*)=(u_i,0)\ \forall i=2,\dots,k}}Q_{\balpha^*,\bbeta^*}(\bz,\bw')\\
&\hspace{10em}\times \frac{-1}{\sum_{j=2}^ku_j} \frac{B^{k-1-\frac12\sum_{j=2}^ku_j}}{\zeta(k-\frac12\sum_{j=2}^ku_j)^2}\Big(\prod_{j=2}^k\zeta(1+u_j) du_{j}dz_j\Big)
}
where $\bw'=(-\frac12\sum_{j=2}^k u_j,0,\dots,0)$. Next, for each $j=3,\dots,k$ we move $c_{u_j}$ to $c_{u_j}=\frac1{2k}-6k\eps$, passing through the pole of $\zeta(1-u_j)$. The contribution on the new line of integration can be bounded trivially, and we obtain that the above is
\est{
&\frac{-1}{2\pi i}\int'_{(c_{u_2})} \Big(\prod_{j=2}^k\frac1{2\pi i}\int'_{(c_{z_j})} \Big)
Q_{\balpha',\bbeta'}(\bz,\bw')\frac{\zeta(1-u_2)}{u_2} \frac{B^{k-\frac12u_2}}{\zeta(k-\frac12u_2)^2}du_2\Big( \prod_{j=2}^k dz_j\Big)\\
&\hspace{24em}+ O_{k,\eps,\ba}\Big(B^{k-1+12k\eps}\Big(TB^{-\frac1{2k}}+T^{-1}\Big)\Big)
}
where $\balpha'=(-\frac12u_2,u_2,0,\dots,0)$ and $\bbeta'=(-\frac12u_2,0,\dots,0)$. Moving $c_{u_2}$ to $c_{u_2}=\frac1{2k}-6k\eps$ picking up the pole at $u_2=0$, we then obtain
\est{
N''_{1,\epsilon}(B)=C_{1,2}(T)B^{k-1}\log B+C_{1,3}(T)B^{k-1}+O_{k,\eps,\ba}\Big(B^{k-1+12k\eps}\Big(TB^{-\frac1{4k}}+T^{-1}\Big)\Big),
}
where
\est{
C_{1,2}(T)&:=\frac{1}2 \Big(\prod_{j=2}^k\frac1{2\pi i}\int'_{(c_{z_j})} \Big)
\frac{Q_{\bzero,\bzero}(\bz,\bzero)}{\zeta(k)^2}\Big( \prod_{j=2}^k dz_j\Big),\\
C_{1,3}(T)&:=- \Big(\prod_{j=2}^k\frac1{2\pi i}\int'_{(c_{z_j})} \Big)
\frac{Q_{\bzero,\bzero}(\bz,\bzero)}{\zeta(k)^2}\Big(\gamma +\frac{\frac{\partial}{\partial u_2}Q_{\balpha',\bbeta'}(\bz,\bw')|_{u_2=0}}{Q_{\bzero,\bzero}(\bzero,\bw')} +\frac{\zeta'(k)}{\zeta(k)}\Big)\Big( \prod_{j=2}^k dz_j\Big)
}
with $\bz=(k-1-\sum_{j=2}^kz_j,z_2,\dots,z_k)$. By~\eqref{bdfq0} and the analogous bound for the logarithmic derivative of $Q_{\balpha',\bbeta'}$ we have
\est{
C_{1,2}(T)&=C_{1,2}'+O(T^{-1}),\qquad C_{1,3}(T)=C_{1,3}'+O(T^{-1}),
}
where $C_{1,2}'$ and $C_{1,3}'$ defined as $C_{1,2},C_{1,3}$ but without the truncation at $|\Im(z_i)|\leq T$ in the integrals. Finally, we write $C_{1,2}'$ as $C_{1,2}'=\mathfrak S(\ba)\sigma_i(\ba)$, where
\est{
\mathfrak S(\ba):=\frac12\sum_{ \ell\geq1}\frac{\prod_{i=1}^{k}(a_i,\ell)}{\ell^{k}}\frac{\varphi(\ell)}{\zeta(k)^2}
}
and
\est{
\sigma'_1(\ba)&:=\Big(\prod_{i=2}^k\frac1{2\pi i}\int_{(c_{z_i})} \Big)\frac{2^{k}\pi^{\frac{k}2-1}|a_1|^{k-2-z_2-\cdots-z_k}}{k-1-z_2-\cdots-z_k} \frac{\Gamma(\frac{2-k+z_2+\dots+z_k}2)}{\Gamma(\frac{k-1-z_2-\cdots-z_k}2)}\prod_{i=2}^k\frac{\Gamma(\frac{1-z_i}2)}{\Gamma(\frac{z_i}2)|a_i|^{1-z_i}} \frac{dz_i}{z_i^2} .
}
Summarizing, we proved
\est{
N_1(B)=B^{k-1}(\mathfrak S(\ba)\sigma'_1(\ba)\log T+C_{1,1}'+C_{1,3}')+O_{k,\eps,\ba}(B^{k-1+26k\eps}(TB^{-\frac1{4k}}+T^7B^{-\frac1{2k}}+T^{-\frac12+\frac1k})).
}
Thus, Theorem~\ref{sectp} follows by taking $T=B^{\frac1{15k-2}}$ and applying the following Lemma.
\begin{lemma}
For $k\geq3$ we have $\sigma'_1(\ba)=\sigma_1(\ba)$
where $\sigma_1(\ba)$ is as in~\eqref{fsda}.
\end{lemma}
\begin{proof}
We only give a sketch, leaving the problem of justifying the manipulation of certain conditionally convergent integrals to the interested reader.

First, by symmetry, we observe that 
\est{
\sigma_1(\ba)&=2^{k}\sum_{\epsilon=(\pm_21,\dots,\pm_k1)\in\{\pm1\}^{k-1}}\Big(\prod_{i=2}^k\int_{0}^1\int_{0}^1\Big) \chi_{(0,1)}\Big(\sum_{i=2}^k(\pm_ia_{i}x_iy_i/a_1)\Big)\prod_{i=2}^kdx_idy_i.\\
}
Then, we detect the characteristic function $\chi_{(0,1)}$ using its Mellin transform obtaining  
\est{
\sigma_1(\ba)&=\sum_{\epsilon=(\pm_21,\dots,\pm_k1)\in\{\pm1\}^{k-1}}\Big(\prod_{i=2}^k\int_{0}^1\int_{0}^1\Big)\frac1{2\pi i}\int_{(c_{z_1})}\frac{2^k (\sum_{i=2}^k(\pm_ia_{i}x_iy_i/a_1))^{z_1-1}}{1-z_1}\\
&\hspace{18em}\times \chi_{\R_{>0}}\Big(\sum_{i=2}^k(\pm_ia_{i}x_iy_i/a_1)\Big)dz_1\prod_{i=2}^kdx_idy_i\\
}
with $c_{z_1}=\frac1k$.
Then, we use Lemma~\ref{sml} below with $B=0$\footnote{The Lemma is stated for $B$ large to avoid issues coming with the conditional convergence of the integrals. To make this rigorous it's enough to take a larger $B$, and then later in the argument recompose the sum over $\nu$ using~\eqref{beid}.} obtaining 
\est{
\sigma_1(\ba)&=\sum_{\epsilon=(\pm_21,\dots,\pm_k1)\in\{\pm1\}^{k-1}}\Big(\prod_{i=2}^k\int_{0}^1\int_{0}^1\Big)\Big(\prod_{i=1}^k\frac1{2\pi i}\int_{(c_{z_i})}-\prod_{i=1}^k\frac1{2\pi i}\int_{(c'_{z_i})} \Big)\frac{2^k}{1-z_1}\\
&\quad\times\frac{ G(1-z_1-z_2-\dots-z_k)}{1-z_1-z_2-\dots-z_k} \frac{\Gamma(z_1)\prod_{i=2}^k\Gamma(z_i)|a_ix_iy_i/a_1|^{-z_i}}{\Gamma(\sum^+_{1\leq i\leq k}z_i)\Gamma(\sum^-_{1\leq i\leq k}z_i)}\Big(\prod _{i=1}^k dz_i\Big)\prod_{i=2}^kdx_idy_i,\\
}
where $G(s)$ is entire with $G(0)=1$, $\sum_i^{\pm}$ indicates the sum is restricted to indexes such that $\pm_i\sign(a_i)=\pm1$, with $\pm_11:=-\sign(a_1)$, and the lines of integrations are $c'_{z_i}=\frac{1}{k-1}$ and $c_{z_i}=\frac1{k+1}$ for $i=2,\dots,k$ and $c_{z_1}=c_{z_1}'=\frac1k$. Then, we notice that we can take instead $c_{z_i}=c'_{z_i}=\frac1{k}$ for $i=2,\dots,k$ and $c_{z_1}=\frac1{k+1},c_{z_1}'=\frac1{k-1}$. We take the integral over $x_i$ and $y_i$ inside and execute them, obtaining 
\est{
\sigma_1(\ba)&=\sum_{\epsilon=(\pm_21,\dots,\pm_k1)\in\{\pm1\}^{k-1}}\Big(\prod_{i=2}^k\frac1{2\pi i}\int_{(c_{z_i})}\Big)\Big(\frac1{2\pi i}\int_{(c_{z_1})}-\,\frac1{2\pi i}\int_{(c'_{z_1})} \Big)\frac{2^k}{1-z_1}\\
&\quad\times\frac{ G(1-z_1-z_2-\dots-z_k)}{1-z_1-z_2-\dots-z_k} \frac{\Gamma(z_1)\prod_{i=2}^k\Gamma(z_i)|a_i/a_1|^{-z_i}}{\Gamma(\sum^+_{1\leq i\leq k}z_i)\Gamma(\sum^-_{1\leq i\leq k}z_i)}dz_1\Big(\prod _{i=1}^k \frac{dz_2}{(1-z_i)^2}\Big).
}
By the residue theorem the difference of the integrals in $z_1$ is equal to minus the residue at $z_1=1-z_2-\dots-z_k$ and so
\est{
\sigma_1(\ba)&=\sum_{\epsilon=(\pm_21,\dots,\pm_k1)\in\{\pm1\}^{k-1}}\Big(\prod_{i=2}^k\frac1{2\pi i}\int_{(c_{z_i})}\Big)\frac{2^{k}}{1-z_1}\frac{\Gamma(z_1)\prod_{i=2}^k\Gamma(z_i)|a_i/a_1|^{-z_i}}{\Gamma(\sum^+_{1\leq i\leq k}z_i)\Gamma(\sum^-_{1\leq i\leq k}z_i)}\prod _{i=2}^k \frac{dz_i}{(1-z_i)^2}
}
where $z_1:=1-z_2-\dots-z_k$. We take the sum over $\epsilon$ inside, and evaluate it using~\eqref{regac} (notice that since the value of $\pm_11$ is fixed we have to multiply by $\frac12$). We obtain
\est{
\sigma_1(\ba)&=\Big(\prod_{i=2}^k\frac1{2\pi i}\int_{(c_{z_i})}\Big)\frac{2^{k}\pi^{\frac{k}2-1}}{z_2+\dots+z_k}\frac{\Gamma(\frac {1-z_2-\dots-z_k}2)}{\Gamma(\frac {z_2+\dots+z_k}2)}\prod_{i=2}^k\frac{\Gamma(\frac {z_i}2)|a_i/a_1|^{-z_i}}{\Gamma(\frac {1-z_i}2)}\frac{dz_i}{(1-z_i)^2}.
}
Making the change of variables $z_i\to1-z_i$ for all $i=2,\dots,k$ we obtain $\sigma'_1(\ba)$.
\end{proof}
\section{The region of absolute convergence}\label{rab}
In this section we prove a bound for $\mathcal A_{\ba;\balpha,\bbeta}(s)$ in the region of absolute convergence. We remark that if we were not concerned with the uniformity in $k$, then an easier argument would have sufficed. 

\begin{lemma}\label{absc}
Let $k\geq3$, $\ba\in\Z_{\neq0}^k$ and $\balpha,\bbeta\in\C^{k}$ with  $|\Re(\alpha_i)|,|\Re(\beta_i)|\leq\frac1{2(k-1)}$ for all $i=1,\dots,k$. Then $\mathcal A_{\ba;\balpha,\bbeta}(s)$ converges absolutely on $\Re(s)>1-\frac1k-\frac1k\sum_{i=1}^k\min(\Re(\alpha_i),\Re(\beta_i))$. Moreover, if 
\es{\label{cos}
\Re(s)\geq1-\frac1k+\frac\eps k-\frac1k\sum_{i=1}^k\min(\Re(\alpha_i),\Re(\beta_i))
} 
for some $\eps>0$, then $\mathcal A_{\ba;\balpha,\bbeta}(s)\ll_\eps  \pr{\frac{A k}\eps}^{4k}$ where the implicit constant depends on $\eps$ only and $A$ is an absolute constant.
\end{lemma}
\begin{proof}
Clearly, we can assume $s,\alpha_i,\beta_i\in\R$ and $\alpha_i\leq\beta_i$ for all $i=1,\dots,k$ and that $\eps<\frac 18$. Also, it is sufficient to establish the claimed bound for $s=1-\frac1k-\frac1k\sum_{i=1}^k\alpha_i+\frac {4\eps} k$.

We have  $\tau_{\alpha,\beta}(n)\leq  d(n) n^{\max(-\alpha,-\beta)}$ and so
\est{
\mathcal A_{\ba;\balpha,\bbeta}(s)\leq \sum_{i=1}^k \sum_{\substack{a_1n_1+\cdots+a_kn_k=0,\\ \max_{j=1}^k(n_j)=n_i}}\frac{d(n_1)\cdots d(n_k)}{n_1^{s_1}\cdots n_{k}^{s_k}}=\sum_{i=1}^kS_i,
}
say, where we wrote $s_i:=s+\alpha_i$. By hypothesis $s_1-\eps>1-\frac1k-\frac1{k-1}-\eps>0$ and so
\est{
S_1& \ll_\eps \sum_{\substack{a_1n_1+\cdots+a_kn_k=0,\\ \max_{j=1}^n(n_j)=n_1}}\frac{d(n_2)\cdots d(n_k)}{n_1^{s_1-\eps}n_2^{s_2}\cdots n_{k}^{s_k}}\leq \sum_{\ell=1}^{k-1}\sum_{\substack{n_2,\dots,n_k\geq1,\\ \max_{r=2}^{k}(n_r)=n_\ell}}\frac{d(n_2)\cdots d(n_k)}{n_\ell^{s_1-\eps}n_2^{s_2}\cdots n_{k}^{s_k}}=\sum_{\ell=2}^kZ_{1,\ell},
}
say.
We write
\est{
c_r:=1-s_r+\xi_r\frac{\eps}k=\frac1k+\frac1k\sum_{\substack{i=1,\\ i\neq r}}^k\alpha_i+\frac{k-1}{k}\alpha_i-\eps+\xi_r\frac{\eps}k
} 
where $\xi_r$ is any real number in the interval $(1,2)$ such that  $|c_r|>\frac\eps {2k}$ (notice that since $|\alpha_r|\leq\frac1{2(k-1)}$ we also have $c_r>-\eps+\xi_r\frac{\eps}k>-\frac18$). Now, 
\est{
\chi_{[0,1]}(x)\leq 2\pr{1-\frac x{2}}\chi_{[0,2]}(x)=\frac1{\pi i}\int_{(1)}(x/2)^{-w}\,\frac{dw}{w(w+1)}=2\chi_{I}(r)+\frac1{\pi i}\int_{(c_r)}(x/2)^{-w}\,\frac{dw}{w(w+1)},
}
by the residue theorem, where $I$ is the set of $r\in\{3,\dots,k\}$ for which $c_r<0$ and $\chi_I$ is the characteristic function of the set $I$. We replace the condition $n_r\leq n_2$ for all $r=3,\dots,k$ by this formula and obtain 
\est{
Z_{1,2}&\leq \sum_{\substack{n_2,\dots,n_k\geq1}}\frac{d(n_2)\cdots d(n_k)}{n_2^{s_1+s_2-\eps}n_3^{s_3}\cdots n_{k}^{s_k}}\prod_{r=3}^k\Big(2\chi_I(r)+\frac1{\pi i}\int_{(c_r)}\pr{\frac{n_r}{2n_2}}^{-w_r}\,\frac{dw_r}{w_r(w_r+1)}\Big)\\
&=\sum_{J\subseteq I}2^{|J|}\prod_{j\in J}\zeta(1+\tfrac \eps k)^2\sum_{\substack{n_2\geq1}}\frac{d(n_2)}{n_2^{s_1+s_2-\eps}}\prod_{\ell \in L}\sum_{n_\ell\geq 1}\frac{d(n_\ell)}{n_\ell^{s_\ell}}\frac1{\pi i}\int_{(c_\ell)}\pr{\frac{n_\ell}{2n_2}}^{-w_\ell} \,\frac{dw_\ell}{w_\ell(w_\ell+1)},
}
since $s_i>1+\frac \eps k$ for all $i\in I$, and where $L:=\{3,\dots,k\}\setminus J$.  (Here and below the exchanges in the orders of sums and integrals are justified by the absolute convergence).
Exchanging the order of summation and integration and summing the Dirichlet series we obtain
\est{
Z_{1,2}&\leq \sum_{J\subseteq I}2^{|J|}\Big(\prod_{j\in J}\zeta(s_j)^2\Big)\Big(\prod_{\ell \in L} \frac1{\pi i}\int_{(c_\ell)} \zeta(s_\ell +w_\ell)^2 \Big)\zeta\Big(s_1+s_2-\eps-\sum_{\ell \in L}w_\ell\Big)^2\prod_{\ell\in L}^k\frac{dw_\ell}{w_\ell(w_\ell+1)} .
}
The real part of the argument of the last $\zeta$ in the above equation is 
\est{
&s_1+s_2+ \sum_{\ell \in L}\pr{s_\ell -\xi_\ell\frac{\eps}k}-|L|-\eps\geq (|L|+2)s-|L|-\eps-2\frac{\eps |L|}{k}+\sum_{\ell\in L\cup\{1,2\}}\alpha_\ell \\
&\hspace{2em}=1 +\frac{k-|L|-2}{k}+\frac{ 2|L|-k+8}{k}\eps-\frac{|L|+2}{k}\sum_{j\in J}\alpha_j+\frac{k-|L|-2}{k}\sum_{\ell\in L\cup\{1,2\}}\alpha_\ell\\
}
by the definition of $s_\ell$ and since $\xi_\ell\leq 2$ and $s=1-\frac1k-\frac1k\sum_{i=1}^k\alpha_i+\frac {4\eps} k$. Since $|\alpha_i|\leq \frac1{2(k-1)}$ for all $i=1,\dots,k$, then the above expression is
\est{
&\hspace{2em}\geq 1  +\frac{k-|L|-2}{k}+\frac{ 2|L|-k+8}{k}\eps-\frac{(k-|L|-2)(|L|+2)}{k(k-1)}\\
&\hspace{2em}=  1+\frac{(k-|L|-2)(k-|L|-3)}{k(k-1)}+\frac{ 2|L|-k+8}{k}\eps=  1+\frac{m(m-1)}{k(k-1)}+\frac{k-2m+4}{k}\eps,\\
}
where $m:=k-2-|L|$ (so that $0\leq m\leq k-2$). Thus, if $m< k/2$ then $\Re(s_1+s_2-\eps-\sum_{\ell \in L}w_\ell)\geq 1+\frac{4}{k}\eps$, and the same holds if $m\geq k/2$ since in this case we have
$$1+\frac{m(m-1)}{k(k-1)}+\frac{k-2m+4}{k}\eps\geq 1+\frac{(k-2)}{4(k-1)}-\eps+\frac{8}{k}\eps\geq 1+\frac{1}{8}-\eps+\frac{8}{k}\eps>1+\frac{8}{k}\eps$$
for $k\geq3$ and $\eps<\frac18$.
It follows that
\est{
Z_{1,2}&\leq \sum_{J\subseteq I} \frac{2^{|J|}}{\pi^{|L|}}\zeta(1+4\tfrac{\eps}k)^2\Big(\prod_{j\in J}\zeta(1+\tfrac \eps k)^2\Big)\prod_{\ell \in L}\zeta(1+\xi_\ell\tfrac{\eps}k)^2 \int_{(c_\ell)}\frac{|dw_\ell|}{|w_\ell(w_\ell+1)|}
}
since by our choice we have $s_\ell+c_\ell=1+\xi_\ell\frac{\eps}k>1 $ for all $\ell\in L$. 
Finally, since $|c_\ell|>\frac\eps {2k}$ and $c_\ell>-\frac12$ we have
\est{
 \int_{(c_\ell)}\frac{|dw_\ell|}{|w_\ell(w_\ell+1)|}\leq \frac{2k}{\eps}\int_\R \frac{1}{|(1+ix)(\frac12+ix)|}\leq \frac{2k}{\eps}\int_\R \frac{1}{|\frac12+ix|^2}=\frac{4\pi k}\eps
}
and so, since $s_r>1+\xi_r\frac{\eps}k$ for all $r\in I$ and $s_\ell+c_\ell>\frac\eps k$ for all $r=3,\dots,k$, we have
\est{
Z_{1,2}&\leq \zeta(1+\tfrac\eps k)^{2k-2}\sum_{J\subseteq I} \frac{2^{|J|}}{\pi^{|L|}}\pr{\frac{4\pi k}\eps}^{|L|}\leq  k!\pr{\frac{8 k}\eps}^{k} \zeta(1+\tfrac\eps k)^{2k-2}\ll \pr{\frac{A k}\eps}^{4k}
}
for some fixed $A>0$. The same bound clearly holds for $Z_{i,j}$ for all $i,j$ and so the Lemma follows.
\end{proof}

\section{Proof of Theorem~\ref{mtmt}}\label{mtmp}
Throughout the rest of the paper, we shall write
\es{\label{dfxy}
X_{s,k,\ba,\eps}:=\tfrac k\eps(1+|s|)\max_{i=1}^k|a_i| ,\qquad 
Y_{\balpha,\bbeta}:=(1+\max_{i=1}^{k}(|\Im(\alpha_i)|+|\Im(\beta_i)|). 
}
Then, we notice that instead of~\eqref{bdmt}, it is enough to prove
\es{\label{inut}
\mathcal A_{\ba;\balpha,\bbeta}(s)-\mathcal M_{\ba;\balpha,\bbeta}(s)\ll_{\eps}\! ({X_{s,k,\ba,\eps} Y_{\balpha,\bbeta}})^{Ak}
}
for $\Re(s)\geq 1-\tfrac{2-2\eta_{\balpha,\bbeta}-\eps}{k+1}$. Indeed, we can apply the Phragmen-Lindel\"of principle on the region  $1-\tfrac{2-2\eta_{\balpha,\bbeta}-\eps}{k+1}\leq \Re(s)\leq 1-\tfrac{1-\eps}{k}+\tfrac{\eta_{\balpha,\bbeta}}{k+1}$ using~\eqref{inut} on the left boundary line and the bound for $\mathcal A_{\ba;\balpha,\bbeta}(s)$ given in Lemma~\ref{absc} with a trivial bound for $\mathcal M_{\ba;\balpha,\bbeta}(s)$ on the line $\Re(s)= 1-\tfrac{1-\eps}{k}+\tfrac{\eta_{\balpha,\bbeta}}{k+1}$. Also, in~\eqref{bdmt} we can take $(1+|s|)^7$ rather than $(1+|s|)^A$ since for $s=\sigma+it$ one has $\mathcal A_{\ba;\balpha,\bbeta}(s)=\mathcal A_{\ba;\balpha+\bt',\bbeta+\bt'}(\sigma)$ where $\bt':=(t,\dots,t)$.

Furthermore, we notice it is sufficient to prove Theorem~\ref{mtmt} in the case $|\alpha_i-\beta_i|> \frac \eps {k}$ for $i=1,\dots,k$. Indeed, assume we have proved Theorem~\ref{mtmt} in this restricted case and let
\est{
\mathcal F_{\ba;\balpha,\bbeta}(s):=\frac1{2\pi i}\int_{|\xi_1|=4\frac \eps {k}} (\mathcal A_{\ba;\balpha'(\xi),\bbeta}(s)-\mathcal M_{\ba;\balpha'(\xi),\bbeta}(s))\,\frac{d\xi}{\xi},
}
where $\balpha'(\xi)=(\alpha_1',\dots,\alpha_k'):=(\alpha_1+\xi,\alpha_2,\dots,\alpha_k)$. Then, $\mathcal F_{\ba;\balpha,\bbeta}(s)$ defines a holomorphic function in $(s,\balpha,\bbeta)$ in the domain
\est{
D:=\Big\{(s,\balpha,\bbeta) \Big| \substack{|\alpha_1-\beta_1|<\frac{2\eps}{k}, |\alpha_i-\beta_i|> \frac \eps {k},\ \forall i\geq2,|\Re(\alpha_i)|,|\Re(\beta_i)|\leq \frac{1-9\eps}{2(k-1)}\ \forall i\geq 1,\\
\Re(s)>1-\frac{2-2\eta_{\balpha,\bbeta}-10\eps}{k+1}
}
\Big\}.
}
Moreover, by Cauchy's theorem we have $\mathcal F_{\ba;\balpha,\bbeta}(s)=\mathcal A_{\ba;\balpha,\bbeta}(s)-\mathcal M_{\ba;\balpha,\bbeta}(s)$ in the subset of $D$ with $\Re(s)>4$, since $\mathcal A_{\ba;\balpha,\bbeta}(s)$ is clearly holomorphic also for $\alpha_1=\beta_1$ when the series defining it is absolutely convergent and the same holds for $\mathcal M_{\ba;\balpha,\bbeta}(s)$ since the poles of 
$\mathcal M_{\ba;\balpha'(\xi),\bbeta}(s)$ at $\alpha_1=\beta_1$ cancel. Thus, we have obtained the analytic continuation of  $\mathcal A_{\ba;\balpha,\bbeta}(s)-\mathcal M_{\ba;\balpha,\bbeta}(s)$ on the domain given by $D$ without any condition on $|\alpha_1-\beta_1|$ and by the above integral representation the bound~\eqref{inut} holds also in the case $|\alpha_1-\beta_1|\leq \frac \eps {k}$. Repeating this procedure for $i=2,\dots,k$ one obtains Theorem~\ref{mtmt} in the general case.
Thus, in the following we will assume $(\balpha,\bbeta)\in \Omega_k$ with
\es{\label{dfok}
\Omega_k:=\{(\balpha,\bbeta)\in \C^k\times \C^k\mid \tfrac {\eps}k<|\alpha_i-\beta_i|,\ \ |\alpha_i|,|\beta_i|<\tfrac{1}{2(k-1)}\, \forall i=1,\dots,k\}.
}
Finally, we can and shall assume $\sigma:=\Re(s)\ll1$.

We localize the variables of summation by introducing partitions of unity
\est{
\sumdagger_{N}P(x/N)=x,\qquad \forall x>0,
}
such that $\sumdagger_{X^{-1}\leq N\leq X}1\ll \log X$ and with $P(x)$ supported on $1\leq x\leq 2$ and satisfying $P^{(j)}(x)\ll j^{Aj}$ for some $A>0$. Notice that under these conditions, the Mellin transform of $P(x)$,
\est{
\tilde P(s):=\int_{0}^{\infty}P(x)x^{s-1}dx,
}
is entire and satisfies
\es{\label{fnn}
\tilde P(\sigma+it)\ll (1+r+|\sigma|)^{Ar}C^{|\sigma|}(1+|t|)^{-r},\qquad \forall r\geq0
}
for some $C>0$.
Thus, for $s$ satisfying~\eqref{cos} we can write 
\est{
\mathcal A_{\ba;\balpha,\bbeta}(s):=\sumdagger_{N_1,\dots,N_k} A_{\balpha,\bbeta}(s),
}
where 
\es{\label{defna}
 A_{\balpha,\bbeta}(s):=\sum_{a_1n_1+\cdots+a_kn_k=0}\frac{\sigma_{\alpha_1-\beta_1}(n_1)\cdots \sigma_{\alpha_k-\beta_k}(n_k)}{n_1^{s+\alpha_1}\cdots n_{k}^{s+\alpha_k}}P(n_1/N_1)\cdots P(n_k/N_k)
}
and, here and in the following, we omit to indicate the dependence on $\ba$ and $N_1,\dots,N_k$ to save notation. The main step in the proof consists in the following lemma which we shall prove in Section~\ref{pmmlle}.
\begin{lemma}\label{mmlle}
We have
\es{\label{vsy}
 A_{\balpha,\bbeta}(s)=\sum_{\substack{\mathcal I\cup \mathcal J=\{1,\dots,k\},\\ \mathcal I\cap \mathcal J=\emptyset,\\ |\mathcal I|>\frac{k+1}{2}}}\sum_{\substack{\{\alpha_i^*,\beta_i^*\}=\{\alpha_i^*,\beta_i^*\}\ \forall i\in \mathcal I,\\ (\alpha_j^*,\beta_j^*)=(\alpha_j,\beta_j)\ \forall j\in \mathcal J}}\prod_{i\in \mathcal I}\zeta(1-\alpha_i+\beta_i)Z_{\mathcal I;\balpha^*,\bbeta^*}(s),
 + E_1(s),
}
where $Z_{\mathcal I;\balpha,\bbeta}$ is as in~\eqref{dfnz} and $E_1(s)$ is holomorphic in $(s,\balpha,\bbeta)\in\C\times\Omega_k$ and satisfies 
\es{\label{bvsy}
E_1(s)&\ll_\eps X_{s,k,\ba,\eps}^{Ak} Y_{\balpha,\bbeta}^{7k}N_{\max}^{\eta_{\balpha,\bbeta}+2\eps} \\
&\quad\times\Big(\frac{N_{\max}^{\frac{k-1}2}}{(N_1\cdots N_k)^{\sigma}}+\frac{N_{\max}^{\frac{k}2-1-\frac1{k-1}}}{(N_1\cdots N_k)^{\sigma-\frac1{k-1}}}+\frac{ (N_1\cdots N_k)^{1-\frac1k-\sigma}}{N_{\max}}+N_{\max}^{\frac{k-3}2-\frac{k+1}{2}\sigma}\Big)
}
for $\frac12\leq\sigma\leq 1$, $(\balpha,\bbeta)\in\Omega_k$, with $\eta_{\balpha,\bbeta}$ is as in~\eqref{xi} and $N_{\max}:=\max_{i=1}^kN_i$.
\end{lemma}

In Section~\ref{pmmlle2} we shall prove the following lemma.
\begin{lemma}\label{mmlle2}
We have
\es{\label{mmlle2e}
 A_{\balpha,\bbeta}(s)=\sum_{\substack{\mathcal I\cup \mathcal J=\{1,\dots,k\},\\ \mathcal I\cap \mathcal J=\emptyset,\\ |\mathcal I|>\frac{k+1}{2}}}\sum_{\substack{\{\alpha_i^*,\beta_i^*\}=\{\alpha_i^*,\beta_i^*\}\ \forall i\in \mathcal I,\\ (\alpha_j^*,\beta_j^*)=(\alpha_j,\beta_j)\ \forall j\in \mathcal J}}\prod_{i\in \mathcal I}\zeta(1-\alpha_i+\beta_i)Z_{\mathcal I;\balpha^*,\bbeta^*}(s),
 + E_2(s),
}
where $E_2(s)$ is holomorphic in $(s,\balpha,\bbeta)\in\C\times\Omega_k$ and satisfies 
 \es{\label{tbb}
E_2(s)\ll_\eps X_{s,k,\ba,\eps}^{Ak} Y_{\balpha,\bbeta}^{k}(N_1\cdots N_k)^{1-\sigma}
N_{\max}^{-1+\eta_{\balpha,\bbeta}+\eps}.
}
\end{lemma}

Finally, in Section~\ref{pcompl} we complete the proof of Theorem~\ref{mtmt} summing over the partitions of unity and computing the minimum of the two error terms.

\section{Proof of Lemma~\ref{mmlle}}\label{pmmlle}

Since both $ A_{\balpha,\bbeta}(s)$ and the main term on the right side of~\eqref{vsy} are symmetric in $N_1,\dots,N_k$ and with respect to the change $\ba\leftrightarrow-\ba$ (cf. Remark~\ref{fnlr} at the end of Subsection~\ref{cl}), we can and shall assume that $N_1$ is the maximum among the $N_i$ and that $a_1<0$.

\subsection{Separating the variables}
The condition $-a_1n_1=a_2n_2+\cdots+a_kn_k$ can be used to eliminate the variable $n_1$ in the definition~\eqref{defna} of $ A_{\balpha,\bbeta}(s)$, by adding the conditions 
$$ a_2n_2+\cdots+a_kn_k\equiv 0\mod {|a_1|},\quad
a_2n_2+\cdots+a_kn_k>0$$
and replacing each occurrences of $n_1$ by $(a_2n_2+\cdots+a_kn_k)/|a_1|$. This poses the problem of expressing $\tau_{\alpha_1,\beta_1}(n_1)=n_1^{-\alpha}\sigma_{\alpha-\beta}(n_1)$ in a more flexible way, which we achieve by the following modification of Ramanujan's identity~\eqref{rffm} which also allows us to remove the above congruence condition.

\begin{lemma}
Let $a,m\in\N$ and $\gamma\in\C$. Then we have
\est{
\delta_{a|n}\sigma_\gamma(n/a)&=\frac1a\sum_\ell\frac{c_\ell(n)}{\ell^{1-\gamma}}(\ell,a)^{1-\gamma}\upsilon_\gamma \pr{\frac{a\ell^2}{(  a,\ell)^2n}}+n^\gamma\sum_\ell\frac{c_\ell(n)}{\ell^{1+\gamma}}(\ell,a)^{1+\gamma}\upsilon_{-\gamma}\pr{\frac{a\ell^2}{(a,\ell)^2n}},\\
}
where $\delta_{a|n}=1$ if $a|n$ and $\delta_{a|n}=0$ otherwise,  $c_\ell(n):=\sumstar_{h\mod n}\e{\frac {hn}\ell}$ is the Ramanujan sum and for any $c_w>|\Re(\gamma)|$ 
\es{\label{dssv}
\upsilon_\gamma\pr{x}=\int_{(c_w)}x^{-\frac w2}\zeta(1-\gamma+w)\frac{G(w)}{w}dw,
}
where $G(w)$ is any even entire function which decays faster than any polynomial in vertical strips and is such that $G(0)=1$.
\end{lemma}
\begin{proof}
We start by observing that for $\Re(s)>1$ we have
\es{\label{dsd}
\delta_{a|m}\sum_{\ell\geq1}\frac{c_{\ell}(m/a)}{\ell^s}=\frac1{a}\sum_{\ell\geq 1}\frac{c_{\ell}(m)}{\ell^s}(\ell,a)^s.
}
Indeed, by the orthogonality of additive characters, the left hand side is
\est{
&\frac1 a\sum_{g=1}^a \sum_{\ell\geq1}\frac{1}{\ell^s}\sum_{\substack{h=1,\\ (h,\ell)=1}}^\ell\e{\frac{mh+mg\ell}{a\ell}}=
\frac1 a \sum_{\ell\geq1}\frac{1}{\ell^s}\sum_{\substack{h\mod {a\ell},\\ (h,\ell)=1}}\e{\frac{mh}{a\ell}}\\
&\hspace{2em}=\frac1a \sum_{\ell\geq1}\frac{1}{\ell^s} \sum_{\substack{d|a,\\ (a/d,\ell)=1}}\sum_{\substack{h\mod {d\ell },\\ (h,d\ell )=1}}\e{\frac{mh}{d\ell}}=\frac1a \sum_{\ell\geq1}\frac{1}{\ell^s} \sum_{\substack{d|a,\\ (a,d\ell)=d}}c_{d\ell}(m)=\frac1a \sum_{\substack{\ell'\geq1,\\ (a,\ell')=d}}\frac{1}{(\ell'/d)^s} c_{\ell'}(m),
}
as claimed. By Ramanujan's identity~\eqref{rffm}, one has that~\eqref{dsd} gives
\es{\label{mll}
\delta_{a|m}\sigma_{1-s}(m/a)=\zeta(s)\frac1{a}\sum_{\ell\geq 1}\frac{c_{\ell}(m)}{\ell^s}(\ell,a)^s
}
for $\Re(s)>1$.
Now, by the residue theorem for any $c_w>|\Re(\gamma)|$ we have
\est{
\sigma_{\gamma}(m/a)&=\int_{(c_w)}(m/a)^{\frac{w}2}\sigma_{\gamma-w}(m/a)\frac{G(w)}{w}\,dw-\int_{(-c_w)}(m/a)^{\frac{w}2}\sigma_{\gamma-w}(m/a)\frac{G(w)}{w}\,dw\\
&=\int_{(c_w)}(m/a)^{\frac{w}2}\sigma_{\gamma-w}(m/a)\frac{G(w)}{w}\,dw+\int_{(c_w)}(m/a)^{\gamma+\frac w2}\tau_{-\gamma-w}(m/a)\frac{G(w)}{w}\,dw
}
since $\sigma_{\gamma+w}(m)=m^{\gamma+w}\tau_{-\gamma-w}(m)$ and $G(-w)=G(w)$, and so the lemma then follows by~\eqref{mll}.
\end{proof}
\begin{remark}\label{rmdfG}
It will be convenient to take $G(w)$ with a zero which cancels the pole of the zeta-function in the definition~\eqref{dssv} of $\upsilon$. Thus we take
\est{
G_{\alpha_1,\beta_1}(w):=\frac{w^2-(\alpha_1-\beta_1)^2}{(\alpha_1-\beta_1)^2}\frac{\xi(\frac 12+w)}{\xi(\frac12)},
}
where $\xi(s):=\frac12s(s-1)\pi^{- \pi/2}\Gamma(\frac12w)\zeta(w)$ is the Riemann $\xi$-function. Notice that $G(0)=1$ and that by the functional equation we have $G_{\alpha_1,\beta_1}(w)=G_{\alpha_1,\beta_1}(-w)$. Also, by Stirling's formula we have
\es{\label{bndG}
G_{\alpha_1,\beta_1}(w)&\ll  (k/\eps)^2e^{-\frac{\pi}{8}|\Im(w)|}(1+|\Re(w)|)^{A(|\Re(w)|+1)}\\
&\ll (k/\eps)^2(1+r+|\Re(w)|)^{A(1+r+|\Re(w)|)}(1+|w|)^{-r}
}
for any $r\geq0$.

\end{remark}

Applying this lemma 
we obtain
\est{
 A_{\balpha,\bbeta}(s)=\sum_{\substack{\{\alpha_1^*,\beta_1^*\}=\{\alpha_1,\beta_1\},\\ (\alpha_i^*,\beta_i^*)=(\alpha_i,\beta_i)\ \forall i= 2,\dots,k}} K_{\balpha^*,\bbeta^*}(s),
}
where the sum is over $\balpha^*=(\alpha_1^*,\dots,\alpha_k^*)$, $\bbeta^*=(\beta_1^*,\dots,\beta_k^*)$ and 
\est{
K_{\balpha,\bbeta}&:=\frac1{|a_1|}\sum_{\substack{n_2,\dots,n_k\geq1,\\ n_1=(a_2n_2+\cdots+a_kn_k)/|a_1|>0}}\hspace{-1.3em}n_1^{-\alpha_1}\sum_{\ell\geq1}\frac{c_\ell\pr{a_2n_2+\cdots+a_k}}{\ell^{1-\alpha_1+\beta_1}}(\ell,a_1)^{1-\alpha_1+\beta_1}\times\\
&\quad\times\frac{\tau_{\alpha_2,\beta_2}(n_2)\cdots \tau_{\alpha_k,\beta_k}(n_k)}{(n_1 \cdots n_{k})^{s}}P(n_1/N_1)\cdots P(n_k/N_k)\upsilon_{\alpha_1-\beta_1}\pr{\frac {\ell^2} {(\ell,a_1)^2n_1}}.
}

Next, we express $P$ and $\upsilon$ using their Mellin transforms so that, after making the change of variables $u_i\rightarrow u_i-s$ for all $i\in\{1,\dots,k\}$, we obtain
\es{\label{tgrso}
K_{\balpha,\bbeta}
&=\sum_{\pm_2a_2n_2\pm_3\cdots\pm_ka_kn_k>0}\sum_{\ell\geq1}\sumstar_{h\mod \ell}\frac1{(2\pi i)^{k+1}}\int_{(c_w,c_{u_1},\dots,c_{u_k})}\frac{N_1^{u_1-s}\cdots N_k^{u_k-s}}{(\ell/(a_1,\ell))^{1-\alpha_1+\beta_1+w}} \\
&\quad\times\frac{\tau_{\alpha_2,\beta_2}(n_2)\cdots \tau_{\alpha_k,\beta_k}(n_k)c_\ell(a_2n_2+\cdots+a_kn_k)(\ell,a_1)^{1-\alpha_1+\beta_1+w}}{(a_2n_2+\cdots+a_kn_k)^{\alpha_1+u_1-\frac w2}n_2^{u_2}\cdots n_{k}^{u_k}|a_1|^{1-\alpha_1-u_1+\frac w2}}\\
&\quad \times \tilde P(u_1-s)\cdots  \tilde P(u_k-s)\zeta(1-\alpha_1+\beta_1+w)\frac{G_{\alpha_1,\beta_1}(w)}wdwdu_1\cdots du_k,
}
with lines of integrations 
\es{\label{fcolio}
c_{w}=|\Re(\alpha_1-\beta_1)|+2\eps, \quad c_{u_1} =-4k-\Re(\alpha_1)+\tfrac {c_w}2+ \eps , \quad c_{u_2}=\cdots=c_{u_k}=5k.
}

Next, we separate the variables in the expression $(a_2n_2+\cdots+a_kn_k)^{\frac12+\alpha_1+u_1-\frac w2}$ using the following lemma which we quote from Section~10 of~\cite{Bet} in a slightly adapted form.

\begin{lemma}\label{sml}
Let $\kappa \geq2$ and  $x_1,\dots x_\kappa >0$. Let  $\epsilon=(\pm_1,\cdots,\pm_\kappa 1)\in \{\pm1\}^\kappa $, with $\pm_11=-1$. Let $B\in\Z_{\geq0}$ be such that $\frac \kappa 2+\frac12<\Re(v_1)< B+1$. Moreover, let $c_{v_2},\dots,c_{v_\kappa },c'_{v_2},\dots,c'_{v_\kappa }>0$ be such that
\est{
\Re(v_1)+c_{v_2}+\cdots +c_{v_\kappa }<B+1<\Re(v_1)+c'_{v_2}+\cdots+c'_{v_\kappa }.
}
Then
\est{
&(\pm_2x_2\pm_3\cdots\pm_\kappa x_{\kappa })^{v_1-1}\chi_{\R_>0}(\pm_2x_2\pm_3\cdots\pm_\kappa x_{\kappa })\\
&=\sum_{\substack {\nu=(\nu_2,\dots,\nu_\kappa )\in\Z_{\geq0}^{\kappa -1},\\\nu_2+\cdots+\nu_\kappa =B,\\\nu_{i}=0\text{ if $\pm_i=-1$}}}\frac{B!}{\nu_2!\cdots \nu_\kappa !}\frac1{(2\pi i)^{\kappa -1}}\Bigg(\int_{(c_{v_2},\dots,c_{v_\kappa })}-\int_{(c'_{v_2},\dots,c'_{v_\kappa })}\Bigg)\frac{\Psi_{\epsilon,B}(v_1,\dots,v_\kappa )}{x_2^{v_2-\nu_2}\cdots x_\kappa ^{v_\kappa -\nu_\kappa }}dv_2\cdots dv_\kappa ,\\[-1.5em]
}
where
\est{
\Psi_{\epsilon,B}(s_1,\dots,s_\kappa )
&:=
\frac{\Gamma(s_1)\cdots\Gamma(s_\kappa )}{\Gamma(V^+_{\epsilon}(s_1,\dots,s_\kappa ))\Gamma(V^-_{\epsilon}(s_1,\dots,s_\kappa ))}\frac{G(B+1-s_1-\cdots-s_\kappa )}{B+1-s_1-\cdots-s_\kappa },\\
V^\pm_{\epsilon}(s_1,\dots,s_\kappa )&:=\sum_{\substack {1\leq i\leq \kappa ,\\\pm_i1=\pm1}}s_{i}
}
and $G(s)$ is any entire function such that $G(0)=1$ and $ G(\sigma+it)\ll e^{-C_1|t|}(1+|\sigma|)^{C_2|\sigma|}$ for some fixed $C_1,C_2>0$. Moreover, writing $s_i=\sigma_i+it_i$ for $i=1,\dots,\kappa $, we have
\es{\label{fego}
\Psi_{\epsilon,B}(s_1,\dots,s_\kappa )&\ll \frac1{\delta^\kappa }\frac{(1+B+|\sigma_1|+\cdots+|\sigma_\kappa |)^{A(1+B+|\sigma_1|+\cdots+|\sigma_\kappa |)}}{(1+|t_1|)^{\frac12-\sigma_1}\cdots (1+|t_\kappa |)^{\frac12-\sigma_\kappa }(1+|t_1|+\cdots+|t_\kappa |)^{\sigma_1+\dots+\sigma_\kappa -1}},\\
}
for some $A>0$, provided that the $s_i$ are located at a distance greater than $\delta>0$ from the poles of $\Psi_{\epsilon,B}$.
\end{lemma}

\begin{remark}
If $\epsilon=(-1,\dots,-1)$, then $\Psi_{\epsilon,B}$ has to be interpreted as being identically zero.
\end{remark}
\begin{remark}\label{polpsi}
The function $\Psi_{\epsilon,B}(s_1,\dots,s_\kappa )$ has poles at  $s_i\in\Z_{\leq0}$ and at $s_1+\dots+s_k=B+1$.
\end{remark}
\begin{remark}
As a function $G$ in this case we take $G(s):={\xi(\frac 12+s)}/{\xi(\frac12)}$.
\end{remark}

We apply Lemma~\ref{sml} to~\eqref{tgrso} with $\epsilon:=(\sign a_1,\dots,\sign a_k)$, $v_1=1-\alpha_1-u_1+\frac w2$ and $B=4k$, so that by our choice for the lines of integration~\eqref{fcolio}, we have $\Re(v_1)= 1+4k-\eps$. 
Notice that thanks to~\eqref{fego} we don't have problems of convergence of the integrals.
We obtain,
\est{
K_{\balpha,\bbeta}(s)=\sum_{\substack{\nu=(\nu_2,\dots,\nu_k)\in\Z_{\geq0}^{k-1},\\\nu_2+\cdots+\nu_k=4k,\\\nu_{i}=0\text{ if $\pm_i=-1$}}}\frac{(4k)!}{\nu_2!\cdots \nu_k!} (R_{\nu;\balpha,\bbeta}(s)-R'_{\nu;\balpha,\bbeta}(s))
}
where, after opening the Ramanujan sum and  summing over $n_2,\dots,n_k$,
\est{
R_{\nu;\balpha,\bbeta}(s)
&:=\sum_{\ell\geq1}\sumstar_{h\mod \ell}\frac1{(2\pi i)^{k+1}}\int_{(c_w,c_{u_1},c_{v_2},\dots,c_{v_k})}\frac{ |a_1|^{-1+\alpha_1+u_1-\frac w2} N_1^{u_1-s}}{(\ell/(a_1,\ell))^{1-\alpha_1+\beta_1+w}}\tilde P(u_1-s) \times\\
&\quad\times\Psi_{\epsilon,4k}(1-\alpha_1-u_1+\tfrac w2,v_2,\dots,v_k)\zeta(1-\alpha_1+\beta_1+w)\frac{G_{\alpha_1,\beta_1}(w)}wdwdu_1\\
&\quad \times \prod_{i=2}^k\frac{|a_i|^{\nu_i-v_i}}{2\pi i}\int_{(c_{u_i})}D_{\alpha_i,\beta_i}(u_i+v_i-\nu_i,\tfrac {ha_i}{\ell})\tilde P(u_i-s)N_i^{u_i-s}du_idv_i\\
}
with $c_{v_2},\dots,c_{v_k}=\frac \eps k$, and where $R'_{\nu;\balpha,\bbeta}(s)$ is defined in the same way, but with lines of integrations $c_{v_i}'=1/2+\eps-\min(\Re(\alpha_i),\Re(\beta_i))$ for $i=2,\dots,k$.
We notice that by our choices for the lines of integration we have that the sum of the arguments of the function $\Psi_{\epsilon,4k}$ in $R_{\nu;\balpha,\bbeta}(s)$ has real part $4k+1-\frac\eps k$ and so is less than $4k+1$ as needed for the application of Lemma~\ref{sml}
 (whereas for $R_{\nu;\balpha,\bbeta}(s)$ one has that such real part is (much) larger than $4k+1$).

Now, $R'_{\nu;\balpha,\bbeta}$ can be bounded trivially by moving the line of integrations $c_{u_i}$ to $c_{u_i}=\frac12+\nu_i$ for $i= 2,\dots,k$, $c_w$ to $c_{w}=1+|\Re(\alpha_1-\beta_1)|+2\eps$, and $c_{u_1}$ to $
c_{u_1} =1-6k-\Re(\alpha_1)+\frac{c_w}2$ without passing through any pole (cf. Remark~\ref{polpsi}).
We obtain that $R'_{\nu;\balpha,\bbeta}(s)$ is bounded by
\est{
R'_{\nu;\balpha,\bbeta}(s)&\ll_\eps (\tfrac k\eps(1+|s|))^{Ak} \frac{N_2^{\frac12+\nu_2-\sigma}\cdots N_k^{\frac12+\nu_k-\sigma} |a_2|^{\nu_2-\frac12}\cdots |a_k|^{\nu_k-\frac12}}{N_1^{6k-1+\frac12|\Re(\alpha_1+\beta_1)|-\frac\eps2+\sigma}}\int_{(c_w,c_{u_1},c_{v_2},\dots,c_{v_k})}|G_{\alpha_1,\beta_1}(w)|\\
&
\quad\times  |a_1|^{-1+\Re(\alpha_1+u_1-\frac w2)} |\tilde P(u_1-s)\Psi_{\epsilon,4k}(1-\alpha_1-u_1+\tfrac w2,v_2,\dots,v_k)|dwdu_1dv_2\cdots dv_k|,\\
&\ll (1+|\Im(\alpha_1)|)^{6k}X_{s,k,\ba,\eps}^{Ak}\frac{N_2^{\frac12+\nu_2-\sigma}\cdots N_k^{\frac12+\nu_k-\sigma}}{N_1^{6k-1+\frac12|\Re(\alpha_1+\beta_1)|-\frac\eps2+\sigma}}\ll \frac{(1+|\Im(\alpha_1)|)^{6k}X_{s,k,\ba,\eps}^{Ak}N_1^{-k}}{(N_1\cdots N_k)^{\sigma-\frac12}}
}
by~\eqref{fnn} and~\eqref{fego} with $X_{s,k,\ba,\eps}$ as in~\eqref{dfxy},  and where for the last bound we used that
 $N_1$ is the largest of the $N_i$ and $\nu_2+\cdots+\nu_k=4k$. Thus, summarizing this section we proved
\es{\label{fs}
 A_{\balpha,\bbeta}(s)=\sum_{\substack{\{\alpha_1^*,\beta_1^*\}=\{\alpha_1,\beta_1\},\\ (\alpha_i^*,\beta_i^*)=(\alpha_i,\beta_i)\ \forall i\neq 1}} \sum_{\substack{\nu=(\nu_2,\dots,\nu_k)\in\Z_{\geq0}^{k-1},\\\nu_2+\cdots+\nu_k=4k,\\\nu_{i}=0\text{ if $\pm_i=-1$}}}\frac{(4k)!}{\nu_2!\cdots \nu_k!}  R_{\nu;\balpha^*,\bbeta^*}(s) +E_3(s),
}
where $E_3(s)$ satisfies
\es{\label{be1}
E_3(s)\ll Y_{\balpha,\bbeta}^{6k}X_{s,k,\ba,\eps}^{Ak}N_1^{-k}(N_1\cdots N_k)^{\frac12-\sigma}
}
with $Y_{\balpha,\bbeta}$ as in~\eqref{dfxy}.
\subsection{Picking up the poles of the Estermann functions}
Next, after moving $c_w$ and $c_{u_1}$ to ensure the convergence of the sum over $\ell$, we move the line of integration $c_{u_i}$ to $c_{u_i}=-\max(\Re(\alpha_i),\Re(\beta_i))-2\frac \eps k+\nu_i$  for each $i=2,\dots,k$ , passing through the poles (cf. Lemma~\ref{lest}) of the Estermann functions. We obtain:
\es{\label{gweo}
R_{\nu;\balpha,\bbeta}(s)=\sum_{\substack{I\cup J=\{2,\dots,k\},\\ I\cap J=\emptyset}}\sum_{\substack{\{\alpha_i^*,\beta_i^*\}=\{\alpha_i^*,\beta_i^*\}\ \forall i\in I,\\ (\alpha_j^*,\beta_j^*)=(\alpha_j,\beta_j)\ \forall j\in J\cup\{1\}}}\prod_{i\in I}\zeta(1-\alpha_i+\beta_i)S_{I;\nu;\balpha^*,\bbeta^*}(s),
}
where after changing the order of sums and integrals (as can by done by the absolute convergence of the sum and integrals)
\es{\label{dfso}
&S_{I;\nu;\balpha,\bbeta}(s):=\frac1{(2\pi i)^{k+1}}\int_{(c_w,c_{u_1},c_{v_2},\dots,c_{v_k})}\hspace{-3em}\Psi_{\epsilon,4k}(1-\alpha_1-u_1+\tfrac w2,v_2,\dots,v_k)\zeta(1-\alpha_1+\beta_1+w)\\
&\hspace{1.5em} \times \frac{G_{\alpha_1,\beta_1}(w)}w\tilde P(u_1-s)N_1^{u_1-s}du_1\Big(\prod_{i\in I}\tilde P(1-\alpha_i-v_i+\nu_i-s)N_i^{1-\alpha_i-v_i+\nu_i-s}dv_i\Big)\\
&\hspace{1.5em}\times |a_1|^{-1+\alpha_1+u_1-\frac w2}|a_2|^{\nu_2-v_2}\cdots |a_{k}|^{\nu_k-v_k}\Big(\prod_{j\in J}\frac1{2\pi i}\int_{(c_{u_j})}\tilde P(u_j-s)N_i^{u_j-s}\Big)\\
&\hspace{1.5em}\times\sum_{ \ell}\frac{(\ell,a_1)^{w }\prod_{i\in I\cup\{1\}}(a_i,\ell)^{1-\alpha_i+\beta_i}}{\ell^{w+\sum_{i\in I\cup\{1\}}(1-\alpha_i+\beta_i)}}dw\hspace{-0.4em}\sumstar_{h\mod \ell}\prod_{j\in J} D_{\alpha_j,\beta_j}(u_j+v_j-\nu_j,\tfrac {ha_j}{\ell})du_jdv_j\\
}
and the lines of integrations can be taken to be
\es{\label{cloi}
&c_{v_2}=\cdots =c_{v_k}=\tfrac\eps k,\qquad c_{u_j}=-\max(\Re(\alpha_j),\Re(\beta_j))-2\tfrac \eps k+\nu_j\ \forall j\in J, \\
& c_w=2\eps +\max\Big(0, 1+|J|-|I|+\sum_{i=1}^k |\Re(\alpha_i-\beta_i)|\Big),\\
&c_{u_1}=-4k-\Re(\alpha_1)+\tfrac {c_w}2+\eps.
}
Notice that with this choice the sum over $\ell$ converges absolutely by the convexity bound~\eqref{convbo} for the Estermann function.

We will treat $S_{I;\nu;\balpha,\bbeta}$ differently depending on whether $|I|\leq |J|$ or $|I|>|J|$.

\subsection{The case  $|I|\leq |J|$}
If $|I|\leq |J|$ (or, equivalently, $|J|\geq \frac{k-1}2$), then we use the following lemma, whose proof we postpone until the end of this subsection.
\begin{lemma}\label{mlb}
Let $\ba \in\Z_{\neq0}^\kappa$ and $\bgamma,\bdelta\in\C^\kappa$.
Let $\mathfrak S$ be the meromorphic function defined by
\est{
\mathfrak S(z):=\sumstar_{h\mod \ell}\prod_{j=1}^\kappa D_{\gamma_j,\delta_j}(z, \tfrac { ha_j}{\ell}). 
}
Then 
$$\mathfrak S(z)=\mathfrak S^*(z)+\mathfrak S^{**}(z)$$ 
where $\mathfrak S^{*}(z)$ is holomorphic on $\Re(z)\leq -\max_{i=1}^\kappa\max(\Re(\gamma_i),\Re(\delta_i))$ and  $\mathfrak S^{**}(z)=0$ if $\kappa=1$ and otherwise $\mathfrak S^{**}(z)$ is holomorphic on $\Re(z)\leq \frac1\kappa-\frac\eps \kappa-\frac1\kappa\sum_{i=1}^\kappa\max(\Re(\gamma_i),\Re(\delta_i)).$ Furthermore, if $|\Re(\gamma_i)|,|\Re(\delta_i)|\leq \frac12$ and  
\est{
\Re(z)\leq -\frac\eps {3\kappa}-\max(\Re(\gamma_i),\Re(\delta_i))\ \forall i=1,\dots,\kappa,\qquad u\geq 1+\kappa-2\kappa\Re(z)-\sum_{j=1}^\kappa\Re(\gamma_j+\delta_j)+\eps,
}
then
\es{\label{awem2}
\sum_{\ell\geq 1} \ell^{-u}|\mathfrak S^*(z)|\ll_\eps  |a_1\cdots a_\kappa|   \pr{\frac{A \kappa}\eps}^{4\kappa}(1+A|\Re(z)|)^{2\kappa |\Re(z)|}\prod_{i=1}^\kappa(1+|\gamma_i|+|\delta_i|+|\Im(z)|)^ {2-2\Re(z)}.
}
Moreover, if $|\Re(\gamma_i)|,|\Re(\delta_i)|\leq\frac1{2(\kappa-1)}$ and
\est{
\Re(z)\leq \frac1\kappa-\frac\eps \kappa-\frac1\kappa\sum_{i=1}^\kappa\max(\Re(\gamma_i),\Re(\delta_i)),\qquad \Re(u) \geq2+\kappa-2\kappa\Re(z)-\sum_{j=1}^\kappa\Re(\gamma_j+\delta_j)+\eps,
}
then
\es{\label{awem}
\sum_{\ell\geq 1} \ell^{-u}|\mathfrak S^{**}(z)|\ll_\eps  \pr{\frac{A \kappa}\eps}^{4\kappa}(1+A|\Re(z)|)^{2\kappa |\Re(z)|}\prod_{i=1}^\kappa(1+|\gamma_i|+|\delta_i|+|\Im(z)|)^ {2-2\Re(z)}.
}
\end{lemma}

We apply this Lemma with $\kappa=|J|$, splitting $S_{I;\nu;\balpha,\bbeta}(s)$ into
\est{
S_{I;\nu;\balpha,\bbeta}(s)=S^*_{I;\nu;\balpha,\bbeta}(s)+S^{**}_{I;\nu;\balpha,\bbeta}(s)
}
in the way suggested by the notation, with $S^{**}=0$ if $|J|=1$. For $S^{*}$ we use~\eqref{awem2} with 
\est{
z=0, \qquad \gamma_j=\alpha_j+u_j+v_j-\nu_j, \qquad \delta_j=\beta_j+u_j+v_j-\nu_j
} (since $D_{\alpha,\beta}(w,\frac hk)=D_{\alpha+w,\beta+w}(0,\frac hk)$). We move the line of integration $c_w$ to
\est{
c_w=2\eps + |J|-|I|+\sum_{i=1}^k |\Re(\alpha_i-\beta_i)|\
}
keeping the other ones as in~\eqref{cloi}. Notice that we stay in the region of holomorphicity of $\mathfrak S^*$ and that we can apply~\eqref{awem2} since $-\max(\Re(\gamma_i),\Re(\delta_i))=\frac\eps k\geq\frac{\eps}{3|J|}$ whereas, using the trivial bound $(a_j,\ell)\leq |a_j|$, the condition on $u$ in the Lemma becomes
$$\Re(w)+\sum_{i\in I\cup\{1\}}\Re(1-\alpha_i+\beta_i))\geq 1+|J|-\sum_{j\in J}\Re(2u_j+2v_j-2\nu_j+\alpha_j+\beta_j)+\eps$$
which is verified with our choice of lines of integration. Thus, we have that~\eqref{awem2} gives 
\est{
S^{*}&\ll_\eps   \pr{\frac{ k}\eps}^{Ak} |a_2|^{2+\nu_2}\cdots |a_{k}|^{2+\nu_k}\int_{(c_w,c_{u_1},c_{v_2},\dots,c_{v_k})} |\Psi_{\epsilon,4k}(\tfrac12-\alpha_1-u_1+\tfrac w2,v_2,\dots,v_k)|\\
&\times  |G_{\alpha_1,\beta_1}(w)\tilde P(u_1-s)N_1^{u_1-s}||dwdu_1|\Big(\prod_{i\in I}\big|\tilde P(1-\alpha_i-v_i+\nu_i-s)N_i^{1-\alpha_i-v_i+\nu_i-s}\big||dv_i|\Big)\\
&\times |a_1|^{-1+\Re(\alpha_1+u_1+\frac w2)}\prod_{j\in J}\int_{(c_{u_j})}|\tilde P(u_j-s)N_j^{u_j-s}|(1+|\alpha_j|+|\beta_j|+|u_j-\nu_j+v_j|)^ {2+\frac{2+\eps}{k}} |du_jdv_j|.
}
%
Using~\eqref{fnn} (with $r=5$ and $r=5k$),~\eqref{bndG} (with $r=5k$) and~\eqref{fego}, we obtain
\est{
S^{*}&\ll_\eps  (\tfrac{ k }\eps \max_{i=1}^k|a_i|)^{Ak}    (N_1\cdots N_k)^{-\sigma}
N_1^{-4k+\frac12(|J|-|I|)+\eta_{\balpha,\bbeta}+2\eps}\Big(\prod_{j\in J} N_j^{\nu_j}\Big)\prod_{i\in I} N_i^{1+\nu_i}\\
&\times
\int_{(\cdots)} \frac{(1+|\alpha_1+u_1-\frac w2|)^{4k-\frac12}(1+|s+u_1|)^{-5k}\prod_{i\in I}(1+|s+v_i+\alpha_i|)^{-5}}
{(1+|w|)^{5k}(1+|v_2|)^{\frac12-\frac\eps k}\cdots (1+|v_k|)^{\frac12-\frac\eps k}(1+|\alpha_1+u_1-\frac w2|+|v_2|+\cdots+|v_k|)^{4k-1}}\\
&\hspace{0em}\times \prod_{j\in J}\frac{(1+|\alpha_j|+|\beta_j|+|u_j-\nu_j+v_j|)^ {3}}{(1+|s+u_j|)^{5} }|dv_2\cdots dv_{k}dw| \prod_{j\in J\cup\{1\}}|du_j|\\
&\ll_\eps  X_{s,k,\ba,\eps}^{Ak} Y_{\balpha,\bbeta}^{7k}   (N_1\cdots N_k)^{-\sigma}
N_1^{-4k+\frac12(|J|-|I|)+\eta_{\balpha,\bbeta}+2\eps}\Big(\prod_{j\in J} N_j^{\nu_j}\Big)
\prod_{i\in I} N_i^{1+\nu_i},\\
}
where $\eta_{\balpha,\bbeta}$ is as in~\eqref{xi} and we used 
$$\frac12\sum_{i=1}^k |\Re(\alpha_i-\beta_i)|-\sum_{i\in I}\min(\Re(\alpha_i),\Re(\beta_i))-\sum_{j\in J}\max(\Re(\alpha_j),\Re(\beta_j))\leq 2\eta_{\balpha,\bbeta}.$$
 Now, since $|J|=k-1-|I|$, the above bound implies 
\est{
S^{*}&\ll_\eps   X_{s,k,\ba,\eps}^{Ak} Y_{\balpha,\bbeta}^{7k}    \frac{
N_1^{\frac{k-1}2+\eta_{\balpha,\bbeta}+2\eps}}{ (N_1\cdots N_k)^{\sigma}}\bigg(\prod_j \pr{\frac {N_j}{N_1}}^{\nu_j}\bigg)\prod_i \pr{\frac {N_i}{N_1}}^{1+\nu_i}
\ll_\eps   X_{s,k,\ba,\eps}^{Ak} Y_{\balpha,\bbeta}^{7k}    \frac{
N_1^{\frac{k-1}2+\eta_{\balpha,\bbeta}+2\eps}}{ (N_1\cdots N_k)^{\sigma}},
}
since $\nu_2+\cdots+\nu_k=4k$ and $N_i\leq N_1$ for all $i=2,\dots,k$. 

If $|I|<|J|$ we can bound $S^{**}$ in the same way, using~\eqref{awem} instead of~\eqref{awem2}, with the difference that now we move the line of integrations to
\es{\label{dsdss}
&c_{u_j}=\nu_j+\frac1{|J|}-2\frac\eps {k}-\frac1{|J|}\sum_{r \in J}\max(\Re(\alpha_r),\Re(\beta_r)),\hspace{0.6em} \forall j\in J,\qquad c_{v_2}=\cdots =c_{v_k}=\tfrac\eps k,\\
&c_w=-1+|J|-|I|+\sum_{i=1}^k|\Re(\alpha_j+\beta_j)|+3\eps,\qquad c_{u_1}=-4k-\Re(\alpha_1)+\tfrac {c_w}2+ \eps.
}
(Notice that, with respect to the case of $S^*$, we have essentially moved $c_{u_j}$ to the right by $\frac1{|J|}$,  $c_w$ to the left by $1$ and thus $c_{u_1}$ to the right by $\frac12$). Thus, we obtain
\est{
S^{**}&\ll_\eps   X_{s,k,\ba,\eps}^{Ak} Y_{\balpha,\bbeta}^{7k}   \frac{
N_1^{\frac{k}2-1+\eta_{\balpha,\bbeta}+2\eps}}{ (N_1\cdots N_k)^{\sigma}}\bigg(\prod_{j\in J} \pr{\frac {N_j}{N_1}}^{\nu_j}\bigg)\bigg(\prod_{i\in I} \pr{\frac {N_i}{N_1}}^{1+\nu_i}\bigg)\prod_{j\in J}N_j^\frac1{|J|}.
}
Also, we have
\est{
\Big(\prod_{i\in I}\frac {N_i}{N_1}\Big)\prod_{j\in J}N_j^\frac1{|J|}=N_1^{-\frac{|I|+1}{|J|}}\prod_{i\in I}\pr{\frac {N_i}{N_1}}^{1-\frac1{|J|}}\prod_{i=1}^kN_i^\frac1{|J|}\leq N_1^{-\frac{|I|+1}{|J|}} \prod_{i=1}^kN_i^\frac1{|J|}= N_1^{1-\frac{k}{|J|}} \prod_{i=1}^kN_i^\frac1{|J|},
}
since $|I|=k-1-|J|$, and the maximum value of the expression on the right is obtained for $|J|=k-1$ since $N_1^{k}\geq N_1\cdots N_k$. Thus, 
\est{
S^{**}&\ll_\eps   X_{s,k,\ba,\eps}^{Ak} Y_{\balpha,\bbeta}^{7k}   
N_1^{\frac{k}2-1-\frac1{k-1}+\eta_{\balpha,\bbeta}+2\eps} (N_1\cdots N_k)^{\frac1{k-1}-\sigma}.
}
If $|I|=|J|$, then we cannot move the line of integration $c_w$ as in~\eqref{dsdss} without passing through the pole at $w=0$. Thus, we move $c_w$ to $c_w=\sum_{i=1}^k|\Re(\alpha_j+\beta_j)|+2\eps$ and leave the other lines of integrations as in~\eqref{dsdss}. Bounding trivially we obtain
\est{
S^{**}&\ll_\eps   X_{s,k,\ba,\eps}^{Ak} Y_{\balpha,\bbeta}^{7k}    
N_1^{\eta_{\balpha,\bbeta}-\sigma+2\eps} \prod_{i=2}^k \pr{\frac {N_i}{N_1}}^{\nu_i}\prod_{i\in I} N_i^{1-\sigma}\prod_{j\in J}N_j^{\frac1{|J|}-\sigma}.
}
Now we have $\frac1{|J|}-\sigma\leq 0$ for $\sigma\geq\frac12$ (we can take $|J|>1$ since otherwise $ S^{**}=0$) and $N_1\geq N_i$ for all $i=2,\dots,k$ and so for $\frac12\leq\sigma\leq 1$ we have
\est{
S^{**}&\ll_\eps   
X_{s,k,\ba,\eps}^{Ak} Y_{\balpha,\bbeta}^{7k}    N_1^{\eta_{\balpha,\bbeta}-\sigma+2\eps+|I|(1-\sigma)}=
X_{s,k,\ba,\eps}^{Ak} Y_{\balpha,\bbeta}^{7k}    N_1^{\frac{k-1}2-\frac{k+1}{2}\sigma+\eta_{\balpha,\bbeta}+2\eps}.
}
since in this case $|I|=\frac{k-1}{2}$.
Thus, summarizing, in this subsection we proved that if $|J|\geq |I|$ and $\frac12\leq \sigma\leq 1$, then
\es{\label{is}
S_{I;\nu;\balpha,\bbeta}(s)&\ll_\eps X_{s,k,\ba,\eps}^{Ak} Y_{\balpha,\bbeta}^{7k} N_1^{-1+\eta_{\balpha,\bbeta}+2\eps} \Big(  \frac{N_1^{\frac{k+1}2}+N_1^{\frac{k}2}(N_2\cdots N_k)^{\frac1{k-1}}}{ (N_1\cdots N_k)^{\sigma}}+N_1^{\frac{k+1}2(1-\sigma)}\Big).
}

We conclude the subsection with the proof of Lemma~\ref{mlb}.
\begin{proof}[Proof of Lemma~\ref{mlb}]
First, we write $\ell_j=\ell/(a_j,\ell)$ and $a_j'=a_j/(a_j,\ell)$, so that 
\est{
\mathfrak S(z):=\sumstar_{h\mod \ell}\prod_{j=1}^\kappa D_{\gamma_j,\delta_j}(z, \tfrac { ha_j'}{\ell_j}). 
}
We apply the functional equation~\eqref{ffeff} to each of the Estermann functions, obtaining
\est{
\mathfrak S(z)=\sum_{\eta=(\eta_1,\dots,\eta_\kappa)\in\{\pm 1\}^{\kappa}}\ell_j^{1-2z-\gamma_j-\delta_j}\chi_{\eta_j}(z;\gamma_j,\delta_j) \sumstar_{h\mod \ell}D_{-\gamma_j,-\delta_j}(1-z, \eta_j\tfrac {\overline {h a_j'}}{\ell_j}),
}
where $\overline {ha_j'}$ is the inverse of $ha_j \mod {\ell_j}$. Now we assume  $\Re(z)< -\max_{i=1}^k\max(\Re(\gamma_i),\Re(\delta_i))$ and we expand the Estermann functions into their Dirichlet series and execute the sum over $h$. We obtain
\est{
\mathfrak S(z)&=\sum_{\eta\in\{\pm 1\}^{\kappa}}\ell_j^{1-2z-\gamma_j-\delta_j}\chi_{\eta_j}(z;\gamma_j,\delta_j)  \\
&\quad\times\sum_{\substack{m_1,\dots m_\kappa}} \frac{\tau_{-\gamma_1,-\delta_1}(m_1)\cdots \tau_{-\gamma_\kappa,-\delta_\kappa}(m_\kappa)}{m_1^{1-z}\cdots m_\kappa^{1-z}}c_\ell(\eta_1\overline a_1' (a_1,\ell)m_1+\dots+\eta_\kappa \overline a_\kappa' (a_\kappa,\ell)m_\kappa)\\
&=\sum_{\eta\in\{\pm 1\}^{\kappa}}\ell_j^{1-2z-\gamma_j-\delta_j}\chi_{\eta_j}(z;\gamma_j,\delta_j) \sum_{\substack{m_1,\dots m_\kappa}} \frac{\tau_{-\gamma_1,-\delta_1}(m_1)\cdots \tau_{-\gamma_\kappa,-\delta_\kappa}(m_\kappa)}{m_1^{1-z}\cdots m_\kappa^{1-z}}c_\ell(\rho)\\
}
where 
\est{
\rho:=\eta_1 (a_1,\ell)m_1\frac{a_1'\cdots a_\kappa'}{ a_1'}+\dots+\eta_\kappa (a_\kappa,\ell)m_1\frac{a_1'\cdots a_\kappa'}{ a_\kappa'}. 
}
Then we divide $\mathfrak S(z)$ into  $\mathfrak S(z)=\mathfrak S^{*}(z)+\mathfrak S^{**}(z)$  according to whether $\rho\neq 0$ or $\rho=0$ (notice that $\mathfrak S^{**}(z)=0$ if $\kappa=1$). For the terms with $\rho=0$, we observe that by Lemma~\ref{absc}  we have
\est{
\sum_{\substack{m_1,\dots, m_\kappa\geq1\\ \rho=0}} \frac{\tau_{-\gamma_1,-\delta_1}(m_1)\cdots \tau_{-\gamma_\kappa,-\delta_\kappa}(m_\kappa)}{m_1^{1-z}\cdots m_\kappa^{1-z}}\ll_\eps  \pr{\frac{A \kappa}\eps}^{4\kappa}
}
for 
$$
\Re(z)\leq \frac1\kappa-\frac\eps \kappa-\frac1\kappa\sum_{i=1}^\kappa\max(\Re(\gamma_i),\Re(\delta_i)).
$$
Thus from Stirling's formula in the crude form
\est{
\Gamma(\Re(z)+it)&\ll c^{-1} (1+A|\Re(z)|)^{|\Re(z)|}(1+|\Im(z)|)^{\Re(z)-\frac12}e^{-\frac{\pi}{2}|\Im(z)|},\qquad \Re(z)\geq c>0,
} 
and since $\ell_j\leq \ell$, for $\Im (\gamma_j),\Im(\delta_j)\ll1$ and $|\Re(\gamma_i)|,|\Re(\delta_i)|\leq \frac12$ we have 
\est{
\mathfrak S^{**}(z)&\ll_\eps   \ell^{\kappa+1-2\kappa\Re(z)-\sum_{j\in J}\Re(\gamma_j+\delta_j)}\pr{\frac{A \kappa}\eps}^{4\kappa}\sum_{\eta\in\{\pm 1\}^{\kappa}}|\chi_{\eta_j}(z;\gamma_j,\delta_j)|\\
&\ll_\eps   \ell^{\kappa+1-2\kappa\Re(z)-\sum_{j\in J}\Re(\gamma_j+\delta_j)}\pr{\frac{A \kappa}\eps}^{4\kappa}(1+A|\Re(z)|)^{2\kappa |\Re(z)|}\prod_{j=1}^\kappa(1+|\Im(z)|)^ {1-2\Re(z)-\Re(\gamma_j+\delta_j)}
}
and~\eqref{awem} follows.

In order to prove~\eqref{awem2} it is sufficient to show that for all fixed $\eps>0$ we have
\es{\label{eewe}
Z:=\sum_{\ell}\ell^{-u'}\Big|\sum_{\substack{m_1,\dots, m_\kappa,\\ \rho\neq0}} \frac{\tau_{-\gamma_1,-\delta_1}(m_1)\cdots \tau_{-\gamma_\kappa,-\delta_\kappa}(m_\kappa)}{m_1^{1-z}\cdots m_\kappa^{1-z}}c_\ell(\rho)\Big|\ll |a_1\cdots a_\kappa|  (A\kappa/\eps)^{4\kappa}
}
for $\Re(z)\leq -\frac\eps {3\kappa}-\max_{i}(\max(\gamma_i,\delta_i))$ $u'\geq1+\eps$. 
Moreover, we observe that we can assume $z,\gamma_i,\delta_i$ are real and we can also assume $z=-\frac\eps {3\kappa}-\max_{i}(\max(\gamma_i,\delta_i))$.

Since $|c_\ell(\rho)|\leq \sum_{d|(\rho,\ell)}d\ll_\eps (\rho,\ell)^{1+\eps}$, we have
\est{
Z&\ll_\eps
\sum_{\substack{m_1,\dots, m_\kappa}} \frac{\tau_{-\gamma_1,-\delta_1}(m_1)\cdots \tau_{-\gamma_\kappa,-\delta_\kappa}(m_\kappa)}{m_1^{1-z}\cdots m_\kappa^{1-z}}\sum_{\substack{\ell\geq1,\\\rho\neq0}}\frac{(\ell,\rho)^{1+\eps}}{\ell^u}.\\
}
(Notice that $\rho$ depends on $\ell$.) Now, we write 
\est{
\rho':=\eta_1 d_1m_1\frac{a_1\cdots a_\kappa d_1}{ a_1 d_1\cdots d_\kappa}+\dots+\eta_\kappa d_km_1\frac{a_1\cdots a_\kappa d_\kappa}{ a_\kappa d_1\cdots d_\kappa}
}
for $d_i=(a_i,\ell)$, so that
\est{
\sum_{\substack{\ell\geq1,\\\rho\neq0}}\frac{(\ell,\rho)^{1+\eps}}{\ell^u}=
\sum_{\substack{d_1|a_1,\dots,d_\kappa|a_\kappa,\\\rho'\neq0}}\sum_{\substack{\ell,\\ d_i|\ell,\ (\ell,a_i/d_i)=1}}\frac{(\ell,\rho')^{1+\eps}}{\ell^u}\ll_\eps \frac1{u-1} \sum_{\substack{d_1|a_1,\dots,d_\kappa|a_\kappa,\\\rho'\neq0}} \sigma_{1-u+\eps}(\rho'),
}
where $\sigma_{\alpha}(n):=\sum_{d|n}d^{\alpha}$, since for all $c\in\N$, $u>1$ we have
\est{
\sum_{\ell\geq1}\frac{(\ell,c)^{1+\eps}}{\ell^u}=\sum_{d|c}d^{1+\eps}\sum_{\substack{d|\ell,\\ (c/d,\ell)=1}}\frac{1}{\ell^u}\ll\zeta(1+u)\sum_{d|c}d^{1-u+\eps}\ll_\eps \frac1{u-1}\sigma_{1-u+\eps}(c).
}
Thus, for $u\geq1+\eps$
\es{\label{dadsa}
Z&\ll_\eps  \frac1{u-1} \sum_{\substack{d_1|a_1,\dots,d_\kappa|a_\kappa}} 
\sum_{\substack{m_1,\dots, m_\kappa\geq1\\\rho'\neq0}} \frac{\tau_{-\gamma_1,-\delta_1}(m_1)\cdots \tau_{-\gamma_\kappa,-\delta_\kappa}(m_\kappa)}{m_1^{1-z}\cdots m_\kappa^{1-z}}d(\rho').\\
}
We divide the sum into a sum of $k$ sums, according to which of the $m_i$ is the largest. For the contribution where $m_1\geq \max_{i=2}^\kappa(m_i)$ we observe that, for $z_1:=1-z-\max(\gamma_1,\delta_1)>1$ we have
\est{
\sum_{\substack{m_1\geq \max_{i=2}^\kappa(m_i),\\ \rho'\neq0}}\frac{d(\rho')\tau_{-\gamma_1,-\delta_1}(m_1)}{m_1^{1-z}}&\ll \sum_{\substack{m_1\geq \max_{i=2}^\kappa(m_i),\\ \rho'\neq0}}\frac{d(\rho')d(m_1)}{m_1^{z_1}}\\
&\ll \Big(\sum_{\substack{m_1\geq \max_{i=2}^\kappa(m_i),\\ \rho'\neq0}}\frac{d(\rho')^2}{m_1^{z_1}}\Big)^\frac12 \Big(\sum_{\substack{m_1\geq \max_{i=2}^\kappa(m_i),\\ \rho'\neq0}}\frac{d(m_1)^2}{m_1^{z_1}}\Big)^{\frac12},\\
}
by the Cauchy-Schwarz inequality. Then, since $|\rho'|\leq \kappa m_1 |a_1\cdots a_\kappa|$ we have that the above sum is
\est{
&\leq \Big( |\kappa a_1\cdots a_\kappa|^{z_1}\sum_{\substack{\rho'\in\Z_{\neq0}}}\frac{d(\rho')^2}{{|\rho'|}^{z_1}}\Big)^\frac12 \Big(\sum_{m_1\geq 1}\frac{d(m_1)^2}{m_1^{z_1}}\Big)^{\frac12}\leq | \kappa m_1 a_1\cdots a_\kappa|^\frac12\Big(2\sum_{m_1\geq 1}\frac{d(m_1)^2}{m_1^{z_1}}\Big)\\
&\ll |\kappa m_1 a_1\cdots a_\kappa|^\frac12(z_1-1)^{-4}
}
since $z_1\leq2$ and $\sum_{\substack{n\geq1}}\frac{d(n)^2}{n^{2s}}=\frac{\zeta(2s)^4}{\zeta(4s)}$. 
Thus,~\eqref{dadsa} can be bounded by
\est{
Z\ll_\eps \frac{\kappa ^2}{(u-1)} |a_1\cdots a_\kappa|\prod_{i=1}^k(-z-\max(\gamma_i,\delta_i))^{-4}
}
and so~\eqref{eewe} follows.
\end{proof}

\subsection{The case $|I|>|J|$}
Here we treat $S_{I;\nu;\balpha,\bbeta}(s)$, which was defined in~\eqref{dfso}, in the case $|I|>|J|$. First, we move the lines of integration $c_{w}$ and $c_{u_1}$ in~\eqref{dfso} to
\est{
&c_{w}=2\eps+|\Re(\alpha_1-\beta_1)|, \qquad c_{u_1}=-4k-\Re(\alpha_1)+\tfrac {c_w}2+ \eps \\
&c_{v_2}=\cdots =c_{v_k}=\tfrac\eps k,\qquad c_{u_j}=-\max(\Re(\alpha_j),\Re(\beta_j))-2\tfrac \eps k+\nu_j\ \forall j\in J. \\
}
Bounding as before (in this case one could also simply use the convexity bound for the Estermann functions) one obtains
\est{
S_{I;\nu;\balpha,\bbeta}(s)
&\ll 
X_{s,k,\ba,\eps}^{Ak} Y_{\balpha,\bbeta}^{7k}  N_1^{-1+\eta_{\balpha,\bbeta}+2\eps}(N_1\cdots N_k)^{1-\sigma}\bigg(\prod_{i\in I}\pr{\frac{N_i}{N_1}}^{\nu_i}\bigg){\prod_{j\in J}\frac{N_j^{\nu_j-1}}{N_1^{\nu_j}}}.\\
}
In particular, if $\nu_j>0$ for some $j\in J$, then the above bound implies
\est{
S_{I;\nu;\balpha,\bbeta}(s)
&\ll 
X_{s,k,\ba,\eps}^{Ak} Y_{\balpha,\bbeta}^{7k}  N_1^{-2+\eta_{\balpha,\bbeta}+2\eps}(N_1\cdots N_k)^{1-\sigma}.
}

For the terms with $\nu_j=0$ for all $j\in J$, we take $j^*\in J$ and move the lines of integrations $c_{u_{j^*}}$ and $c_{v_{j^*}}$ to the left and to the right by $1$ respectively, 
picking up the residue from the simple pole of $\Psi_{\epsilon,4k}$  at $v_{j^*}=0$. 
In the contribution coming from the integrals of the new lines of integration we move $c_{u_1}$ to the right by $1$ (as we can now do without passing through other poles of $\Psi_{\epsilon,4k}$) so that bounding trivially we obtain a contribution which is
\est{
&\ll 
X_{s,k,\ba,\eps}^{Ak} Y_{\balpha,\bbeta}^{7k}  N_1^{-2+\eta_{\balpha,\bbeta}+2\eps}(N_1\cdots N_k)^{1-\sigma}\bigg(\prod_{i\in I}\pr{\frac{N_i}{N_1}}^{\nu_i}\bigg)N_{j^*}\prod_{j\in J}N_j^{-1}\\
&\ll 
X_{s,k,\ba,\eps}^{Ak} Y_{\balpha,\bbeta}^{7k}  N_1^{-2+\eta_{\balpha,\bbeta}+2\eps}(N_1\cdots N_k)^{1-\sigma}.
}
We repeat this for all $j\in J$ obtaining 
\es{\label{stt}
S_{I;\nu;\balpha,\bbeta}(s)
&=T_{I;\nu;\balpha,\bbeta}(s)+O\bigg(
X_{s,k,\ba,\eps}^{Ak} Y_{\balpha,\bbeta}^{7k}  \frac{(N_1\cdots N_k)^{1-\sigma}}{N_1^{2-\eta_{\balpha,\bbeta}-2\eps}}\bigg)
}
where $T_{I;\nu;\balpha,\bbeta}(s)$ is obtained by $S_{I;\nu;\balpha,\bbeta}(s)$ by taking the residue in at $v_j=0$ for all $j\in J$, that is
\est{
&T_{I;\nu;\balpha,\bbeta}(s):=\frac1{(2\pi i)^{k+1}}\int_{(c_w,c_{u_1})}\Big(\prod_{i\in I}\int_{(c_{v_i})}\Big)\Psi'_{I_1,\epsilon}(\bv_{I_1})\frac{\tilde P(u_1-s)N_1^{u_1-s}}{|a_1|^{1-\alpha_1-u_1+\frac w2}}\\
&\qquad\times  \frac{G(B+\alpha_1+u_1-\tfrac w2-\sum_{i\in I}v_i  )}{B+\alpha_1+u_1-\tfrac w2-\sum_{i\in I}v_i  }\sum_{ \ell\geq1}\frac{\prod_{i\in I_1}(a_i,\ell)^{1-\alpha_i+\beta_i}}{\ell^{w+\sum_{i\in I_1}(1-\alpha_i+\beta_i)}(\ell,a_1)^{-w }}\frac{G_{\alpha_1,\beta_1}(w)}wdu_1\\
&\qquad \times\zeta(1-\alpha_1+\beta_1+w) \Big(\prod_{i\in I}\tilde P(1-\alpha_i-v_i+\nu_i-s)N_i^{1-\alpha_i-v_i+\nu_i-s}|a_i|^{\nu_i-v_i}dv_i\Big)dw\\
&\qquad \times \sumstar_{h\mod \ell}\prod_{j\in J}\int_{(c_{u_j})}\tilde P(u_j-s)N_i^{u_j-s} D_{\alpha_j,\beta_j}(u_j,\tfrac {ha_j}{\ell})du_j.\\
}
with $I_1:=I\cup\{1\}$, $\bv_I:=(v_i)_{i\in I_1}$, where we put $v_1:=1-\alpha_1-u_1+\tfrac w2$ and
\es{\label{dfpsp}
\Psi_{I_1,\epsilon}(\bv_{I_1})&:=\frac{\prod_{i\in I_1}\Gamma(v_i)}{\Gamma(V^+_{I;\epsilon}(\bv_{I_1}))\Gamma(V^-_{I_1;\epsilon}(\bv_{I_1}))},\qquad V^\pm_{I_1;\epsilon}(\bv_{I_1}):=\sum_{\substack{i\in I_1,\\\pm_i1=\pm1}}v_{i}.
}
Notice that in the integral defining $T_{I;\nu;\balpha,\bbeta}(s)$ we have a fast decaying function for each of the variables of integration and so we don't have to worry anymore about the convergence of the integrals.

Next, we move the line of integrations to
\est{
c_{u_j}=1-\max(\alpha_j,\beta_j)-\tfrac\eps k\quad \forall j\in J,
\qquad
c_w=-|I|+\sum_{i=1}^k|\alpha_j-\beta_j|+2\eps
}
moving also $c_{u_1}$ so that we still have $c_{u_1}=-4k-\Re(\alpha_1)+\tfrac {c_w}2+ \eps $. We pass through a simple pole at $w=0$ only, since the pole of $\zeta(1-\alpha_1+\beta_1+w)$ is canceled by the zero of $G_{\alpha_1,\beta_1}(w)$ (cf. Remark~\ref{rmdfG}).
Notice that on the new line of integrations the convexity bound~\eqref{convbo} gives
\est{
D_{\alpha_j,\beta_j}(u_j,\tfrac {ha_j}{\ell})\ll (k/\eps)^{2}(\ell(|u_j|+|\alpha_j|)^\frac12(|u_j|+|\beta_j|)^\frac12)^{|\Re(\alpha_j-\beta_j)|+\frac\eps k}
}
which suffices for the convergence of the sum over $\ell$. Thus the integral on the new lines of integrations gives a contribution bounded by
\est{
&X_{s,k,\ba,\eps}^{Ak} Y_{\balpha,\bbeta}^{7k}  N_1^{-|I|+\eta_{\balpha,\bbeta}+2\eps} (N_1\cdots N_k)^{1-\sigma}.\\
}
Thus, since $|I|>|J|$ with $|I|+|J|=k-1\geq 2$ implies $|I|\geq2$, then by~\eqref{stt} we have
\es{\label{fbs}
S_{I;\nu;\balpha,\bbeta}(s)&=Y_{I;\nu;\balpha,\bbeta}(s)+O\Big(
X_{s,k,\ba,\eps}^{Ak} Y_{\balpha,\bbeta}^{7k} {(N_1\cdots N_k)^{1-\sigma}}{ N_1^{-2+\eta_{\balpha,\bbeta}+2\eps}}\Big)
}
where $Y_{I;\nu;\balpha,\bbeta}(s)$ is the contribution from the residue at $w=0$.

Now, we move the line of integration $c_{u_1}$ to $c_{u_1}=1-\Re(\alpha_1)+\frac\eps k$ and we make the change of variables $v_i\to v_i+\nu_i$ for all $i\in I$ moving the lines of integration $c_{v_i}$ so that we still have $c_{v_i}=\frac\eps{k}$ for all $i\in I$. Since $B=\sum_{i\in I}\nu_i$ (as we only have to consider the terms with $\nu_j=0$ for all $j\in J\cup\{1\}$) we have that the only factor depending on $\nu$ is the function $\Psi_{I_1,\epsilon}'(\bv_i+\bnu)$, where $\bnu=(0,\nu_2,\dots,\nu_k)$. Thus, summing over $\nu$ we are left with
\est{
&\sum_{\substack{\nu=(\nu_2,\dots,\nu_k)\in\Z_{\geq0}^{k-1},\\\nu_2+\cdots+\nu_k=4k,\\\nu_{i}=0\text{ if $\pm_i=-1$ or $i\in J$}}}\frac{(4k)!}{\nu_2!\cdots \nu_k!}\Psi_{I_1,\epsilon}'(\bv_i+\bnu)\\[-1.5em]
&\hspace{14em}=\frac{\prod_{\substack{i\in I^-_{1}}}\Gamma(v_i)}{\Gamma(V^-_{I_1;\epsilon}(\bv_{I_1}))}
\sum_{\substack{\nu_{I}=(\nu_i)_{i\in I_1^+},\,\nu_i\in\Z_{\geq0},\\\sum_{i\in I_1^+}\nu_i=4k}}\frac{(4k)!}{\prod_{i\in I_1^+}\nu_i!}\frac{\prod_{i\in I_1^+}\Gamma(v_i+\nu_i)}{\Gamma(\sum_{i\in I_1^+}(v_i+\nu_i))},
}
where $I_1^{\pm}:=\{i\in I\mid \pm_i1=\pm1\}$. The identity  $\B(s_1+1,s_2)+\B(s_1,s_2+1)=\B(s_1,s_2)$ for the Beta function $\B(s_1,s_2):=\Gamma(s_1)\Gamma(s_2)\Gamma(s_1+s_2)^{-1}$ generalizes to 
\es{\label{beid}
\sum_{\substack {(r_1,\dots,r_m)\in\Z_{\geq0}^m,\\ r_1+\cdots+r_m=r }}\frac{r!}{r_1!\cdots r_m!}\frac{\Gamma(s_1+r_1)\cdots\Gamma(s_m+r_m)}{\Gamma(r+s_1+\cdots+s_m)}=\frac{\Gamma(s_1)\cdots\Gamma(s_m)}{\Gamma(s_1+\cdots+s_m)},
}
for $m,r\geq1$, $s_1,\dots,s_m\in\C$, and so we have
\est{
&\sum_{\substack{\nu=(\nu_2,\dots,\nu_k)\in\Z_{\geq0}^{k-1},\\\nu_2+\cdots+\nu_k=4k,\\\nu_{i}=0\text{ if $\pm_i=-1$ or $i\in J$}}}\frac{(4k)!}{\nu_2!\cdots \nu_k!}\Psi_{I_1,\epsilon}'(\bv_i+\bnu)=\frac{\prod_{\substack{i\in I_{1}}}\Gamma(v_i)}{\Gamma(V^-_{I_1;\epsilon}(\bv_{I_1}))\Gamma(V^+_{I_1;\epsilon}(\bv_{I_1}))}=\Psi_{I_1,\epsilon}'(\bv_i).
}
Thus, after the change of variables $u_j\rightarrow u_j+s$ for all $j\in J\cup\{1\}$ and $1-\alpha_i-v_i-s\to u_i$ for $i\in I$, we obtain
\es{\label{fbs2}
\sum_{\substack{\nu=(\nu_2,\dots,\nu_k)\in\Z_{\geq0}^{k-1},\\\nu_2+\cdots+\nu_k=4k,\\\nu_{i}=0\text{ if $\pm_i=-1$ or $i\in J$}}}\frac{(4k)!}{\nu_2!\cdots \nu_k!} Y_{I;\nu;\balpha,\bbeta}(s) = \zeta(1-\alpha_1+\beta_1) Z_{ I_1;\balpha,\bbeta}(s)
}
where for any set ${\mathcal I}$ with $|\mathcal I|>\frac{k+1}{2}$  we define
\es{\label{dfnz}
Z_{\mathcal I;\balpha,\bbeta}(s)&= \frac{1}{(2\pi i)^{k}}\int_{(c_{u_1},\dots,c_{u_k})}
\Big(\prod_{i=1}^k\tilde P(u_i)N_i^{u_i}\Big)\Psi_{\mathcal I,\epsilon}(\bv_{\mathcal I}')\Big(\prod_{i\in \mathcal I}|a_i|^{-1+\alpha_i+u_i+s}\Big)\frac{G(1-\sum_{i\in I_1}v_i' )}{1-\sum_{i\in I}v_i'  }\\
&\quad \times \sum_{ \ell\geq1}\frac{\prod_{i\in \mathcal I}(a_i,\ell)^{1-\alpha_i+\beta_i}}{\ell^{\sum_{i\in \mathcal I}(1-\alpha_i+\beta_i)}}\sumstar_{h\mod \ell}\prod_{j\in \mathcal J} \Big( D_{\alpha_j,\beta_j}(u_j+s,\tfrac {ha_j}{\ell})\Big)du_1\cdots d u_k)
}
where $\bv'_{\mathcal I}=(v_i')_{i\in \mathcal I}=(1-\alpha_i-u_i-s)_{i\in \mathcal I}$, $\mathcal J:=\{1,\dots,k\}\setminus \mathcal I$ and where we can take the lines of integrations to be
\est{
c_{u_j}=1-\max(\alpha_j,\beta_j)-\sigma-\tfrac\eps k\quad \forall j\in \mathcal J,\quad
c_{u_i}=1-\alpha_i-\sigma-\tfrac\eps k\quad \forall i\in \mathcal I.
}
Notice that for any $i\in \mathcal I$ if we move $c_{u_i^*}$ to $-\alpha_{i^*}-\sigma+ \eps$ we obtain the bound
\es{\label{bfy1}
Z_{\mathcal I;\balpha,\bbeta}(s)\ll 
X_{s,k,\ba,\eps}^{Ak} Y_{\balpha,\bbeta}^{k} (N_1\cdots N_k)^{1-\sigma}
N_{i^*}^{-1+\eta_{\balpha,\bbeta}+\eps}.
}
Also, for all $j^*\in \mathcal J$ (notice we can assume $k\geq4$ since for $k=3$ we have $|\mathcal I|>2$ and so $\mathcal J=\emptyset$) we have
\es{\label{bfy2}
Z_{\mathcal I;\balpha,\bbeta}(s)\ll 
X_{s,k,\ba,\eps}^{Ak} Y_{\balpha,\bbeta}^{k}(N_1\cdots N_k)^{1-\sigma-\frac1k}
(N_{j^*}^{-1+\eta_{\balpha,\bbeta}+\eps})
}
as can be seen by moving the lines of integrations to
\est{
&c_{u_j}=1-\max(\alpha_j,\beta_j)-\sigma-\tfrac1{k}\quad \forall j\in J\setminus \{j^*\},
\qquad c_{u_i}=1-\alpha_i-\sigma-\tfrac1 k+\tfrac\eps k\quad \forall i\in \mathcal I,\\
&c_{u_{j^*}}=-\sigma-\tfrac 1k-\max(\alpha_{j^*},\beta_{j^*})+\eta_{\balpha,\bbeta}.\\
}
(Notice that since $|\mathcal I |\geq [\frac{k+3}2]$, with this choice the sum over $\ell$ is absolutely convergent by the convexity bound~\eqref{convbo}.)
%

To summarize, by~\eqref{fbs} and~\eqref{fbs2} in this section we proved that for $|I|>|J|$
\es{\label{il}
\sum_{\substack{\nu=(\nu_2,\dots,\nu_k)\in\Z_{\geq0}^{k-1},\\\nu_2+\cdots+\nu_k=4k,\\\nu_{i}=0\text{ if $\pm_i=-1$}}}\frac{(4k)!}{\nu_2!\cdots \nu_k!} S_{I;\nu;\balpha,\bbeta}(s) = \zeta(1-\alpha_1+\beta_1) Z_{ I_1;\balpha,\bbeta}(s)+E_4(s)
}
where $E_4(s)$ satisfies
\est{
E_4(s)\ll
X_{s,k,\ba,\eps}^{Ak} Y_{\balpha,\bbeta}^{7k} {(N_1\cdots N_k)^{1-\sigma}}{N_1^{-2+\eta_{\balpha,\bbeta}+2\eps}}
}
and $Z_{ I_1;\balpha,\bbeta}(s)$ satisfies~\eqref{bfy1} and~\eqref{bfy2}.

\subsection{Conclusion of the proof of Lemma~\ref{mmlle}}\label{cl}
By~\eqref{fs},~\eqref{be1},~\eqref{gweo},~\eqref{is} and~\eqref{il}
\es{\label{fsd}
 A_{\balpha,\bbeta}(s)=\sum_{\substack{\mathcal I\cup \mathcal J=\{1,\dots,k\},\\ \mathcal I\cap \mathcal J=\emptyset,\\ 1\in \mathcal I,\ |\mathcal I|>\frac{k+1}2}}\sum_{\substack{\{\alpha_i^*,\beta_i^*\}=\{\alpha_i^*,\beta_i^*\}\ \forall i\in \mathcal I,\\ (\alpha_j^*,\beta_j^*)=(\alpha_j,\beta_j)\ \forall j\in \mathcal J}}\prod_{i\in \mathcal I}\zeta(1-\alpha_i+\beta_i)Z_{\mathcal I;\balpha^*,\bbeta^*}(s),
 +E_5(s),
}
where $E_5(s)$ is an entire function of $s$ satisfying 
\est{
E_5(s)&\ll
X_{s,k,\ba,\eps}^{Ak} Y_{\balpha,\bbeta}^{7k}N_1^{\eta_{\balpha,\bbeta}+2\eps} \Bigg(\frac{N_1^{\frac{k-1}2}}{(N_1\cdots N_k)^{\sigma}}+\frac{N_1^{\frac{k}2-1-\frac1{k-1}}}{(N_1\cdots N_k)^{\sigma-\frac1{k-1}}}\\
&\hspace{11em}+\frac{ (N_1\cdots N_k)^{1-\frac1k-\sigma}}{N_1}+N_1^{\frac{k-1}2-\frac{k+1}{2}\sigma}\Bigg)
}
for $\frac12\leq \sigma\leq 1$. Thus, to obtain~\eqref{vsy} and~\eqref{bvsy} we just need to extend the sum over $\mathcal I$ in~\eqref{fsd} to include also the sets $\mathcal I$ which do not contain $1$ at a cost of an eligible error given by the bound~\eqref{bfy2} with $j^*=1$. Finally, we conclude by observing that the analyticity of $E_1(s)$ for $(s,\balpha,\bbeta)\in\C\times\Omega_k$ (where $\Omega_k$ is defined in~\eqref{dfok}) can be immediately verified from the definition of the various error terms.

\begin{remark}\label{fnlr}
Notice that the main term on the right of~\eqref{fsd} is symmetric in $N_1,\dots,N_k$ and, by the definition~\eqref{dfpsp} of $\Psi_{\mathcal I,\epsilon}$, with respect to $\ba\leftrightarrow -\ba$ (i.e. $\epsilon\leftrightarrow-\epsilon$).
\end{remark}

\section{Proof of Lemma~\ref{mmlle2}}\label{pmmlle2}
As in the proof of Lemma~\ref{mmlle} we assume $N_1$ is the maximum among $N_1,\dots,N_k$. 
Writing the partitions of unity in terms of their Mellin transform, we obtain
\es{\label{fdff}
 A_{\balpha,\bbeta}(s)=\int_{(c_{u_1},\cdots c_{u_k})}
\mathcal A_{\balpha',\bbeta'}(s') 
 \tilde P(u_1)\cdots \tilde P(u_k) N_1^{u_1}\cdots N_k^{u_k}\,du_1\cdots du_k
 }
where $s':=1+s-\sigma+\tfrac\eps k-\eta'_{\balpha,\bbeta}$ with $\eta'_{\balpha,\bbeta}:=\frac1k\sum_{i=1}^k\min(\Re(\alpha_i),\Re(\beta_i))$, $\balpha':=(-1+\alpha_i+u_i-\frac\eps k+\sigma+\eta'_{\balpha,\bbeta})_{i\in\{1,\dots,k\}}$ and $\bbeta':=(-1+\beta_i+u_i-\frac\eps k+\sigma+\eta'_{\balpha,\bbeta})_{i\in\{1,\dots,k\}}$. We move the lines of integration to 
\est{
&c_{u_i}=1-\sigma-\eta'_{\balpha,\bbeta}+\tfrac\eps k\quad\forall i\in\{2,\dots,k\},\qquad
c_{u_1}=-\sigma-\eta'_{\balpha,\bbeta}+\tfrac\eps k
}
so that on the new lines of integration we have $|\Re(\alpha_i')|,|\Re(\beta_i')|\leq \frac1{2(k-1)}$ for $i=2,\dots,k$ and
\est{
\Re(s')\geq1-\frac1k+\frac\eps k-\frac1k\sum_{i=1}^k\min(\Re(\alpha_i'),\Re(\beta_i')).
} 
We cannot apply directly Lemma~\ref{absc} since we would need also $|\Re(\alpha_1')|,|\Re(\beta_1')|\leq\frac1{2(k-1)}$, whereas we have $\Re(\alpha_1')=-1+\Re(\alpha_1)$, $\Re(\beta_1')=-1+\Re(\beta_1)$. However, we observe that
\est{
|\mathcal A_{\balpha,\bbeta}(s')|&\ll_\eps  \sum_{\substack{a_1n_1+\cdots+a_kn_k=0}}\Big|\frac{n_1^{-s_1'+\eps}d(n_2)\cdots d(n_k)}{n_2^{s'_2}\cdots n_{k}^{s'_k}} \Big|\\
&\leq \big(1+(k\max_{i=2}^k|a_i|)^{-s_1'+\eps}\big)  \sum_{\ell=2}^k\sum_{\substack{\max_{j=2}^kn_j=n_\ell}} \Big|\frac{d(n_2)\cdots d(n_k)}{n_\ell ^{s_1'-\eps}n_2^{s'_2}\cdots n_{k}^{s'_k}} \Big|
}
since $(\max_{j=2}^kn_j)\leq n_1\leq k(\max_{j=2}^kn_j)( \max_{i=2}^k|a_i|)$ and where $s_i'=1+\frac\eps k-\eta'_{\balpha,\bbeta}+\min(\Re(\alpha_i),\Re(\beta_i))$ for $i=2,\dots,k$ and $s_1'=\frac\eps k-\eta'_{\balpha,\bbeta}+\min(\Re(\alpha_1),\Re(\beta_1))$. Then, we can proceed exactly as in the proof of Lemma~\ref{absc} obtaining 
\est{
|\mathcal A_{\balpha,\bbeta}(s')|\ll \max_{i=2}^k|a_i|  \pr{\frac{A k}\eps}^{4k}.
}
Thus,~\eqref{fdff} gives
\est{
 A_{\epsilon;\balpha,\bbeta}(s)&\ll_\eps \max_{i=2}^k|a_i|  \pr{\frac{A k}\eps}^{4k} N_1^{-1+k\eta'_{\balpha,\bbeta}} (N_1\cdots N_k)^{1-\sigma}\ll_\eps X_{s,k,\ba,\eps}^{Ak} Y_{\balpha,\bbeta}^{k}(N_1\cdots N_k)^{1-\sigma}
N_{\max}^{-1+\eta_{\balpha,\bbeta}+\eps}
}
since $k\eta'_{\balpha,\bbeta}\leq \eta_{\balpha,\bbeta}$ and the Lemma follows since by~\eqref{bfy1} and~\eqref{bfy2} the main term on the right hand side of~\eqref{mmlle2e} also satisfies the bound~\eqref{tbb}.

\section{The completion of the proof of Theorem~\ref{mtmt}}\label{pcompl}
We define
\est{
E^*(s)=\sumdagger_{\substack{N_1,\cdots, N_k,\\ B_1<B_2}}E_1(s)+\sumdagger_{\substack{N_1,\cdots, N_k,\\ B_1> B_2}}E_2(s)
}
with $E_1(s)$ and $E_2(s)$ as in Lemma~\ref{mmlle} and~\ref{mmlle2},
where
\est{
B_1&:=\frac{N_{\max}^{\frac{k-1}2}}{(N_1\cdots N_k)^{\sigma}}+\frac{N_{\max}^{\frac{k}2-1-\frac1{k-1}}}{(N_1\cdots N_k)^{\sigma-\frac1{k-1}}}+\frac{ (N_1\cdots N_k)^{1-\frac1k-\sigma}}{N_{\max}}+N_{\max}^{\frac{k-1}2-\frac{k+1}{2}\sigma}
\\
B_2&:=(N_1\cdots N_k)^{1-\frac1k-\sigma}N_{\max}^{-1}
}

Since $E_1(s)$ and $E_2(s)$ are entire for all $N_1,\cdots, N_k$ we have that $E^*(s)$ is holomorphic in the half-plane where the above sums converge absolutely (and uniformly in $s$). Thus, writing $N_1\cdots N_k=:N_{\max}^x$ with $1\leq x\leq k$, we see that $E^*(s)$ is holomorphic provided that $\frac12\leq \sigma\leq 1$, $(\balpha,\bbeta)\in\Omega_k$ and
\est{
&\max_{1\leq x\leq k}\min(\max_{i=1,\dots,4}(L_i(x)),(1-\tfrac1k-\sigma)x-1)+\eta_{\balpha,\bbeta}+3\eps\leq 0
}
where 
\est{
&L_1(x):=\tfrac {k-1}2-\sigma x,\quad L_2(x):=\tfrac k2-1-\tfrac1{k-1}-(\sigma-\tfrac1{k-1})x\\
&L_3(x):=(1-\tfrac1k-\sigma)x-1,\quad
L_4(x):=\tfrac{k-1}2-\tfrac{k+1}{2}\sigma.
}
Equivalently, we need 
\est{
M_i:=\max_{1\leq x\leq k}\min(L_i(x),(1-\sigma)x-1)+\eta_{\balpha,\bbeta}+3\eps<0
}
for $i=1,\dots,4$. For $\frac12\leq \sigma\leq 1$ one has
\est{
M_1=M_2=M_4=(1-\sigma)\tfrac{k+1}2-1,\quad M_3=\max(k-2-k\sigma,-\tfrac1k-\sigma).
}
Thus, we need $ 1-\frac{2(1-\eta_{\balpha,\bbeta}-3\eps)}{k+1}\leq \sigma\leq1$ and $(\balpha,\bbeta)\in\Omega_k$. Also, in this strip we have the bound
\es{\label{bffe}
E^*(s)\ll_\eps X_{s,k,\ba,\eps}^{Ak} Y_{\balpha,\bbeta}^{7k}.
}

Finally, for any function $F(u)$ which is analytic and grows at most polynomially on a strip $|\Re(u)|<c$ for some $c>0$ we have
\est{
\sumdagger_N \frac1{2\pi i}\int_{(0)}F(u)\tilde P(u)N^{u} du=F(0).
}
Thus, by the definition~\eqref{dfnz} of $Z_{\mathcal I;\balpha,\bbeta}(s)$, for 
\es{\label{cds}
1-\frac1k+\frac\eps k-\frac1k\sum_{i=1}^k\min(\Re(\alpha_i),\Re(\beta_i))\leq \sigma< 1-\max_{1\leq i\leq k}\max(\Re(\alpha_i),\Re(\beta_i)),\hspace{0.5em} (\balpha,\bbeta)\in\Omega_k
}
we have
\est{
\sumdagger_{N_1,\dots,N_k}\sum_{\substack{\mathcal I\cup \mathcal J=\{1,\dots,k\},\\ \mathcal I\cap \mathcal J=\emptyset,\\ |\mathcal I|>\frac{k+1}{2}}}\sum_{\substack{\{\alpha_i^*,\beta_i^*\}=\{\alpha_i^*,\beta_i^*\}\ \forall i\in \mathcal I,\\ (\alpha_j^*,\beta_j^*)=(\alpha_j,\beta_j)\ \forall j\in \mathcal J}}\Big(\prod_{i\in \mathcal I}\zeta(1-\alpha_i+\beta_i)\Big)Z_{\mathcal I;\balpha^*,\bbeta^*}(s)&=\mathcal W_{\balpha,\bbeta}(s),
}
where $\mathcal W_{\balpha,\bbeta}(s)$ is defined as
\est{
\mathcal W_{\balpha,\bbeta}(s)&:=\sum_{\substack{\mathcal I\cup \mathcal J=\{1,\dots,k\},\\ \mathcal I\cap \mathcal J=\emptyset,\\ |\mathcal I|>\frac{k+1}{2}}}\sum_{\substack{\{\alpha_i^*,\beta_i^*\}=\{\alpha_i^*,\beta_i^*\}\ \forall i\in \mathcal I,\\ (\alpha_j^*,\beta_j^*)=(\alpha_j,\beta_j)\ \forall j\in \mathcal J}}\Psi_{\epsilon,\mathcal I}(\bv_{\mathcal I}')\Big(\prod_{i\in \mathcal I}\zeta(1-\alpha_i^*+\beta_i^*)|a_i|^{-1+\alpha_i^*+s} \Big)\\
&\times \frac{G(1-\sum_{i\in \mathcal I}(1-\alpha_i^*-s) )}{1-\sum_{i\in \mathcal I}(1-\alpha_i^*-s) } \sum_{ \ell\geq1}\frac{\prod_{i\in \mathcal I}(a_i,\ell)^{1-\alpha_i^*+\beta_i^*}}{\ell^{\sum_{i\in\mathcal I}(1-\alpha_i^*+\beta_i^*)}}\sumstar_{h\mod \ell}\prod_{j\in \mathcal J} D_{\alpha_j^*,\beta_j^*}(s,\tfrac {ha_j}{\ell})
}
where $\bv'_{\mathcal I}=(1-\alpha_i-s)_{i\in \mathcal I}$ (notice that the inequality~\eqref{cds} ensures the convergence of the series over $\ell$). 

We notice that we can assume that
\es{\label{dsdscd}
1-\tfrac1k+\tfrac\eps k-\tfrac1k\sum_{i=1}^k\min(\Re(\alpha_i),\Re(\beta_i))< 1-\tfrac\eps k-\max_{1\leq i\leq k}\max(\Re(\alpha_i),\Re(\beta_i))
}
so that  the strip for $s$ in~\eqref{cds} is non-empty. Indeed, if~\eqref{dsdscd} doesn't hold then we would have $$-\sum_{i=1}^k\min(\Re(\alpha_i),\Re(\beta_i))>1-\eps-k\max_{1\leq i\leq k}\max(\Re(\alpha_i),\Re(\beta_i))$$ and so  
$\frac23\eta_{\balpha,\bbeta}>1-\eps-k\max_{1\leq i\leq k}|\max(\Re(\alpha_i),\Re(\beta_i))|$ which implies
\est{
(2k-1)\sum_{i=1}^k(|\Re(\alpha_i)|+|\Re(\beta_i)|)&>(2k-1)(1-2\eps)-
(2k^2-k)\max_{1\leq i\leq k}|\max(\Re(\alpha_i),\Re(\beta_i))|\\
&\hspace{-2em}>k-1-\tfrac1{2(k-1)}-\eps(4k-2)-(k+1)\max_{1\leq i\leq k}|\max(\Re(\alpha_i),\Re(\beta_i))|,
}
since $\max_{1\leq i\leq k}|\max(\Re(\alpha_i),\Re(\beta_i))|\leq\frac1{2(k-1)}$. Thus, one also has
\est{
(2k-1)\sum_{i=1}^k(|\Re(\alpha_i)|+|\Re(\beta_i)|)+(k-1)\sum_{i=1}^k|\max(\Re(\alpha_i),\Re(\beta_i))|>k-1-\tfrac1{2(k-1)}-\eps(4k-2)
}
and so
\est{
2k\eta_{\balpha,\bbeta}+(k+1)\sum_{i=1}^k\min(\Re(\alpha_i),\Re(\beta_i))>k-1-\tfrac1{2(k-1)}-\eps(4k-2).
}
In particular, choosing suitably the values of $\eps$ in the two statements, Theorem~\ref{mtmt} is implied by Lemma~\ref{absc} in this case.

Now, by the definition of $A_{\balpha,\bbeta}(s)$, if $\sigma\leq 1-\max_{1\leq i\leq k}\max(\Re(\alpha_i),\Re(\beta_i))$ we have
\est{
\sumdagger_{N_1,\dots,N_k} A_{\balpha,\bbeta}(s)=\mathcal A_{\ba;\balpha,\bbeta}(s),
}
whence we have $\mathcal A_{\ba;\balpha,\bbeta}(s)=\mathcal W_{\balpha,\bbeta}(s)+E^*(s)$ for $(s,\balpha,\bbeta)$ satisfying~\eqref{cds}.
Then, we observe that, since $G(0)=1$, one has that $\mathcal W_{\balpha,\bbeta}(s)-\mathcal M_{\ba,\balpha,\bbeta}(s)$ is holomorphic for $\frac12\leq \sigma\leq 1-\max_{1\leq i\leq k}\max(\Re(\alpha_i),\Re(\beta_i))-\frac \eps k$ and $(\balpha,\bbeta)\in\Omega_k$ (with $\mathcal M_{\ba;\balpha,\bbeta}(s)$ defined in~\eqref{dfm}), and one easily checks that in this region $\mathcal W_{\balpha,\bbeta}(s)-\mathcal M_{\ba,\balpha,\bbeta}(s)$ satisfies the same bound as in~\eqref{bffe}. Thus,  taking $$E(s):=E^*(s)+\mathcal W_{\balpha,\bbeta}(s)-\mathcal M_{\ba,\balpha,\bbeta}(s)$$ we obtain Theorem~\ref{mtmt} for 
$$1-\tfrac{2(1-\eta_{\balpha,\bbeta}-3\eps)}{k+1}\leq \sigma\leq 1-\max_{1\leq i\leq k}\max(\Re(\alpha_i),\Re(\beta_i))-\tfrac \eps k.$$ Finally, by~\eqref{dsdscd} one has that for $\sigma\geq 1-\max_{1\leq i\leq k}\max(\Re(\alpha_i),\Re(\beta_i))-\frac \eps k$ Theorem~\ref{mtmt} is a consequence of Lemma~\ref{absc} and so the proof of Theorem~\ref{mtmt} is complete.

\section*{Acknowledgment}
This paper was completed while the author was visiting the Centre de Recherches Math\'ematiques in Montr\'eal. He wishes to thank this institution for providing an excellent research environment.

\addresses

\end{document}